\documentclass[11pt]{article}
\usepackage{epsfig,mst-stylefile,amssymb,amsmath,harvard}
\usepackage[svgnames]{xcolor}

\newcommand{\bbZ}{{\Bbb Z}}
\newcommand{\bbR}{{\Bbb R}}
\newcommand{\bbN}{{\Bbb N}}
\newcommand{\bbC}{{\Bbb C}}

\newcommand{\bbE}{{\Bbb E}}

\newcommand{\vecoper}{\textnormal{vec}}

\newcommand{\norm}[1]{\left\|#1\right\|}

\renewcommand{\cite}{\citeyear}

\begin{document}

\title{Wavelet eigenvalue regression for $n$-variate operator fractional Brownian motion
\thanks{The first author was partially supported by grant ANR-16-CE33-0020 MultiFracs. The second author was partially supported by the prime award no.\ W911NF-14-1-0475 from the Biomathematics subdivision of the Army Research Office, USA. The second author's long term visits to ENS de Lyon were supported by the school.
The authors would like to thank Mark M.\ Meerschaert for his comments on this work. The second author would also like to thank Tewodros Amdeberhan for the enlightening mathematical discussions.}
}

\author{Patrice Abry \\ Univ Lyon, Ens de Lyon, Univ Claude Bernard, \\
CNRS, Laboratoire de Physique, F-69342 Lyon, France
\and   Gustavo Didier \\ Mathematics Department\\ Tulane University}

\bibliographystyle{agsm}

\maketitle

\begin{abstract}
In this contribution, we extend the methodology proposed in Abry and Didier \cite{abry:didier:2017} to obtain the first joint estimator of the real parts of the Hurst eigenvalues of $n$-variate OFBM. The procedure consists of a wavelet regression on the log-eigenvalues of the sample wavelet spectrum. The estimator is shown to be consistent for any time reversible OFBM and, under stronger assumptions, also asymptotically normal starting from either continuous or discrete time measurements. Simulation studies establish the finite sample effectiveness of the methodology and illustrate its benefits compared to univariate-like (entrywise) analysis. As an application, we revisit the well-known self-similar character of Internet traffic by applying the proposed methodology to 4-variate time series of modern, high quality Internet traffic data. The analysis reveals the presence of a rich multivariate self-similarity structure.
\end{abstract}

\section{Introduction}
An operator fractional Brownian motion (OFBM) $B_H = \{B_H(t)\}_{t \in \bbR}$ is a $\bbR^n$-valued Gaussian stochastic process with stationary increments that satisfies the operator self-similarity relation
\begin{equation}\label{e:o.s.s.}
\{B_H(ct)\}_{t \in \bbR} \stackrel{{\mathcal L}}= \{c^{H}B_H(t)\}_{t \in \bbR}, \quad c > 0,
\end{equation}
where $\stackrel{{\mathcal L}}=$ stands for the equality of finite-dimensional distributions. In relation \eqref{e:o.s.s.}, which generalizes the univariate concept of self-similarity, $H$ is a $n \times n$ matrix called the Hurst matrix, and $c^H  := \exp(\log c H )$, where $\exp A := \sum^{\infty}_{k=0}\frac{A^k}{k!}$ is the usual matrix exponential. If the Jordan form
\begin{equation}\label{e:H=PJP^(-1)}
H = P J_H P^{-1}
\end{equation}
is diagonalizable with real (Hurst) eigenvalues for a nonsingular $P$, then the eigenvectors form a coordinate system in which the $q$-th marginal $\{B_H(t)_q\}_{t \in \bbR}$ of $B_H$, $q = 1,\hdots,n$, is a fractional Brownian motion (FBM) with Hurst scaling index $h_q$ (namely, a Gaussian, self-similar, stationary increment stochastic process -- see Embrechts and Maejima \cite{embrechts:maejima:2002}, Taqqu \cite{taqqu:2003}). These coordinate processes need not be independent. It is generally assumed that OFBM is stochastically continuous, i.e., $B_H(t)\stackrel{P}\rightarrow B_H(t_0)$ whenever $t\stackrel{P}\rightarrow t_0$, and proper, namely, its variance matrix $\bbE B_H(t)B_H(t)^*$ has full rank for $t \neq 0$.

OFBM is a multivariate fractional process. Univariate fractional processes have been used with great success in the modeling of data sets from many fields of science, technology and engineering (e.g., Mandelbrot \cite{Mandelbrot1974}, Taqqu et al.\ \cite{taqqu97}, Ivanov et al.\ \cite{ivanov1999}, Ciuciu et al.\ \cite{ciuciu:abry:he:2014}, Foufoula-Georgiou and Kumar \cite{Foufoula94}). The literature on the probability theory and statistical methodology for univariate fractional processes is now voluminous (e.g., Mandelbrot and Van Ness \cite{mandelbrot:vanness:1968}, Taqqu \cite{taqqu:1975,taqqu:1979}, Dobrushin and Major \cite{dobrushin:major:1979}, Granger and Joyeux \cite{granger:joyeux:1980}, Hosking \cite{hosking:1981}, Fox and Taqqu \cite{fox:taqqu:1986}, Dahlhaus \cite{dahlhaus:1989}, Beran \cite{beran:1994}, Robinson \cite{robinson:1995-gaussian,robinson:1995-logperiodogram_regression}, Abry et al.\ \cite{aftv00}, Stoev et al.\ \cite{stoev:pipiras:taqqu:2002}, Moulines et al.\ \cite{moulines:roueff:taqqu:2007:Fractals,moulines:roueff:taqqu:2007:JTSA,moulines:roueff:taqqu:2008}, Beran et al.\ \cite{beran:feng:ghosh:kulik:2013}, Bardet and Tudor \cite{bardet:tudor:2014}, Clausel et al.\ \cite{clausel:roueff:taqqu:tudor:2014:waveletestimation}, Pipiras and Taqqu \cite{pipiras:taqqu:2017}, to cite a few).

In modern applications, however, data sets are often multivariate, since several natural and artificial systems are monitored by a large
number of sensors. Accordingly, the literature on multivariate fractional processes has been expanding at a fast pace. The contributions include Hosoya \cite{hosoya:1996,hosoya:1997}, Lobato \cite{lobato:1997}, Marinucci and Robinson \cite{marinucci:robinson:2000}, Becker-Kern and Pap \cite{becker-kern:pap:2008}, Robinson \cite{robinson:2008}, Hualde and Robinson \cite{hualde:robinson:2011}, Sela and Hurvich \cite{sela:hurvich:2012}, Kristoufek \cite{kristoufek2013mixed,kristoufek2015can} and Kechagias and Pipiras \cite{kechagias:pipiras:2015,kechagias:pipiras:2015:ident}, in the time and Fourier domains, and Wendt et al.\ \cite{WENDT:2009:C}, Amblard et al.\ \cite{amblard:coeurjolly:lavancier:philippe:2012}, Coeurjolly et al.\ \cite{coeurjolly:amblard:achard:2013}, Achard and Gannaz \cite{achard:gannaz:2016}, Frecon et al.\ \cite{frecon:didier:pustelnik:abry:2016}, in the wavelet domain (see also Marinucci and Robinson \cite{marinucci:robinson:2001}, Robinson and Yajima \cite{robinson:yajima:2002}, Hualde and Robinson \cite{hualde:robinson:2010}, Nielsen and Frederiksen \cite{nielsen:frederiksen:2011}, Shimotsu \cite{shimotsu:2012} on the related fractional cointegration literature in econometrics).

The framework of operator self-similar (o.s.s.) random processes and fields was originally conceived by Laha and Rohatgi \cite{laha:rohatgi:1981}, Hudson and Mason \cite{hudson:mason:1982}, and has attracted much attention recently (e.g., Maejima and Mason \cite{maejima:mason:1994}, Mason and Xiao \cite{mason:xiao:2002}, Bierm\'{e} et al.\ \cite{bierme:meerschaert:scheffler:2007}, Xiao \cite{xiao:2009}, Guo et al.\ \cite{guo:lim:meerschaert:2009}, Didier and Pipiras \cite{didier:pipiras:2011,didier:pipiras:2012}, Clausel and Vedel \cite{clausel:vedel:2011,clausel:vedel:2013}, Li and Xiao \cite{li:xiao:2011}, Dogan et al.\ \cite{dogan:vandam:liu:meerschaert:butler:bohling:benson:hyndman:2014}, Puplinskait{\.e} and Surgailis \cite{puplinskaite:surgailis:2015}, Didier et al.\ \cite{didier:meerschaert:pipiras:2017symmetries,didier:meerschaert:pipiras:2017exponents}). If $H = \textnormal{diag}(h_1,\hdots,h_n)$ and $P = I$ in \eqref{e:o.s.s.}, then the latter relation breaks down into simultaneous entrywise expressions
\begin{equation}\label{e:o.s.s._entrywise}
\{ B_H(ct) \}_{t \in \bbR} \stackrel{\mathcal{L}}=  \{ (c^{h_1}B_H(t)_1, \hdots,c^{h_n}B_H(t)_n)^* \}_{t \in \bbR}, \quad c > 0.
\end{equation}
Relation \eqref{e:o.s.s._entrywise} is henceforth called \textit{entrywise scaling}. Several estimators have been developed by building upon the univariate-like, entrywise scaling laws, e.g., the Fourier-based multivariate local Whittle (e.g., Shimotsu \cite{shimotsu:2007}, Nielsen \cite{nielsen:2011}) and the multivariate wavelet regression (Wendt et al.\ \cite{WENDT:2009:C}, Amblard and Coeurjolly \cite{amblard:coeurjolly:2011}). However, if $H$ is non-diagonal, then the matrix $P$ mixes together the several entries of $B_H$. In this case, the univariate-like statistical analysis of each entry of $Y$ will often generate estimates that are undetermined convex combinations of Hurst eigenvalues or, at large scales, estimates of the \textit{largest} Hurst eigenvalue (see, for instance, Chan and Tsai \cite{chan:tsai:2010}, Didier et al.\ \cite{didier:helgason:abry:2015}, Tsai et al.\ \cite{tsai:rachinger:chan:2017}, Abry et al.\ \cite{abry:didier:li:2017}).

In Abry and Didier \cite{abry:didier:2017}, the use of the eigenstructure of wavelet variance matrices is proposed for the estimation of the Hurst parameters of OFBM. The main results are obtained in the bivariate context, in which it is shown that wavelet log-eigenvalues -- and also wavelet eigenvectors, under assumptions -- are consistent and asymptotically normal estimators of the eigenstructure of the Hurst matrix $H$.

In this paper, we extend this approach by proposing a wavelet eigenvalue regression estimator of the Hurst eigenvalues of $n$-variate OFBM. The estimator is shown to be consistent for the real parts of the eigenvalues of $H$ for (essentially) any time reversible OFBM. Under the stronger assumption that Hurst eigenvalues are real and simple (pairwise distinct), we further show that the wavelet eigenvalue regression estimator is asymptotically normal. Establishing the latter properties involves showing that the wavelet log-eigenvalues themselves are a consistent and asymptotically normal estimator of (the real parts of) the eigenvalues of the Hurst matrix. Under the additional assumption that the matrix of Hurst eigenvectors $P$ (mixing matrix) in \eqref{e:H=PJP^(-1)} is orthogonal, a consistent sequence of wavelet eigenvectors is also shown to exist. With a view toward hypothesis testing, we also investigate conditions for asymptotic normality when all Hurst eigenvalues are equal. The mathematical framework is much more general than that in Abry and Didier \cite{abry:didier:2017}, which builds upon closed form expressions for eigenvalues and eigenvectors in dimension 2.


In the context of scaling properties, the use of eigenanalysis was first proposed in Meerschaert and Scheffler \cite{meerschaert:scheffler:1999,meerschaert:scheffler:2003} for operator stable laws, and later in Becker-Kern and Pap \cite{becker-kern:pap:2008} for o.s.s.\ processes in the time domain. It has also been used in the cointegration literature (e.g., Phillips and Ouliaris \cite{phillips:ouliaris:1988}, Harris and Poskitt \cite{harris:poskitt:2004}, Li et al.\ \cite{li:pan:yao:2009}, Zhang et al.\ \cite{zhang:robinson:yao:2016}). The wavelet framework has the benefit of computational efficiency (Daubechies \cite{daubechies:1992}, Mallat \cite{mallat:1999}), which is especially important in a multivariate setting (see Abry et al.\ \cite{abry:didier:li:2017}, Section 5.2, for a computational comparison between maximum likelihood and a wavelet-based estimation methodology). In addition, for a large enough number of vanishing moments $N_{\psi}$, wavelet coefficients $\{D(2^j,k) \}_{k \in \bbZ} \in \bbR^n$ are stationary in the shift parameter $k$ at every octave $j$, and the sample wavelet variance matrix is asymptotically normal at a fixed octave $j$. These properties are in part a consequence of the quasi-decorrelation property of the wavelet transform (Flandrin \cite{flandrin:1992}, Wornell and Oppenheim \cite{wornell:oppenheim:1992}, Masry \cite{masry:1993}, Bardet and Tudor \cite{bardet:tudor:2010}, Clausel et al.\ \cite{clausel:roueff:taqqu:tudor:2014:quadraticvariation}). To the best of our knowledge, we are proposing the first eigenanalysis-based asymptotically normal estimator of Hurst eigenvalues in general dimension $n$, under assumptions. The most general case of multiple blocks of Hurst eigenvalues with algebraic multiplicity greater than 1 (see Section \ref{s:preliminaries} on terminology) calls for special efforts and remains a topic for future research, since asymptotic distributions may be normal or nonnormal (see Remark \ref{r:difficulties} on the difficulties involved).

We conducted broad Monte Carlo experiments which illustrate the appropriate use of the estimator and demonstrate its finite sample size effectiveness. In addition, we apply the proposed methodology in the modeling of 4-variate Internet traffic time series from the so-named MAWI archive. The latter comprises Internet traffic traces captured on a high-speed, high-capacity backbone that mostly connects academic institutions in Japan and the USA. Our study reveals, for the first time, the presence of multivariate scaling properties in Internet traffic.

This paper is organized as follows. Section \ref{s:preliminaries} contains the notation, theoretical background, assumptions and definitions. The main mathematical results can be found in Section \ref{s:asymptotic_theory}, namely, the consistency and asymptotic normality of wavelet log-eigenvalues for Hurst eigenvalues, as well as the corresponding results for the wavelet eigenvalue regression estimator. In Section \ref{s:asymptotic_theory_discrete}, we extend the results from Section \ref{s:asymptotic_theory} to the more realistic context where measurements are made in discrete time. Section \ref{s:mc} contains Monte Carlo studies. Section \ref{s:application} contains the wavelet eigenvalue analysis of Internet traffic data. All proofs can be found in the Appendix, together with auxiliary results.

\section{Preliminaries}\label{s:preliminaries}

\subsection{Notation and background}

Hereinafter, ${\mathcal H}(n,\bbR)$, ${\mathcal H}_{\geq 0}(n,\bbR)$, ${\mathcal H}_{>0}(n,\bbR)$, ${\mathcal H}_{\geq 0}(n,\bbC)$, ${\mathcal H}_{>0}(n,\bbC)$ and $M(n,\bbR)$ denote, respectively, the space of symmetric matrices and the cones of symmetric positive semidefinite, symmetric positive definite, Hermitian positive semidefinite and Hermitian positive definite matrices, and the space of $n \times n$ matrices. The real and complex spheres are represented by $S^{n-1}$ and $S^{n-1}_{\bbC}$, respectively. For $h_\cdot \in \bbC$, $J_{h_\cdot} \in M(n_{h_\cdot} ,\bbC)$ denotes a Jordan block of size $n_{h_\cdot}$ (see \eqref{e:Jordan_block} for an explicit expression). For a matrix $M \in M(n,\bbR)$, recall that the multiplicity of an eigenvalue $\lambda$ is its multiplicity as a zero of the characteristic polynomial of $M$. An eigenvalue $\lambda$ is called simple when its algebraic multiplicity is 1 (Horn and Johnson \cite{horn:johnson:2012}, p.\ 76). The operator $\textnormal{vec}_{{\mathcal S}}(\cdot)$ vectorizes the upper triangular entries of a symmetric matrix $S$.

Let $B_H = \{B_H(t) \}_{t \in \bbR}$ be an OFBM with Hurst matrix $H$. Following the results in Didier and Pipiras \cite{didier:pipiras:2011}, if the eigenvalues of $H$ satisfy
\begin{equation}\label{e:eigen-assumption}
0< \Re(h_q)<1,\quad q=1,\hdots,n,
\end{equation}
then the OFBM admits the harmonizable representation
\begin{equation}\label{e:OFBM_harmonizable}
\{B_H(t) \}_{t \in \bbR} \stackrel{{\mathcal L}}= \Big\{ \int_{\bbR} \Big( \frac{e^{{\mathbf i}tx} - 1}{{\mathbf i} x}\Big)
\{ x^{-D}_{+}A + x^{-D}_{-}\overline{A} \} \widetilde{B}(dx) \Big\}_{t \in \bbR},
\end{equation}
where $\stackrel{{\mathcal L}}=$ denotes the equality of finite dimensional distributions, $D := H - (1/2)I$, $A \in M(n,\bbC)$ and $\widetilde{B}(dx)$ is a $\bbC$-valued, Gaussian random measure satisfying the constraints $\widetilde{B}(-dx) = \overline{\widetilde{B}(dx)}$, $\bbE \widetilde{B}(dx)\widetilde{B}(dx)^* = dx$. Conversely, if
\begin{equation}\label{e:full-rank}
\textnormal{$\Re(AA^*)$ has full rank},
\end{equation}
then the process defined by the expression on the right-hand side of \eqref{e:OFBM_harmonizable} is proper; hence, it defines an OFBM $B_H$. If
\begin{equation}\label{e:time_revers}
\Im(AA^*) = 0,
\end{equation}
then the OFBM $B_H$ is time reversible, i.e., $\{B_H(t) \}_{t \in \bbR} \stackrel{{\mathcal L}}= \{B_H(-t) \}_{t \in \bbR}$.

\subsection{Assumptions and definitions}

Throughout the paper, we assume that the underlying stochastic process is a $\bbR^n$-valued OFBM $B_H = \{B_{H}(t)\}_{t \in \bbR}$ under the following conditions.

\medskip

\noindent {\sc Assumption (OFBM1)}: condition \eqref{e:full-rank} holds.

\medskip

\noindent {\sc Assumption (OFBM2)}: condition \eqref{e:time_revers} holds.

\medskip

\noindent In Sections \ref{s:asymptotic_theory} and \ref{s:asymptotic_theory_discrete}, we will make use of assumptions (OFBM 1--2) combined with one of the following two assumptions.

The first one, called (OFBM3), is the more general and will be applied in consistency statements. In fact, it simply recasts \eqref{e:eigen-assumption} based on Jordan blocks.

\medskip

\noindent {\sc Assumption (OFBM3)}:
$$
H = P J_H P^{-1}, \quad J_H = \textnormal{diag}(J_{h_{1}},J_{h_{2}},\hdots,J_{h_{n'}}), \quad 1 \leq n' \leq n,
$$
where each $J_{h_{\cdot}}$ is a Jordan block of length $n_{\cdot}$,
$$
\quad 0 < \Re h_1 \leq \Re h_2 \leq \hdots \leq \Re h_{n'} < 1,
$$
\begin{equation}\label{e:H_Jordan}
\quad P \in GL(n,\bbC), \quad \|p_{\cdot,q}\| = 1, \quad q = 1,\hdots,n,
\end{equation}
and $p_{\cdot,q}$ denotes a column vector of $P$.

\medskip

The second one, named (OFBM3$'$), is more stringent and will be used in (most) asymptotic normality statements (n.b.: the latter should not to be confused with Theorem \ref{t:asymptotic_normality_wavecoef_fixed_scales}, which holds under great generality for fixed scales).

\medskip

\noindent {\sc Assumption (OFBM3$'$)}:
\begin{equation}\label{e:H_h1<...<hn}
H = PJ_HP^{-1}, \quad J_H = \textnormal{diag}(h_1,\hdots,h_n), \quad 0 < h_1 < \hdots < h_n < 1, \quad P \in GL(n,\bbR).
\end{equation}

\medskip

In particular, condition \eqref{e:H_h1<...<hn} implies that every eigenvalue of the Hurst matrix $H$ is real and simple.

Throughout the paper, we will make the following assumptions on the underlying wavelet basis. For this reason, such assumptions will be omitted in statements.

\medskip

\noindent {\sc Assumption $(W1)$}: $\psi \in L^1(\bbR)$ is a wavelet function, namely,
\begin{equation}\label{e:N_psi}
\int_{\bbR} \psi^2(t)dt = 1 , \quad \int_{\bbR} t^{p}\psi(t)dt = 0, \quad p = 0,1,\hdots, N_{\psi}-1, \quad N_{\psi} \geq 2.
\end{equation}
\noindent {\sc Assumption $(W2)$}:
\begin{equation}\label{e:supp_psi=compact}
\textnormal{$\textnormal{supp}(\psi)$ is a compact interval}.
\end{equation}
\noindent {\sc Assumption $(W3)$}: there is $\alpha > 1$ such that
\begin{equation}\label{e:psihat_is_slower_than_a_power_function}
\sup_{x \in \bbR} |\widehat{\psi}(x)| (1 + |x|)^{\alpha} < \infty.
\end{equation}

\medskip

Under \eqref{e:N_psi}, \eqref{e:supp_psi=compact} and \eqref{e:psihat_is_slower_than_a_power_function}, $\psi$ is continuous, $\widehat{\psi}(x)$ is everywhere differentiable and its first $N_{\psi}-1$ derivatives are zero at $x = 0$ (see Mallat \cite{mallat:1999}, Theorem 6.1 and the proof of Theorem 7.4).

\begin{example}
If $\psi$ is a Daubechies wavelet with $N_{\psi}$ vanishing moments, $\textnormal{supp}(\psi) = [0,2N_{\psi} -1]$ (see Mallat \cite{mallat:1999}, Proposition 7.4).
\end{example}

Next, we define the wavelet transform and sample wavelet variance (spectrum) of OFBM.
\begin{definition}
Let $B_H = \{B_H(t)\}_{t \in \bbR}$ be an OFBM satisfying the assumptions (OFBM 1--3). Its (normalized) wavelet transform at octave $j$ and shift $k$ is given by
\begin{equation}\label{e:D(2^j,k)}
\bbR^n \ni D(2^j,k) = 2^{-j/2} \int_{\bbR} 2^{-j/2} \psi(2^{-j}t - k)B_H(t) dt, \quad j \in \bbN \cup \{0\}, \quad k \in \bbZ.
\end{equation}
For a (wavelet) sample size $\nu$, the sample wavelet variance is defined by the random matrix
\begin{equation}\label{e:W(2j)}
W(2^j) = \frac{1}{K_j}\sum^{K_j}_{k=1}D(2^j,k)D(2^j,k)^*, \quad K_j = \frac{\nu}{2^j}.
\end{equation}
\end{definition}
The following theorem shows that $\{\textnormal{vec}_{{\mathcal S}}W(2^j)\}_{j = j_1,\hdots,j_2}$ is asymptotically normal (see Section \ref{s:preliminaries} on the definition of the operator $\textnormal{vec}_{{\mathcal S}}$).

\begin{theorem}\label{t:asymptotic_normality_wavecoef_fixed_scales}
(Abry and Didier \cite{abry:didier:2017}, Theorem 3.1) Let $B_H = \{B_{H}\}_{t \in \bbR}$ be an OFBM under the assumptions (OFBM1--3) and consider
\begin{equation}\label{e:j1<...<j2_m=j2-j1+1}
j = j_1 , j_1 + 1, \hdots , j_2, \quad m := j_2 - j_1 + 1.
\end{equation}
Let $F \in {\mathcal S}(\frac{n(n+1)}{2}m,\bbR)$ be the asymptotic covariance matrix described in Proposition 3.3 in Abry and Didier \cite{abry:didier:2017}. Then,
\begin{equation}\label{e:asymptotic_normality_wavecoef_fixed_scales}
\Big(\sqrt{K_j} \hspace{1mm}\textnormal{vec}_{{\mathcal S}} (W(2^j)- {\Bbb E}W(2^j) ) \Big)_{j = j_1 , \hdots, j_2} \stackrel{d}\rightarrow {\mathcal N}_{\frac{n(n+1)}{2} \times m}(0,F),
\end{equation}
as $\nu \rightarrow \infty$.
\end{theorem}

When $H = P \textnormal{diag}(h_1,\hdots,h_n) P^{-1}$ with real eigenvalues and a scalar matrix $P$ --  i.e., it has the form $P = p  I $ for some constant $p \neq 0$ --, the sample wavelet variance satisfies the so-named entrywise scaling relation
$$
W(2^j) \stackrel{d}= \Big\{2^{j (h_i + h_{i'})}W(1)_{ii'}\Big\}_{i,i'= 1,\hdots,n}, \quad j \in \bbN
$$
(c.f.\ Introduction). In this case, the (Hurst) eigenvalues can be estimated by means of an entrywise log-regression procedure (Amblard and Coeurjolly \cite{amblard:coeurjolly:2011}, Coeurjolly et al.\ \cite{coeurjolly:amblard:achard:2013}). However, for a general matrix $P \in GL(n,\bbC)$, entrywise analysis is highly biased, since there is no simple relation between Hurst eigenvalues and the entrywise behavior of the wavelet variance matrix.

Likewise, wavelet eigenvalues do not satisfy a simple scaling relation based on Hurst eigenvalues. However, as it turns out, an approximate scaling relation appears in the coarse scale limit. So, rewrite the sample wavelet variance at scale $a(\nu)2^j$ as
\begin{equation}\label{e:Wa(a2^j)}
W_a(a(\nu)2^j) = \frac{1}{K_{a,j}}\sum^{K_{a,j}}_{k=1}D(a(\nu) 2^j,k)D(a(\nu) 2^j,k)^*, \quad K_{a,j} = \frac{\nu}{a(\nu) 2^j},
\end{equation}
The dyadic, slow-growth scaling factor $a(\nu)$ in \eqref{e:Wa(a2^j)} satisfies the relation
\begin{equation}\label{e:a(nu)/J->infty}
a(\nu) \leq \frac{\nu}{2^j},\quad \frac{a(\nu)}{\nu}+\frac{\nu}{a(\nu)^{1+2 \varpi_0}}\rightarrow 0, \quad \nu \rightarrow \infty,
\end{equation}
where $\varpi_0$ is the regularity parameter
$$
\varpi_0 = \min\{\hspace{0.5mm}\Re h_1,\min_{1 \leq q_1 < q_2 \leq n}(\Re h_{q_2}-\Re h_{q_1})\}1_{\{\min_{1 \leq q_1 < q_2 \leq n}(\Re h_{q_2}-\Re h_{q_1}) >0\}}
$$
$$
+ \Re h_1 1_{\{\min_{1 \leq q_1 < q_2 \leq n}(\Re h_{q_2}-\Re h_{q_1}) = 0\}}.
$$
Then, by the operator self-similarity property (see Abry and Didier \cite{abry:didier:2017}, Proposition 3.1),
\begin{equation}\label{e:EW(a(nu)2j)_in_the_case_blindsourcing}
W_a(a(\nu)2^j) \stackrel{d}= P a(\nu)^{J_H} \widehat{B}_a(2^j) a(\nu)^{J^*_H} P^*, \quad {\Bbb E}W_a(a(\nu)2^j) = P a(\nu)^{J_H} B(2^j) a(\nu)^{J^*_H} P^*,
\end{equation}
where
\begin{equation}\label{e:Bhat_a(2^j)_B(2^j)}
\widehat{B}_a(2^j) := P^{-1} W_a(2^j)(P^*)^{-1}, \quad B(2^j) := P^{-1} \bbE W(2^j)(P^*)^{-1}.
\end{equation}
In particular, if $H$ is diagonalizable, the latter matrices satisfy entrywise scaling relations
$$
\widehat{B}_a(2^j) = \Big( \widehat{b}(2^j)_{ii'} \Big)_{i,i'=1,\hdots,n} \stackrel{d}= \Big( 2^{j(h_i + h_{i'})}\widehat{b}(1)_{ii'} \Big)_{i,i'=1,\hdots,n},
$$
\begin{equation}\label{e:B(2^j)_entrywise_scaling}
B(2^j) = \Big( b(2^j)_{ii'} \Big)_{i,i'=1,\hdots,n} = \Big( 2^{j(h_i + h_{i'})}b(1)_{ii'} \Big)_{i,i'=1,\hdots,n}.
\end{equation}

We are now in a position to define the wavelet eigenstructure estimator of the real parts of the Hurst eigenvalues \eqref{e:eigen-assumption} by means of a weighted regression procedure on wavelet log-eigenvalues.
\begin{definition}\label{def:eigenvalue_estimator}
Let $B_H = \{B_H(t)\}_{t \in \bbR}$ be an OFBM satisfying the assumptions (OFBM 1--3), and let $\{W_a(a(\nu)2^j)\}_{j=j_1,\hdots,j_2}$ be its sample wavelet variance matrices corresponding to scales $\{a(\nu)2^{j_1},\hdots, a(\nu)2^{j_2}\}$. The wavelet eigenstructure estimator of the Hurst eigenvalues is given by the regression system
\begin{equation}\label{e:hl-hat}
\{ \widehat{\Re  h}_q \}_{q = 1,\hdots,n} =  \Big\{\frac{1}{2}\sum_{j=j_1}^{j_2} w_j \log_2 \lambda_{q}(W_a(a(\nu)2^j)) \Big\}_{q = 1,\hdots,n}.
\end{equation}
In \eqref{e:hl-hat}, $w_j$, $j = j_1,\hdots,j_2$, are weights satisfying the relations
\begin{equation}\label{e:sum_wj=0,sum_jwj=1}
\sum^{j_2}_{j=j_1}w_j = 0, \quad \sum^{j_2}_{j=j_1}j w_j = 1.
\end{equation}
\end{definition}
Since $W_a(a(\nu)2^j) \in {\mathcal H}_{\geq 0}(n,\bbC)$ a.s., then expression \eqref{e:hl-hat} is well-defined a.s. If, in addition, $\Re h_q = h_q$ for some $q = 1,\hdots,n$, we will simply write $\widehat{h}_{q}$ instead of $\widehat{\Re  h}_q$.

\section{Asymptotic theory: continuous time}\label{s:asymptotic_theory}

In this section, assuming measurements in continuous time, we establish the asymptotic properties of wavelet log-eigenvalues, as well as the corresponding results for the wavelet eigenvalue regression estimator described in Definition \ref{def:eigenvalue_estimator}.

In Theorem \ref{t:consistency}, ordered wavelet log-eigenvalues are shown to be consistent for their respective (real parts of) Hurst eigenvalues for any time reversible OFBM. Consistency appears as a consequence of the operator self-similarity property \eqref{e:EW(a(nu)2j)_in_the_case_blindsourcing} of wavelet variance matrices by applying the Courant-Fischer principle (see \eqref{e:Courant-Fischer}).
\begin{theorem}\label{t:consistency}
Let $B_{H} = \{B_H(t)\}_{t \in \bbR}$ be an OFBM under the assumptions (OFBM 1--2). Fix $j \in \bbN$. If, in addition, $B_{H}$ satisfies (OFBM3), then
\begin{equation}\label{e:log_lambdaE/2_log_a(n)}
\frac{\log\lambda_{q}(W_a(a(\nu)2^j))}{2 \log a(\nu)} \stackrel{P}\rightarrow \Re h_{q'}, \quad \frac{\log\lambda_{q}(\bbE W_a(a(\nu)2^j))}{2 \log a(\nu)} \rightarrow \Re h_{q'}, \quad q = 1,\hdots,n,
\end{equation}
as $\nu \rightarrow \infty$, where $q' \in \{1,\hdots,n'\}$ is such that
\begin{equation}\label{e:n1+...+nq'1<q=<n1+...+nq'}
n_1 + n_2 + \hdots + n_{q'-1} < q \leq n_1 + n_2 + \hdots + n_{q'}.
\end{equation}In particular, if $ h_1 < \hdots < h_n$, then
\begin{equation}\label{e:log_lambdaE/2_log_a(n)_distinct Re}
\frac{\log\lambda_{q}(W_a(a(\nu)2^j))}{2 \log a(\nu)} \stackrel{P}\rightarrow h_{q}, \quad \frac{\log\lambda_{q}(\bbE W_a(a(\nu)2^j))}{2 \log a(\nu)} \rightarrow h_{q}, \quad q = 1,\hdots,n.
\end{equation}
\end{theorem}

Theorem \ref{t:asympt_normality_lambda2}, which requires the stronger assumption (OFBM3$'$), establishes the asymptotic normality of wavelet log-eigenvalues. Proving it requires establishing Proposition \ref{p:convergence} first, which contains some properties of interest of wavelet variance matrices. For sample wavelet variance matrices, these properties can be summed up as follows. First, the ratio between the $q$-th wavelet eigenvalue and the power law $a(\nu)^{2h_q}$ converges to a limiting function $\xi_q(\cdot)$ that satisfies a scaling relation. Second, for each eigenvalue $\lambda_q(W_a(a(\nu)2^j))$, $q = 1,\hdots,n$, there is a convergent sequence of associated eigenvectors $\{u_q(\nu)\}_{\nu \in \bbN} \subseteq S^{n-1}$. Therefore, we can assume that the eigenvectors $u_{q}(\nu)$ converge (in probability) in the space $\textnormal{span}\{p_{\cdot,q},\hdots,p_{\cdot,n}\}$ (see \eqref{e:H_Jordan} on the definition of the vectors $p_{\cdot,q},\hdots,p_{\cdot,n}$). In particular, $u_{n}(\nu) \stackrel{P}\rightarrow p_{\cdot,n}$.
\begin{proposition}\label{p:convergence}
Let $B_{H} = \{B_H(t)\}_{t \in \bbR}$ be an OFBM under the assumptions (OFBM 1,2,3$'$). Let $W_a(a(\nu)2^j)$ be the sample wavelet variance matrix \eqref{e:Wa(a2^j)}. Fix $q \in \{1,\hdots,n\}$ and an octave $j$. Then, as $\nu \rightarrow \infty$,
\begin{itemize}
\item [$(i)$] there is a function $\xi_{q} >0$ such that
\begin{equation}\label{e:lambdai/a^(2hi)_conv}
\frac{\lambda_{q}(W_a(a(\nu)2^j))}{a(\nu)^{2h_{q}}}\stackrel{P}\rightarrow \xi_{q}(2^j);
\end{equation}
\item [$(ii)$] the wavelet eigenvalue limiting function in \eqref{e:lambdai/a^(2hi)_conv} satisfies the scaling relation
\begin{equation}\label{e:xi-i0_scales}
\xi_{q}(2^j)  = 2^{j \hspace{0.5mm}2 h_{q}} \xi_{q}(1).
\end{equation}
In particular, $\xi_{n}(2^j)  = b_{nn}(2^{j})= 2^{j \hspace{0.5mm}2 h_{n}} b_{nn}(1)$ (see \eqref{e:B(2^j)_entrywise_scaling});
\item [$(iii)$] for some sequence $\{u_{q}(\nu)\}_{\nu \in \bbN}$ of unit eigenvectors associated with the $q$-th eigenvalue of $W_a(a(\nu)2^j)$, there is a unit vector $u_{q}$ such that
\begin{equation}\label{e:u1,u2,u3_conv}
u_{q}(\nu) \stackrel{P}\rightarrow u_{q},
\end{equation}
where
\begin{equation}\label{e:angles_uq_in_the_limit}
u_{q} \in \left\{\begin{array}{cc}
\{p_{\cdot,q + 1},\hdots,p_{\cdot,n}\}^{\perp}, & 1 \leq q \leq n - 1,\\
\textnormal{span}\{p_{\cdot,q}\}, & q = n.
\end{array}\right.
\end{equation}
In particular,
$$
u_{q} \in \textnormal{span}\{p_{\cdot,q},p_{\cdot,q + 1},\hdots,p_{\cdot,n}\}, \quad 1 \leq q \leq n;
$$
\item [$(iv)$] the sequence $\{u_{q}(\nu)\}_{\nu \in \bbN}$ satisfies
\begin{equation}\label{e:inner*scaling=o(1)}
\Big(\langle p_{\cdot,q+1},u_{q}(\nu) \rangle a^{h_{q+1} - h_{q}}, \hdots, \langle p_{\cdot,n},u_{q}(\nu) \rangle a^{h_n - h_{q}} \Big) \stackrel{P}\rightarrow {\mathbf x}_{q,*}(2^j) \in \bbR^{n-q},
\end{equation}
for some limiting vector function ${\mathbf x}_{q,*}(2^j)$ (see \eqref{e:x*_sol_determ}).
\end{itemize}
All claims above hold with the matrix $\bbE W_a(a(\nu)2^j)$ as in \eqref{e:Wa(a2^j)} replacing $W_a(a(\nu)2^j)$, and with deterministic convergence in expressions \eqref{e:lambdai/a^(2hi)_conv}, \eqref{e:u1,u2,u3_conv} and \eqref{e:inner*scaling=o(1)}.
\end{proposition}

\begin{corollary}
Under the assumptions of Proposition \ref{p:convergence}, suppose, in addition, that $P \in O(n)$ in \eqref{e:H_Jordan}. Then, there is a consistent sequence of wavelet eigenvectors for $P$.
\end{corollary}

\begin{remark}\label{r:consistency}
Note that, under the stronger assumptions of Proposition \ref{p:convergence}, \eqref{e:lambdai/a^(2hi)_conv} is a consistency statement that implies the consistency property \eqref{e:log_lambdaE/2_log_a(n)_distinct Re} of wavelet log-eigenvalues.
\end{remark}

\begin{example}
For $q = 1,\hdots,n$ and a fixed $j$, the limiting eigenvalue scaling function $\xi_{q}(2^j)$ in \eqref{e:xi-i0_scales} depends on $2^{j2h_q}$, the matrix $B(1)$ and the angle term $\langle p_{\cdot,q},u_q\rangle$. For the sake of illustration, consider $n = 3$. For $q = 3$, by the entrywise scaling relation \eqref{e:B(2^j)_entrywise_scaling},
$$
\xi_3(2^j) = b_{33}(2^j) = 2^{j2h_3}b_{33}(1) = 2^{j2h_3}\xi_3(1).
$$
Moreover, for $q = 2$, again by the entrywise scaling relation \eqref{e:B(2^j)_entrywise_scaling} and by Lemma \ref{l:gtilde_g_min},
$$
\xi_2(2^j) = \langle p_{\cdot,2},u_2\rangle^2 \frac{b_{22}(2^j)b_{33}(2^j) - b^2_{23}(2^j)}{b_{33}(2^j)}
= \langle p_{\cdot,2},u_2\rangle^2 \frac{2^{j(2h_2 + 2h_3)}}{2^{j2h_3}}\frac{b_{22}(1)b_{33}(1) - b^2_{23}(1)}{b_{33}(1)}
$$
$$
=  2^{j2h_2} \langle p_{\cdot,2},u_2\rangle^2\frac{b_{22}(1)b_{33}(1) - b^2_{23}(1)}{b_{33}(1)} = 2^{j2h_2} \xi_2(1).
$$
Likewise, for $q = 1$,
$$
\xi_1(2^j) = \langle p_{\cdot,1},u_1\rangle^2\Big\{ b_{11}(2^j) + b_{22}(2^j)\Big(\frac{b_{33}(2^j)b_{12}(2^j)-b_{23}(2^j)b_{13}(2^j)}{b_{22}(2^j)b_{33}(2^j)- b^{2}_{23}(2^j)}\Big)^2
$$
$$
+ b_{33}(2^j) \Big( \frac{-b_{23}(2^j)b_{12}(2^j)+b_{22}(2^j)b_{13}(2^j)}{b_{22}(2^j)b_{33}(2^j)- b^{2}_{23}(2^j)} \Big)^2
$$
$$
- 2 b_{12}(2^j)\frac{b_{33}(2^j)b_{12}(2^j)-b_{23}(2^j)b_{13}(2^j)}{b_{22}(2^j)b_{33}(2^j)- b^{2}_{23}(2^j)}
- 2 b_{13}(2^j)\frac{(-b_{23}(2^j)b_{12}(2^j)+b_{22}(2^j)b_{13}(2^j))}{b_{22}(2^j)b_{33}(2^j)- b^{2}_{23}(2^j)}
$$
$$
+ 2 b_{23}(2^j)\frac{(b_{33}(2^j)b_{12}(2^j)-b_{23}(2^j)b_{13}(2^j))(-b_{23}(2^j)b_{12}(2^j)+b_{22}(2^j)b_{13}(2^j))}{(b_{22}(2^j)b_{33}(2^j)- b^{2}_{23}(2^j))^2} \Big\}
= 2^{j 2 h_1}  \xi_{1}(1).
$$
\end{example}

We are now in a position to prove the asymptotic normality of wavelet log-eigenvalues. Apart from Proposition \ref{p:convergence}, the latter is mainly a consequence of the operator self-similarity property \eqref{e:EW(a(nu)2j)_in_the_case_blindsourcing}, Theorem \ref{t:asymptotic_normality_wavecoef_fixed_scales}, and the fact that all eigenvalues of $W_a(a(\nu)2^j)$ become simple for large enough $\nu$ by virtue of condition \eqref{e:H_h1<...<hn}.
\begin{theorem}\label{t:asympt_normality_lambda2}
Let $B_{H} = \{B_H(t)\}_{t \in \bbR}$ be an OFBM under the assumptions (OFBM 1--2). Consider the range of octaves \eqref{e:j1<...<j2_m=j2-j1+1}. If, in addition, $B_{H}$ satisfies (OFBM3$'$), then
\begin{equation}\label{e:asympt_normality_lambda2}
\bbR^{m \times n} \ni \Big( \sqrt{K_{a,j}}\Big( \log \lambda_q(W_a(a(\nu)2^{j}))  - \log \lambda_q(\bbE W_a (a(\nu)2^{j})) \Big)_{q=1,\hdots,n} \Big)_{j=j_1,\hdots,j_2} \stackrel{d}\rightarrow {\mathcal N}(0,\Sigma_{\lambda})
\end{equation}
as $\nu \rightarrow \infty$. If we write the asymptotic covariance matrix in block form $\Sigma_{\lambda} = \Big( \Sigma_{\lambda}(jj') \Big)_{j,j'= j_1,\hdots,j_2}$, then its main diagonal entries satisfy $\Sigma_{\lambda}(jj)_{ii}> 0$, $i = 1,\hdots,n$.
\end{theorem}


The asymptotic properties of the wavelet eigenvalue regression estimator \eqref{e:hl-hat} are a consequence of those of wavelet log-eigenvalues, as established in Theorems \ref{t:consistency} and \ref{t:asympt_normality_lambda2}.
\begin{corollary}\label{c:h^q_asymptotically_normal}
Let $B_{H} = \{B_H(t)\}_{t \in \bbR}$ be an OFBM under the assumptions (OFBM 1--2) and consider the estimator described in Definition \ref{def:eigenvalue_estimator}.
\begin{itemize}
\item [$(i)$] If, in addition, $B_{H}$ satisfies (OFBM3), then, for $q = 1,\hdots,n$,
\begin{equation}\label{e:h^q_consistent}
\frac{\widehat{\Re h}_q}{\log_2 a(\nu)} \stackrel{P}\rightarrow \Re h_{q'}, \quad \nu \rightarrow \infty,
\end{equation}
where $q' \in \{1,\hdots,n'\}$ satisfies \eqref{e:n1+...+nq'1<q=<n1+...+nq'}. In particular, if $ h_1 < \hdots < h_n$, then
$$
\frac{\widehat{h}_q}{\log_2 a(\nu)} \stackrel{P}\rightarrow h_{q}.
$$
\item [$(ii)$] If, in addition, $B_{H}$ satisfies (OFBM3$'$), then
\begin{equation}\label{e:h^q_asymptotically_normal}
\sqrt{\frac{\nu}{a(\nu)}} \Big( \widehat{h}_q - h_q \Big)_{q=1,\hdots,n} \stackrel{d}\rightarrow {\mathcal N}(0,M \Sigma_{\lambda} M^*),
\end{equation}
as $\nu \rightarrow \infty$, for some weight matrix $M$ (see \eqref{e:weight_matrix_M}) and $\Sigma_{\lambda}$ as in Theorem \ref{t:asympt_normality_lambda2}.
\end{itemize}
\end{corollary}

As discussed before the statement of Theorem \ref{t:asympt_normality_lambda2}, assumption \eqref{e:H_h1<...<hn} of simple Hurst eigenvalues plays an important role in \eqref{e:asympt_normality_lambda2} and \eqref{e:h^q_asymptotically_normal}. Proposition \ref{p:h1=...=hn}, stated next, provides a basic framework for testing the hypothesis that there is a single Hurst eigenvalue with multiplicity $n$. To establish it, we make the following assumption.

\medskip

\noindent {\sc Assumption (OFBM3$''$)}:
\begin{equation}\label{e:h1=...=hn}
0 < h := h_1 = \hdots = h_n < 1, \quad P \in GL(n,\bbR),
\end{equation}
and
\begin{equation}\label{e:AA*_has pairwise distinct eigenvalues}
\textnormal{every eigenvalue of $AA^*$ is simple}.
\end{equation}

\medskip

\begin{proposition}\label{p:h1=...=hn}
Let $B_{H} = \{B_H(t)\}_{t \in \bbR}$ be an OFBM satisfying the assumptions (OFBM 1--2,3$''$). Then, the weak limits \eqref{e:log_lambdaE/2_log_a(n)}, \eqref{e:asympt_normality_lambda2}, \eqref{e:h^q_consistent} and \eqref{e:h^q_asymptotically_normal} hold, namely, the wavelet log-eigenvalues and the wavelet eigenvalue regression estimator \eqref{e:hl-hat} are consistent and asymptotically normal for their respective Hurst eigenvalues.
\end{proposition}

\begin{remark}\label{r:difficulties}
A convergent sequence of wavelet eigenvectors (see Proposition \ref{p:convergence}) is required in the proof of Theorem \ref{t:asympt_normality_lambda2}. However, the existence of such a sequence is in general not guaranteed. For instance, without \eqref{e:AA*_has pairwise distinct eigenvalues}, eigenvectors do not necessarily converge under \eqref{e:h1=...=hn}. Under the latter condition, the asymptotic distribution of
\begin{equation}\label{e:log_lambdaq}
\log \lambda_q(W_a(a(\nu)2^j)), \quad q = 1,\hdots,n,
\end{equation}
depends on whether or not $\bbE W(2^j)$ has simple eigenvalues. In particular, \eqref{e:log_lambdaq} may not be asymptotically normal (c.f.\ Abry et al.\ \cite{abry:didier:li:2017}, Proposition F.1). Tackling the most general case of multiple blocks of Hurst eigenvalues with algebraic multiplicity greater than 1 requires addressing all these issues.
\end{remark}

\section{Asymptotic theory: discrete time}\label{s:asymptotic_theory_discrete}

In this section, instead of a continuous time OFBM path $\{B_H(t)\}_{t \in \bbR}$,  we assume that only a discrete OFBM sample
\begin{equation}\label{e:BH(k)}
\{B_H(k)\}_{k \in \bbZ}
\end{equation}
is available. Starting from the so-called discretized wavelet coefficients (as defined in \eqref{e:approxD_2j,k} below), we develop the asymptotic properties of wavelet log-eigenvalues, as well as of the redefined wavelet eigenvalue regression estimator.

We suppose the wavelet approximation coefficients stem from Mallat's pyramidal algorithm, under a multiresolution analysis of $L^2(\bbR)$ (MRA; see Mallat \cite{mallat:1999}, chapter 7, and Stoev et al.\ \cite{stoev:pipiras:taqqu:2002}, Proposition 2.4 and Theorem 3.2). Accordingly, we need to replace ($W2$) with the following more restrictive condition.

\medskip

\noindent {\sc Assumption ($W2'$)}:
\begin{align*}
& \textnormal{the functions $\varphi$ (a bounded scaling function) and $\psi$ correspond to a MRA of $L^2(\bbR)$,} \\
& \textnormal{and $\textnormal{supp}(\varphi)$ and $\textnormal{supp}(\psi)$ are compact intervals.}
\end{align*}


\noindent Throughout this section, we assume that ($W1$), ($W2'$) and ($W3$) hold. Given \eqref{e:BH(k)}, we initialize the algorithm with the vector-valued sequence
$$
\bbR^n \ni \widetilde{a}_{0,k} := a_{\varphi}B_H(k), \quad k \in \bbZ, \quad a_{\varphi} := \int_{\bbR}\varphi(t) dt,
$$
also called the approximation coefficients at scale $2^0 = 1$. At coarser scales $2^j$, Mallat's algorithm is characterized by the iterative procedure
$$
\widetilde{a}_{j+1,k} = \sum_{k'\in \bbZ}h_{k'- 2k}\widetilde{a}_{j,k'}, \quad \widetilde{d}_{j+1,k} = \sum_{k'\in \bbZ}g_{k'- 2k}\widetilde{a}_{j,k'}, \quad j \in \bbN, \quad k \in \bbZ,
$$
where the filter sequences $\{h_k\}_{k \in \bbZ}$, $\{g_k\}_{k \in \bbZ}$ are called low- and high-pass MRA filters, respectively. Due to ($W2'$), only a finite number of filter terms is nonzero, which is convenient for computational purposes (see Daubechies \cite{daubechies:1992}, chapter 6).

\medskip

\begin{definition}\label{def:eigenvalue_estimator_discrete}
The normalized discretized wavelet coefficients are defined by
\begin{equation}\label{e:approxD_2j,k}
\bbR^{n} \ni \widetilde{D}(2^j,k) := 2^{-j/2} \widetilde{d}_{j,k}.
\end{equation}
Let $j, \log_2 a(\nu) \in \bbN$, and let $\widetilde{D}(a(\nu)2^j,k)$ be the discretized wavelet coefficient \eqref{e:approxD_2j,k} at scale $a(\nu)2^j$ and shift $k \in \{1,\hdots,K_{a,j}\}$. We define the associated sample wavelet variance by
$$
\widetilde{W}(a(\nu)2^j) = \frac{1}{K_{a,j}}\sum^{K_{a,j}}_{k=1}\widetilde{D}(a(\nu)2^j,k)\widetilde{D}(a(\nu)2^j,k)^*.
$$
Likewise, the discrete time wavelet eigenvalue regression estimator is defined by the relation
\begin{equation}\label{e:hl-tilde}
\widetilde{\Re} h_q =  \frac{1}{2}\sum_{j=j_1}^{j_2} w_j \log_2 \lambda_{q}(\widetilde{W}(a(\nu)2^j)),\quad  q = 1,\hdots,n,
\end{equation}
where the weights $w_j$, $j = j_1,\hdots,j_2$, satisfy \eqref{e:sum_wj=0,sum_jwj=1}.
\end{definition}

The following theorem contains the discrete time version of the main results in Section \ref{s:asymptotic_theory}.
\begin{theorem}\label{t:asympt_log_a(nu)_discrete}
Let $B_{H} = \{B_H(t)\}_{t \in \bbR}$ be an OFBM under the assumptions (OFBM 1--2) and the condition
\begin{equation}\label{e:eigen-assumption_stronger}
\Re(h_{q}) \in (0,1)\backslash \{1/2\},\quad q=1,\ldots,n
\end{equation}
on its Hurst eigenvalues. Consider the estimator described in Definition \ref{def:eigenvalue_estimator_discrete} and the following three different settings.
\begin{itemize}
\item [($i$)] If, in addition, $B_{H}$ satisfies (OFBM3), then, as $\nu \rightarrow \infty$,
\begin{itemize}
\item [(a)]
\begin{equation}\label{e:log_lambdaE/2_log_a(n)_discrete}
\frac{\log\lambda_{q}(\widetilde{W}(a(\nu)2^j))}{2 \log a(\nu)} \stackrel{P}\rightarrow \Re h_{q'}, \quad q = 1,\hdots,n,
\end{equation}
where $q' \in \{1,\hdots,n'\}$ satisfies \eqref{e:n1+...+nq'1<q=<n1+...+nq'}. In particular, if $ h_1 < \hdots < h_n$, then
\begin{equation}\label{e:log_lambdaE/2_log_a(n)_distinct Re_discrete}
\frac{\log\lambda_{q}(\widetilde{W}(a(\nu)2^j))}{2 \log a(\nu)} \stackrel{P}\rightarrow h_{q}, \quad q = 1,\hdots,n;
\end{equation}
\item [(b)] for $q = 1,\hdots,n$,
\begin{equation}\label{e:h^q_consistent_discrete}
\frac{\widetilde{\Re}h_q}{\log_2 a(\nu)} \stackrel{P}\rightarrow \Re h_{q'},
\end{equation}
as $\nu \rightarrow \infty$, where $q' \in \{1,\hdots,n'\}$ satisfies \eqref{e:n1+...+nq'1<q=<n1+...+nq'}.
\end{itemize}
\item [($ii$)] If, in addition, $B_{H}$ satisfies (OFBM3$'$), then, as $\nu \rightarrow \infty$,
\begin{itemize}
\item [(c)]
\begin{equation}\label{e:asympt_normality_lambda2_discrete}
\Big( \sqrt{K_{a,j}}\Big( \log \lambda_{q}(\widetilde{W}(a(\nu)2^{j}))  - \log \lambda_{q}(\bbE W_a (a(\nu)2^{j})) \Big)_{q=1,\hdots,n} \Big)_{j=j_1,\hdots,j_2} \stackrel{d}\rightarrow {\mathcal N}(0,\Sigma_{\lambda})
\end{equation}
where $\Sigma_{\lambda}$ is given in Theorem \ref{t:asympt_normality_lambda2};
\item [(d)]
\begin{equation}\label{e:h^q_asymptotically_normal_discrete}
\sqrt{\frac{\nu}{a(\nu)}} \Big( \widetilde{h}_q - h_q \Big)_{q=1,\hdots,n} \stackrel{d}\rightarrow {\mathcal N}(0,M \Sigma_{\lambda} M^*)
\end{equation}
for some weight matrix $M$ (see \eqref{e:weight_matrix_M}) and $\Sigma_{\lambda}$ as in Theorem \ref{t:asympt_normality_lambda2}.
\item [(e)] there is a consistent sequence of wavelet eigenvectors for $P$ assuming $P \in O(n)$ in \eqref{e:H_Jordan}.
\end{itemize}
\item [$(iii)$] If, in addition, $B_{H}$ satisfies (OFBM3$''$), then
\begin{itemize}
\item [(f)] the weak limits \eqref{e:asympt_normality_lambda2_discrete} and \eqref{e:h^q_asymptotically_normal_discrete} hold.
\end{itemize}
\end{itemize}
\end{theorem}

\begin{remark}
The assumption \eqref{e:eigen-assumption_stronger} stems from a technical condition for the existence of a convenient moving average representation of OFBM (see Didier and Pipiras \cite{didier:pipiras:2011}, Theorem 3.2). Even after removing \eqref{e:eigen-assumption_stronger}, the properties listed in Theorem 4.1 are expected to hold in general.
\end{remark}

\begin{figure}[!h]
\centerline{
\includegraphics[width=0.5\linewidth]{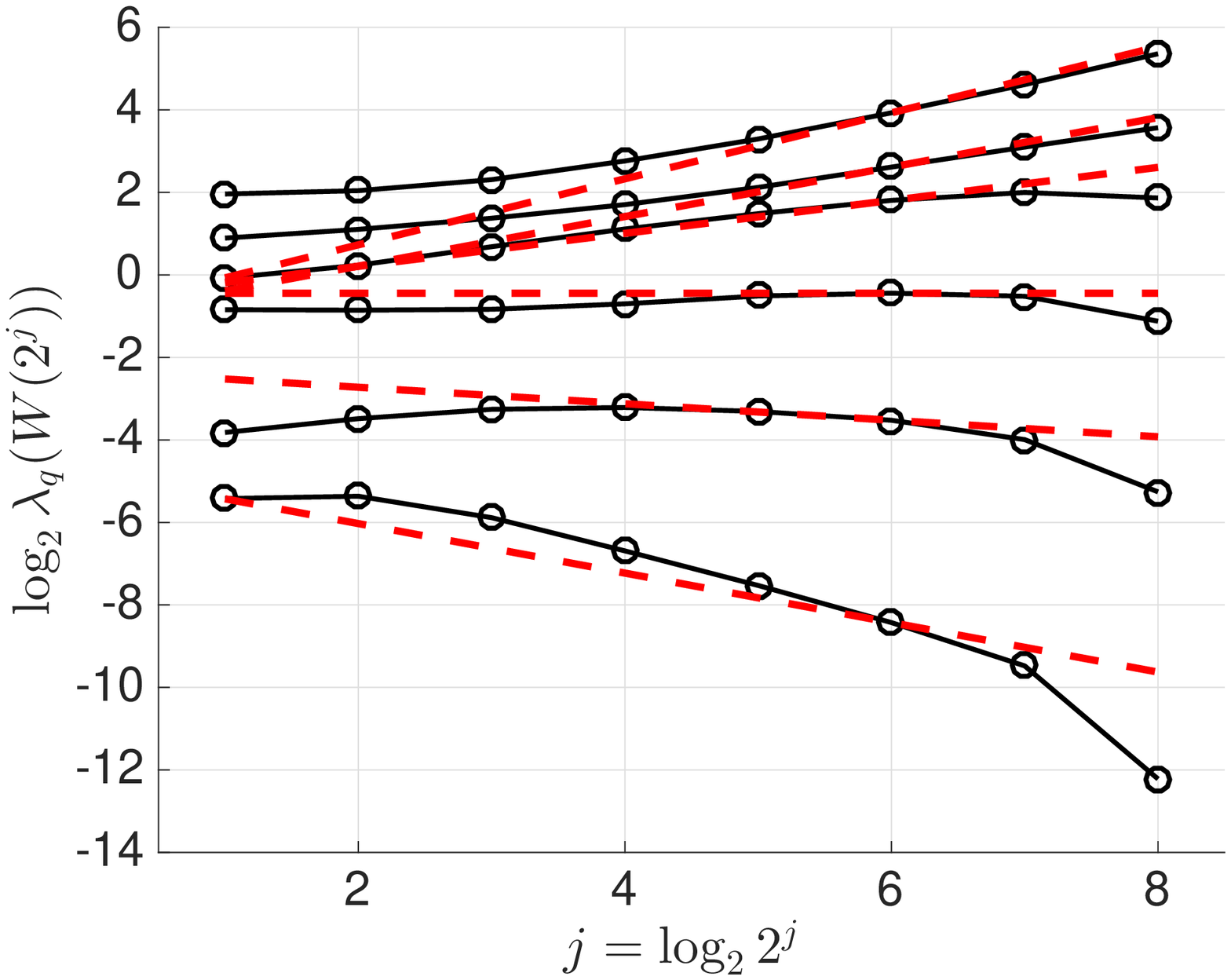}
\includegraphics[width=0.5\linewidth]{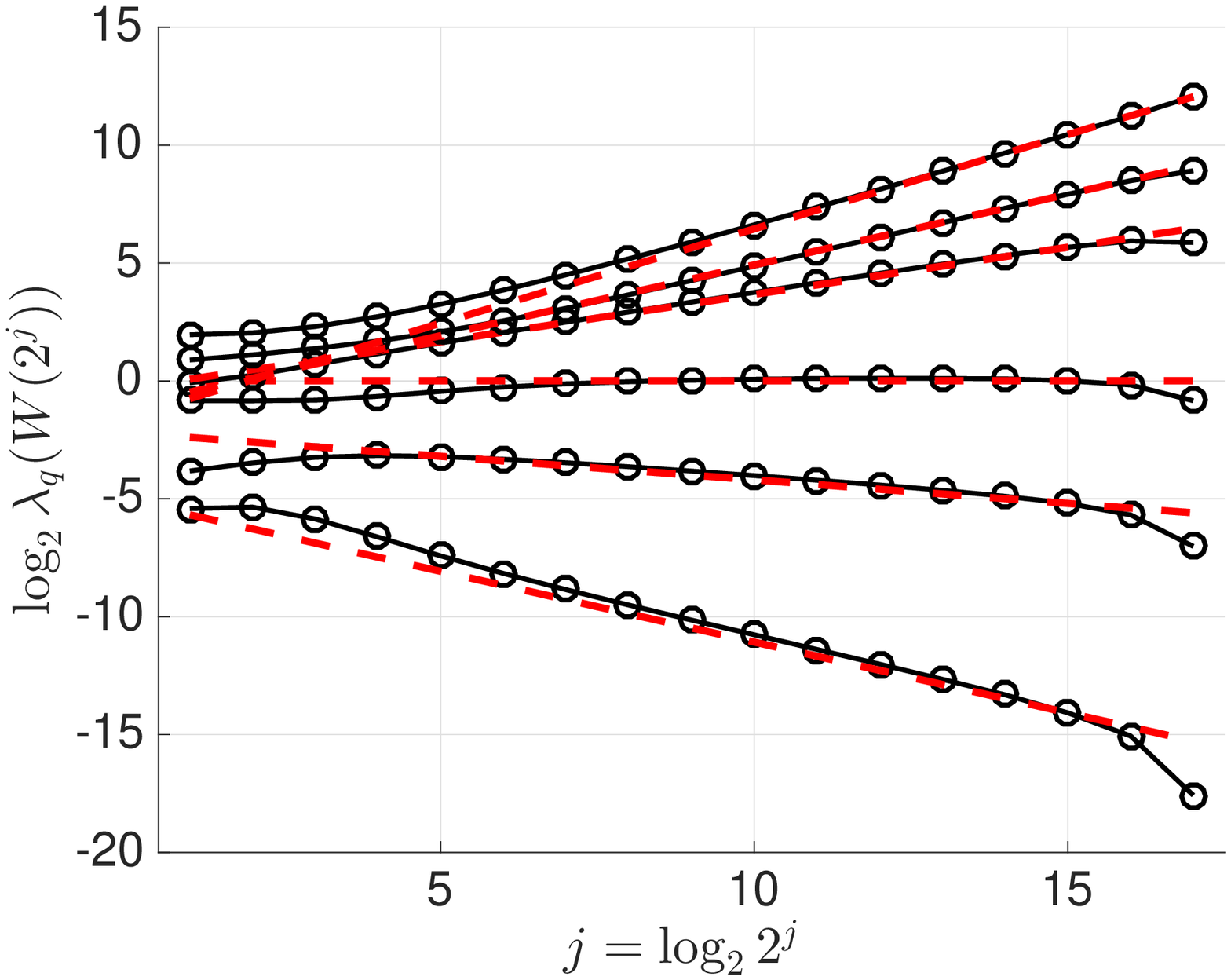}
}
\centerline{
\includegraphics[width=0.5\linewidth]{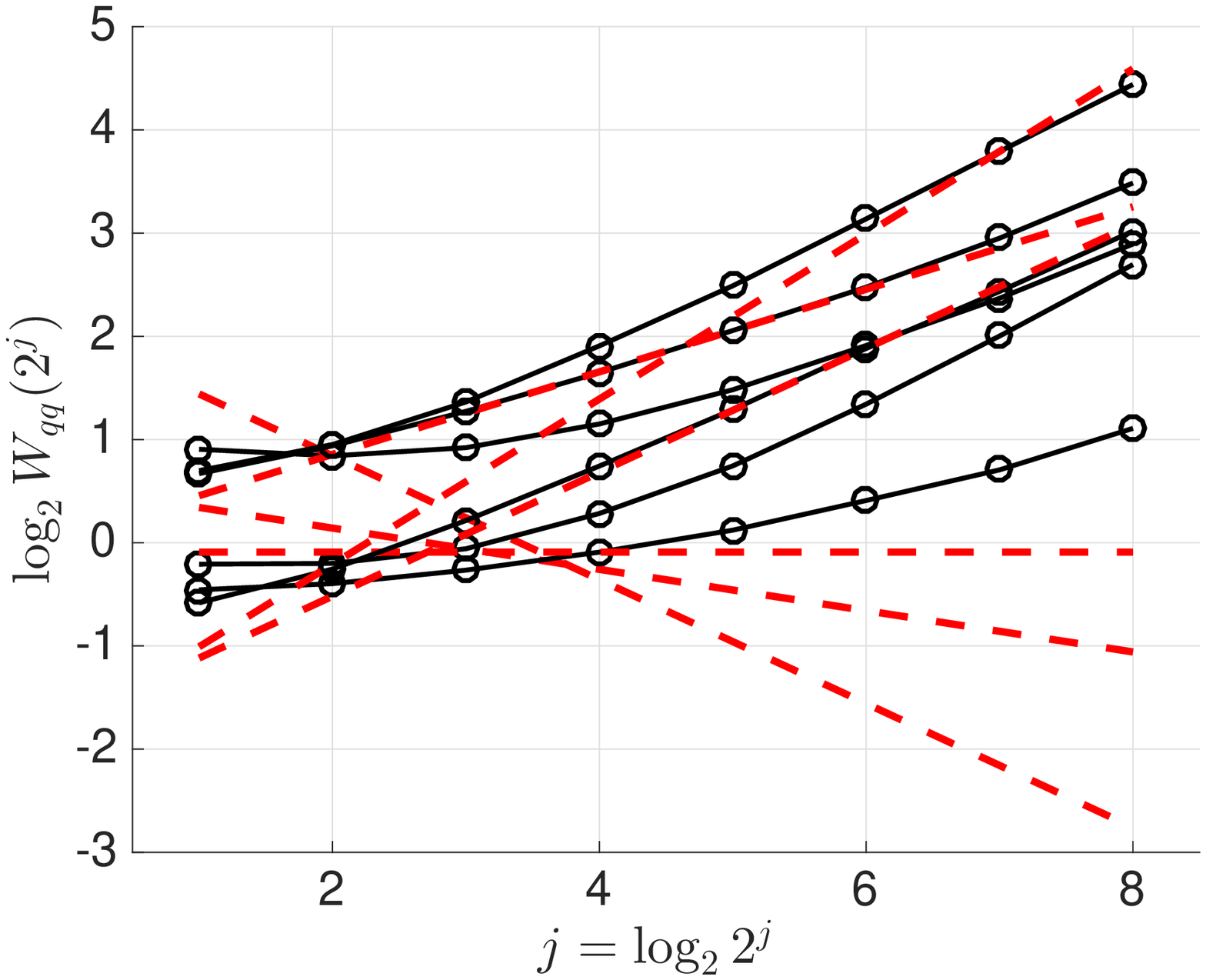}
\includegraphics[width=0.5\linewidth]{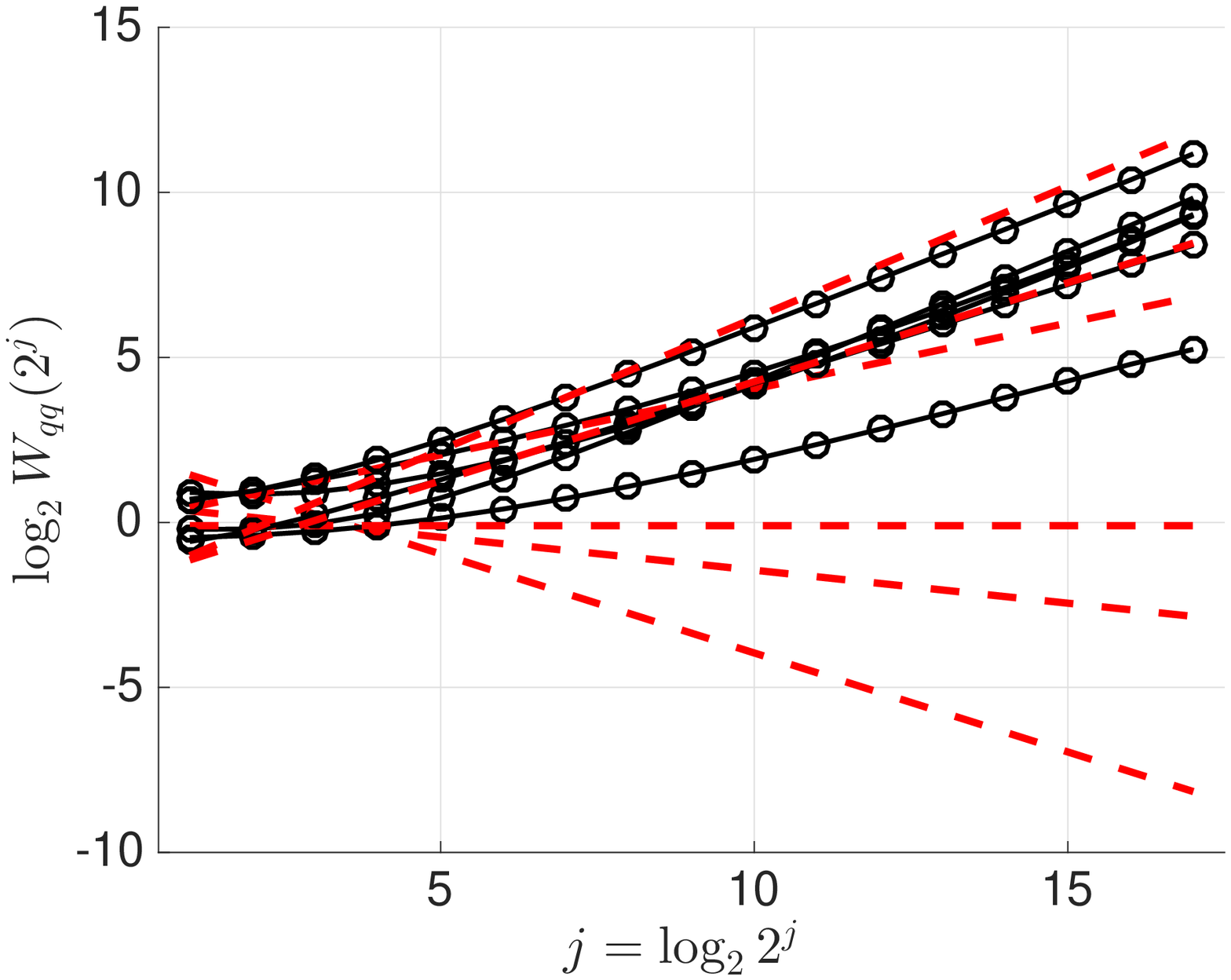}
}
\caption{\label{figa} {\bf Logscale diagrams: univariate-like vs multivariate analysis.} Superimposition of the Monte Carlo averages $\log_2  \langle \lambda_q(W(2^j)) \rangle_{1000} $ (top plots) and $\log_2  \langle W(2^j)_{qq} \rangle_{1000} $ (bottom plots), with the theoretical asymptotic trends  $c_q + 2h_q  \times j$ (dashed red lines), $q = 1,\hdots,6$, for two different sample sizes (left, $\nu= 2^{10}$; right $\nu=2^{20}$). Univariate-like analysis fails to capture the theoretical asymptotic trend and replicates, for all $q$, the trend for the largest Hurst eigenvalue $c_6 + 2 h_{6}  \times j$. Multivariate analysis captures the correct theoretical asymptotic trend $c_q + 2h_q  \times j$ for each $q$.
}
\end{figure}

\section{Monte Carlo studies}\label{s:mc}

\noindent {\bf Numerical experiment setting.} To study the performance of the estimator \eqref{e:hl-hat}, broad Monte Carlo experiments were conducted for sample sizes in the range $ \nu = 2^{10}, \hdots, 2^{20}$, with 1,000 independent OFBM sample paths for each of the latter.
The synthesis of OFBM was performed using the multivariate toolbox devised in Helgason et al.\ \cite{Helgason_H_2011_j-sp_fessmgtsuce,Helgason_H_2011_j-sp_smsspmdccme} and available at \texttt{www.hermir.org}.
We opted for showing results in dimension $n=6$ as representative of the general multivariate situation $n \geq 2$, while keeping the number of plots reasonable.
Results are reported for a single representative instance of OFBM with Hurst eigenvalues
\begin{equation}\label{e:h1=<...=<h6}
h_1 = 0.3, \quad h_2 = 0.4, \quad h_3 = 0.5, \quad h_4 = 0.7, \quad h_5 = 0.8, \quad h_6 = 0.9,
\end{equation}
and Hurst eigenvector matrix
\begin{equation}
P = \left(\begin{array}{cccccc}
0.6468  &  0.3846   & 0.4436 &  -0.5175    &     0  &  0.4000 \\
-0.3234  &  0.7692  & -0.5070   &   0   &  0.1387  &  0.4667 \\
0.1941  &  -0.1538  &  0.6337  & -0.3696  & -0.1387   &      0 \\
   -0.2587  &  0.4615  &  0.3802  & 0.7392 &  -0.4160  &  0.4000 \\
    0.3234   &      0     &    0  &       0  &  0.6934  & -0.1333  \\
    0.5175   & 0.1538     &    0  &  -0.2218  &  0.5547  &  0.6667 \\
\end{array}\right),
\end{equation}
since similar conclusions can be drawn from several other instances.

The analysis was conducted using orthogonal least asymmetric Daubechies wavelets, with $N_\psi = 2$ vanishing moments.
It has been checked that varying $N_\psi \geq 2$ or using other regular enough wavelets yields qualitatively identical conclusions.
The log-linear regressions \eqref{e:hl-hat} were performed across scales $(j_1,j_2) = (6, \log_2 \nu -N_\psi)$ using weights
$$
w_{j} = b_{j} \frac{V_0 \hspace{0.5mm}j - V_1}{V_0 V_2 - V^2_1}, \quad j = j_1,\hdots,j_2, \quad \quad V_p := \sum^{j_2}_{j=j_1} j^p b_{j} , \quad p =0,1,2,
$$
which satisfy \eqref{e:sum_wj=0,sum_jwj=1}.
The scalars $b_{j} \geq 0$ can be freely chosen and reflect the degree of confidence in each term $\log_2 \lambda_{q}( 2^j)$.
Following Abry et al.\ \cite{abfrv:2002}, we picked $b_j= \nu /2^j$.
We compare the estimation performance to that of the univariate-like analysis of each component separately, i.e., of the log-linear regressions
$$
\widehat{h}^{U}_{q} := \frac{1}{2}\sum_{j=j_1}^{j_2} w_j \log_2  W(2^j)_{qq}, \quad q =1, \ldots, n,
$$
based on the main diagonal entries of $W(2^j)$ (see, for instance, Veitch and Abry \cite{veitch:abry:1999} and Ciuciu et al.\ \cite{ciuciu:abry:he:2014}).\\

\noindent {\bf Estimation principle.} To illustrate the estimation procedure, for each $q = 1,\hdots,6$ and for the smallest $ \nu = 2^{10} $ and largest $\nu = 2^{20}$ sample sizes, Figure~\ref{figa} compares the multivariate and univariate-like wavelet analysis functions $\log_2  \langle \lambda_{q}(W(2^j)) \rangle_{1000} $ (top plots) and $\log_2  \langle W(2^j)_{qq} \rangle_{1000} $, respectively. The symbol $\langle \cdot \rangle_{1000}$ denotes the Monte Carlo average, used as a numeric surrogate for the ensemble average $\bbE \cdot$.

Figure~\ref{figa} clearly shows that, for each $q$, the Monte Carlo averaged univariate-like analysis functions $\log_2  \langle W(2^j)_{qq} \rangle_{1000} $ fail to reproduce the theoretical asymptotic behavior $c_q+  2h_{q}  \times j$ (dashed red lines) and essentially follow the dominant asymptotic behavior $c_6+  2h_{6}  \times j$. This leads to the incorrect conclusion that the 6 components have the same Hurst eigenvalue $h_6$. By contrast, Figure~\ref{figa} shows that the Monte Carlo averaged multivariate analysis functions $\log_2  \langle \lambda_{q}(W(2^j)) \rangle_{1000} $, $q = 1,\hdots,6$, closely follow the theoretical asymptotic behavior $c_q+  2h_{q}  \times j$. This provides evidence of the existence of different Hurst eigenvalues in the multivariate data.
Interestingly, the agreement of observed and theoretical scaling remains very satisfactory even for small sample sizes (in this case, $\nu= 2^{10}$!).\\ 

\begin{figure*}[th]
\centerline{
\includegraphics[width=0.3\linewidth]{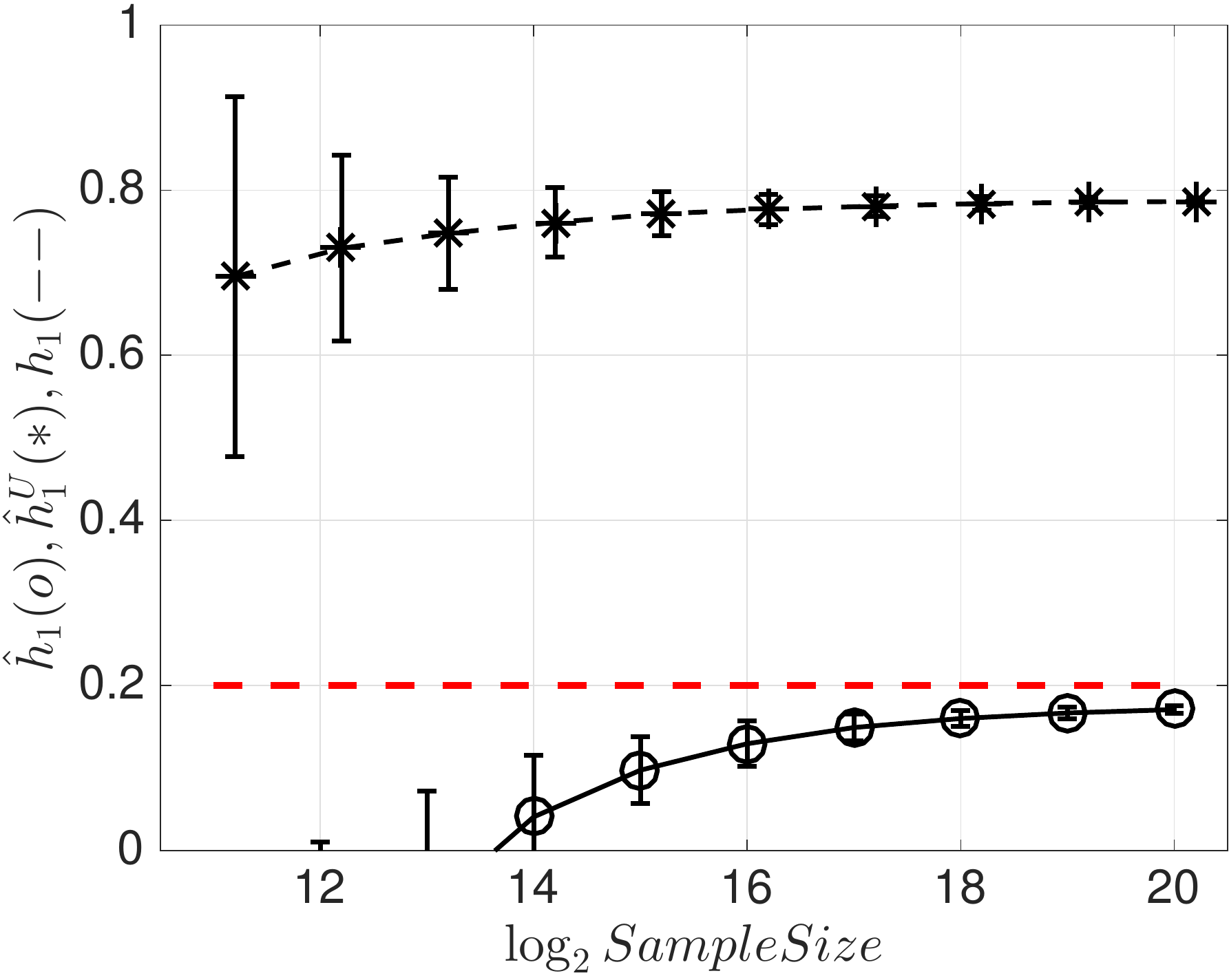}
\includegraphics[width=0.3\linewidth]{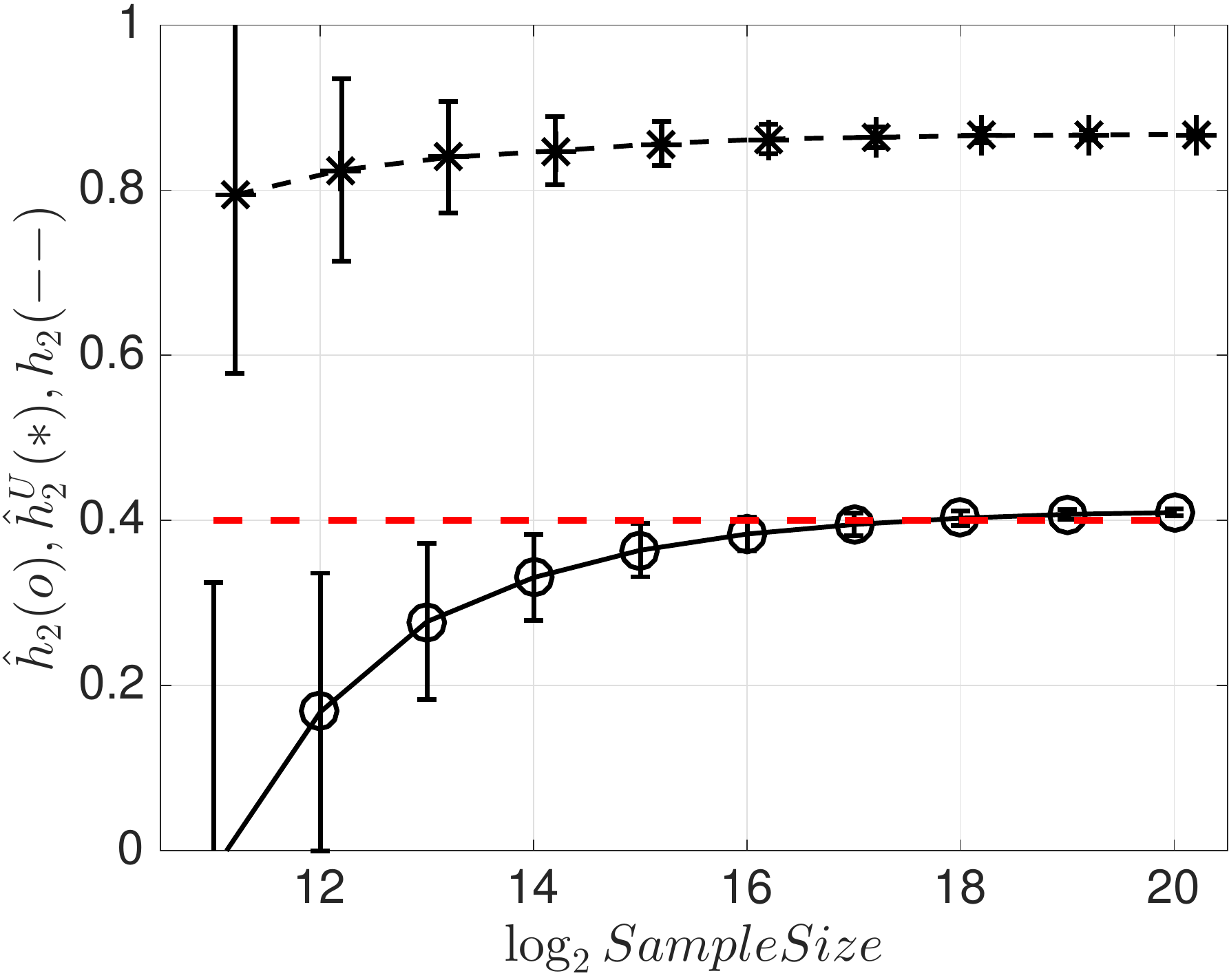}
\includegraphics[width=0.3\linewidth]{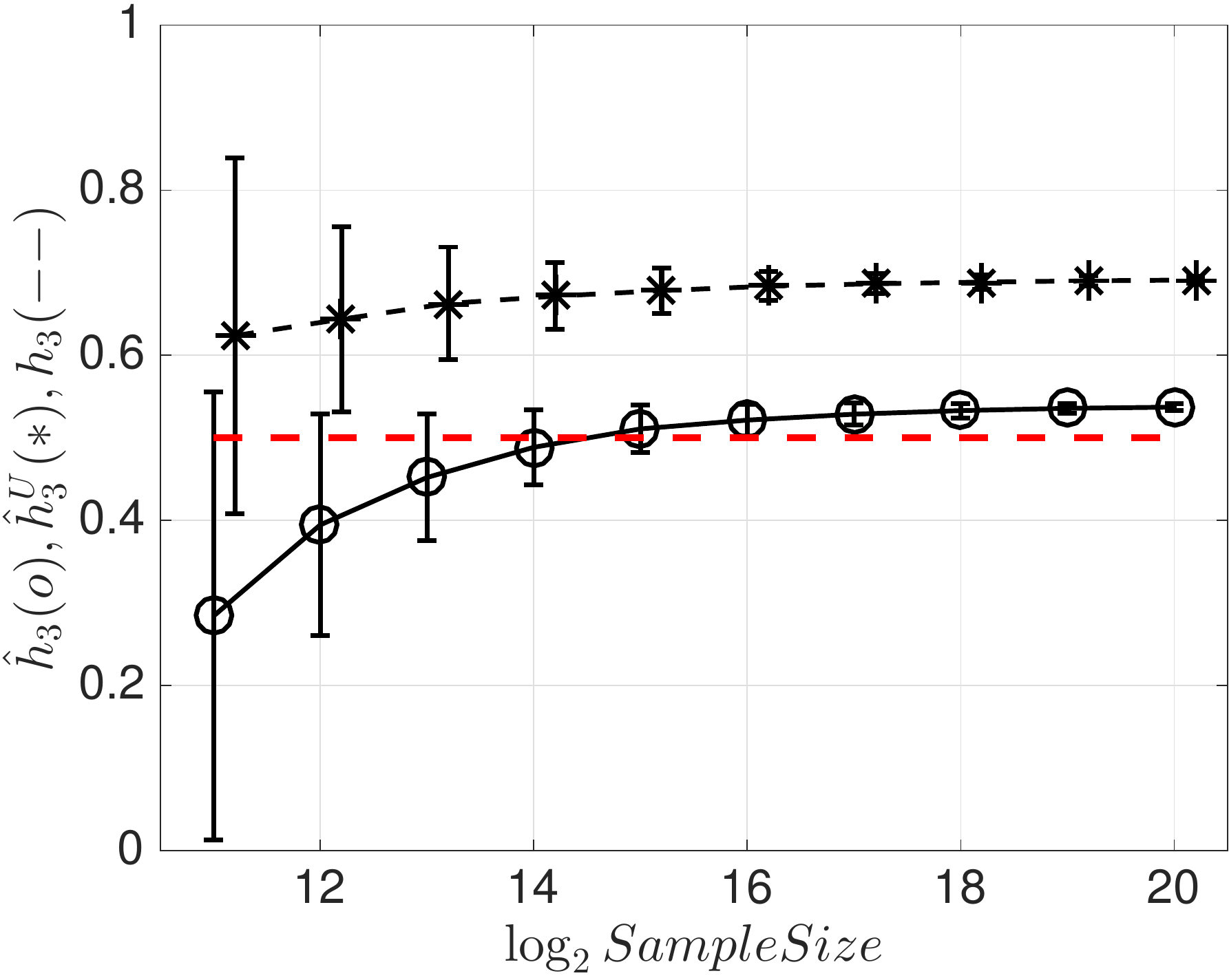}
}
\centerline{
\includegraphics[width=0.3\linewidth]{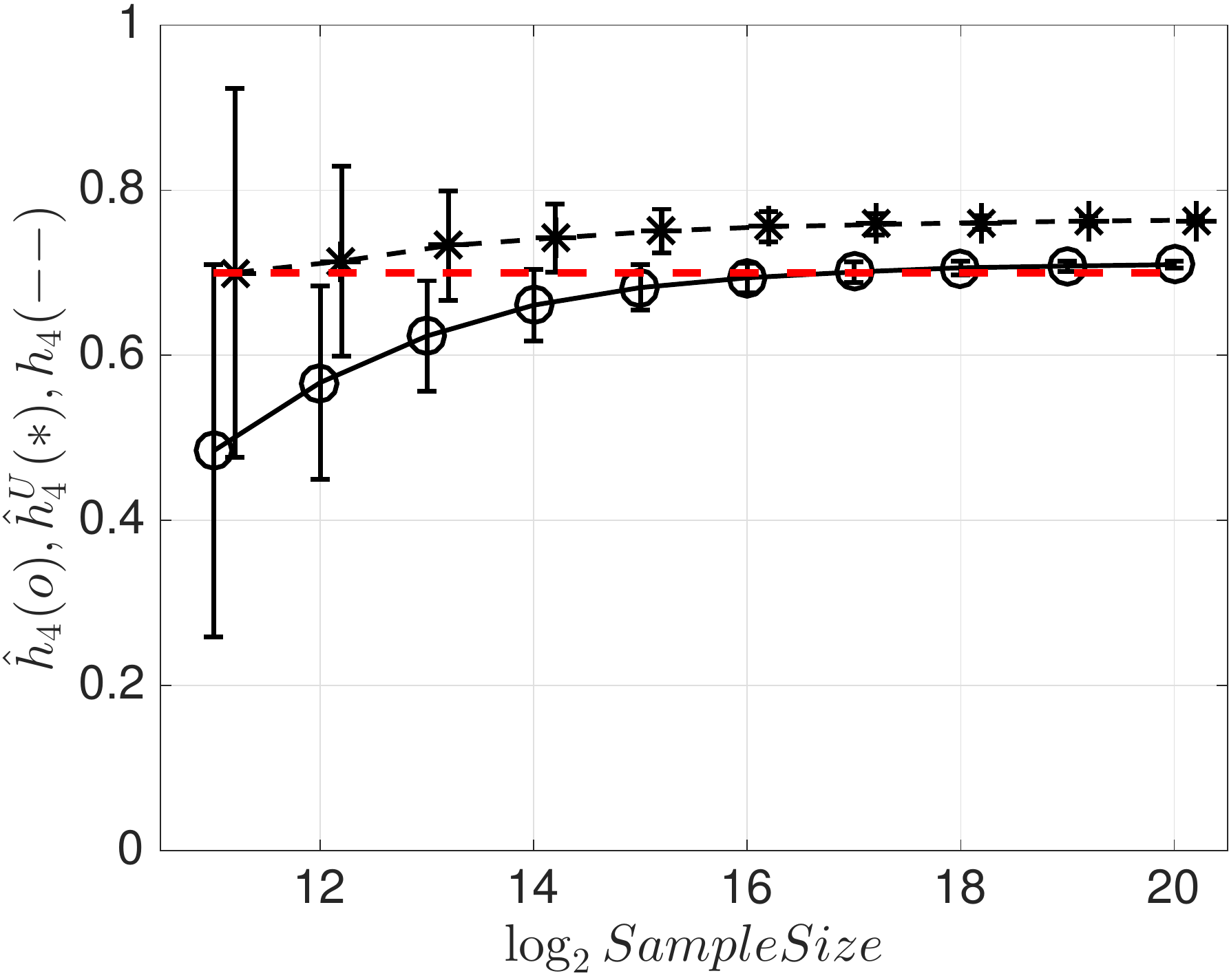}
\includegraphics[width=0.3\linewidth]{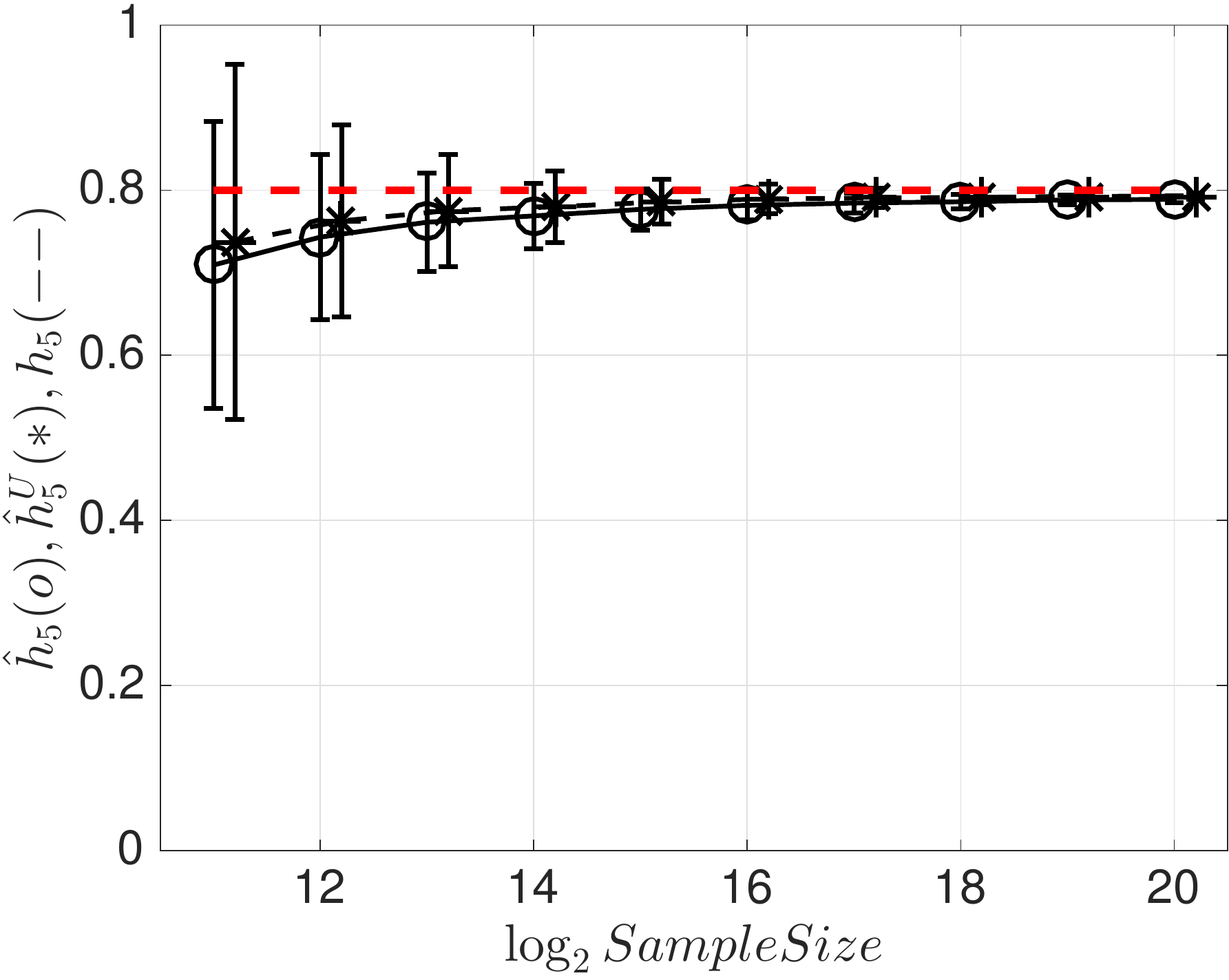}
\includegraphics[width=0.3\linewidth]{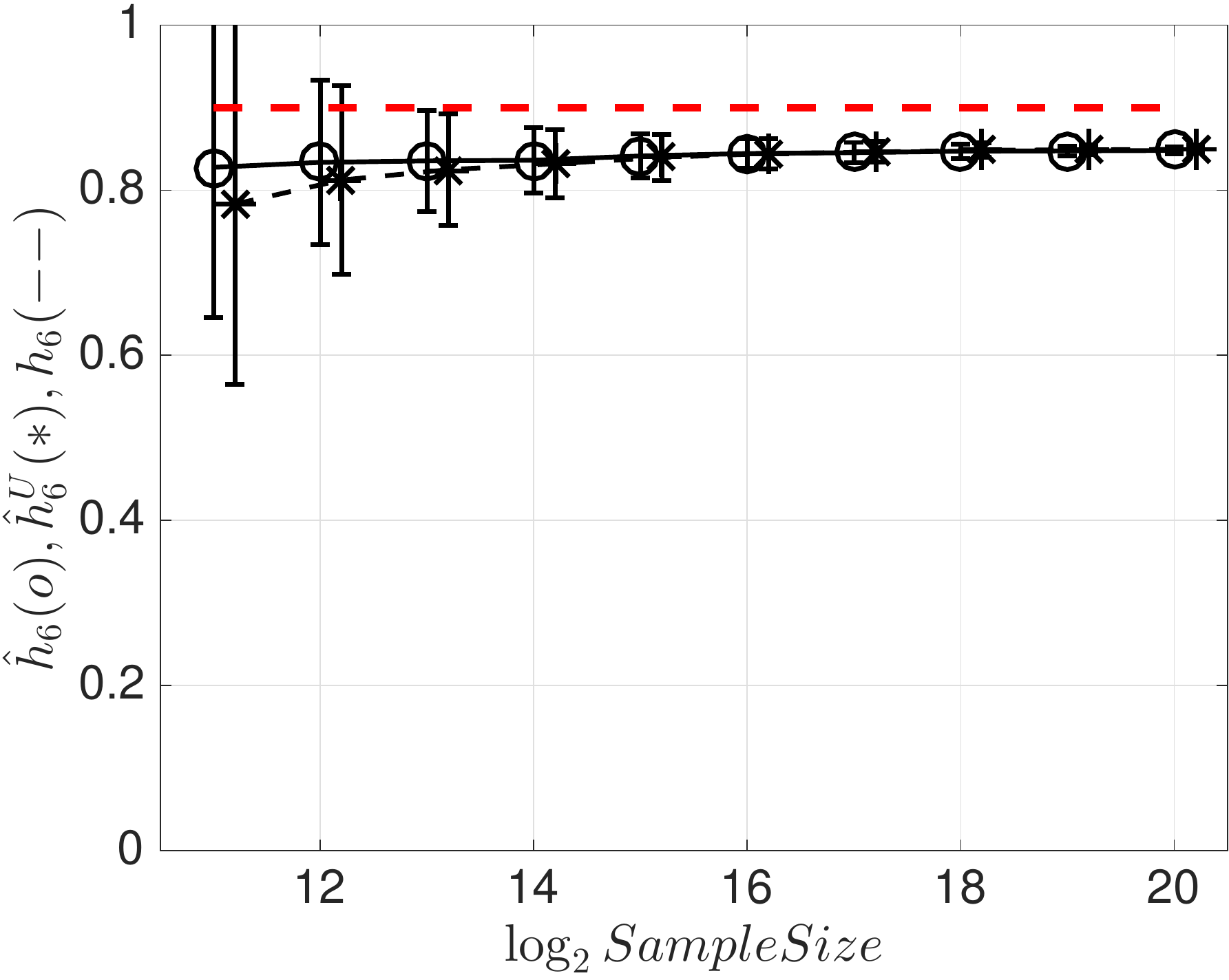}
}
\caption{\label{figb} {\bf Estimation performance (bias).} Bias of each estimator $ \widehat{h}_{q}$ and $ \widehat{h}^{U}_{q}$ as a function of the ($\log_2$ of the) sample size, for $\widehat{h}_q$, $q= 1,\ldots,6$. The horizontal red dashed line indicates the true $h_q$,
the black solid lines with `$o$' represent the Monte Carlo estimate of $\bbE \widehat{h}_q$, the dashed black lines with $\ast$ represent the Monte Carlo based univariate-like estimation  $\widehat{h}^U_q$ of $h_q$ (bootstrapped confidence intervals).
 }
\end{figure*}

\begin{figure*}[th]
\centerline{
\includegraphics[width=0.3\linewidth]{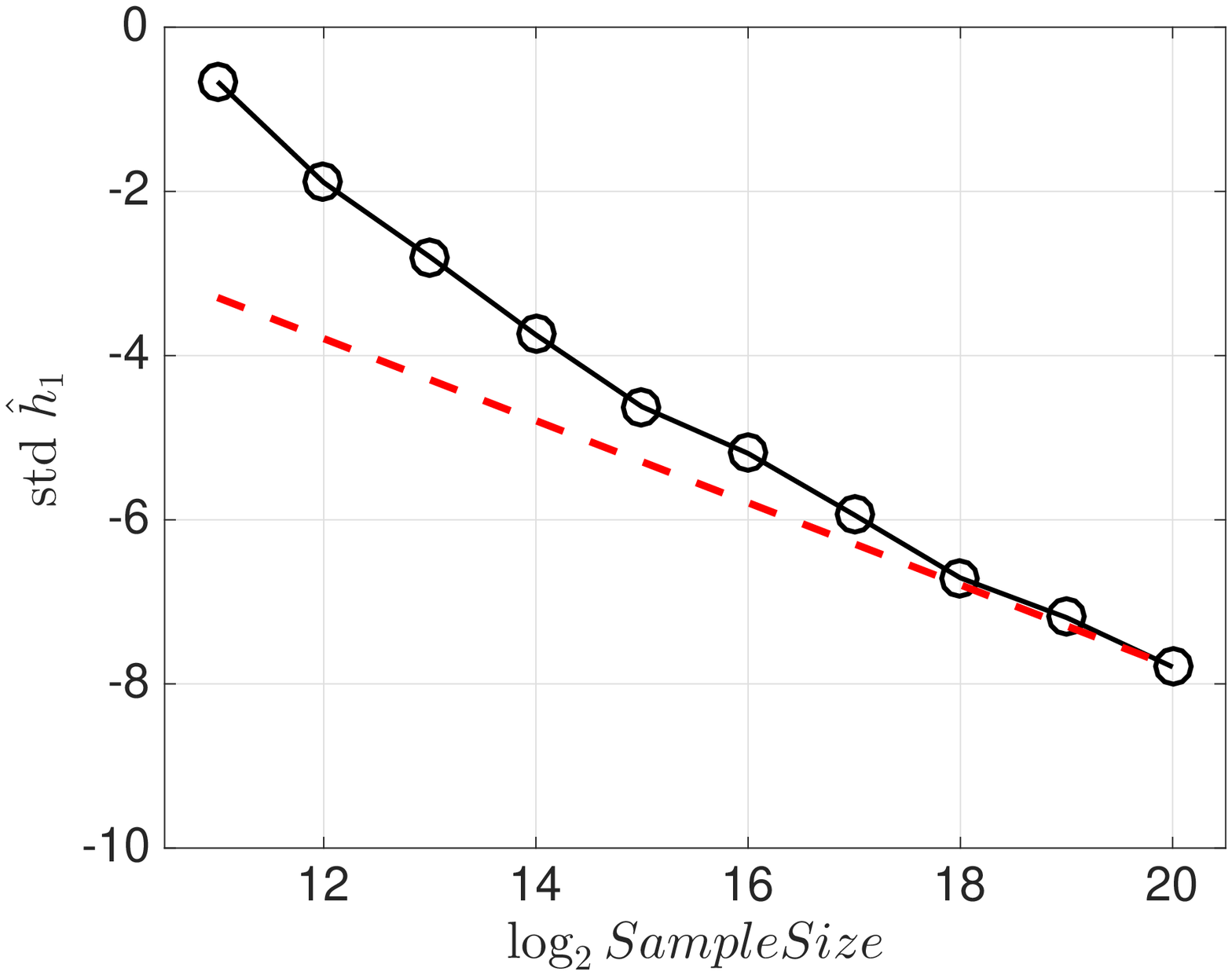}
\includegraphics[width=0.3\linewidth]{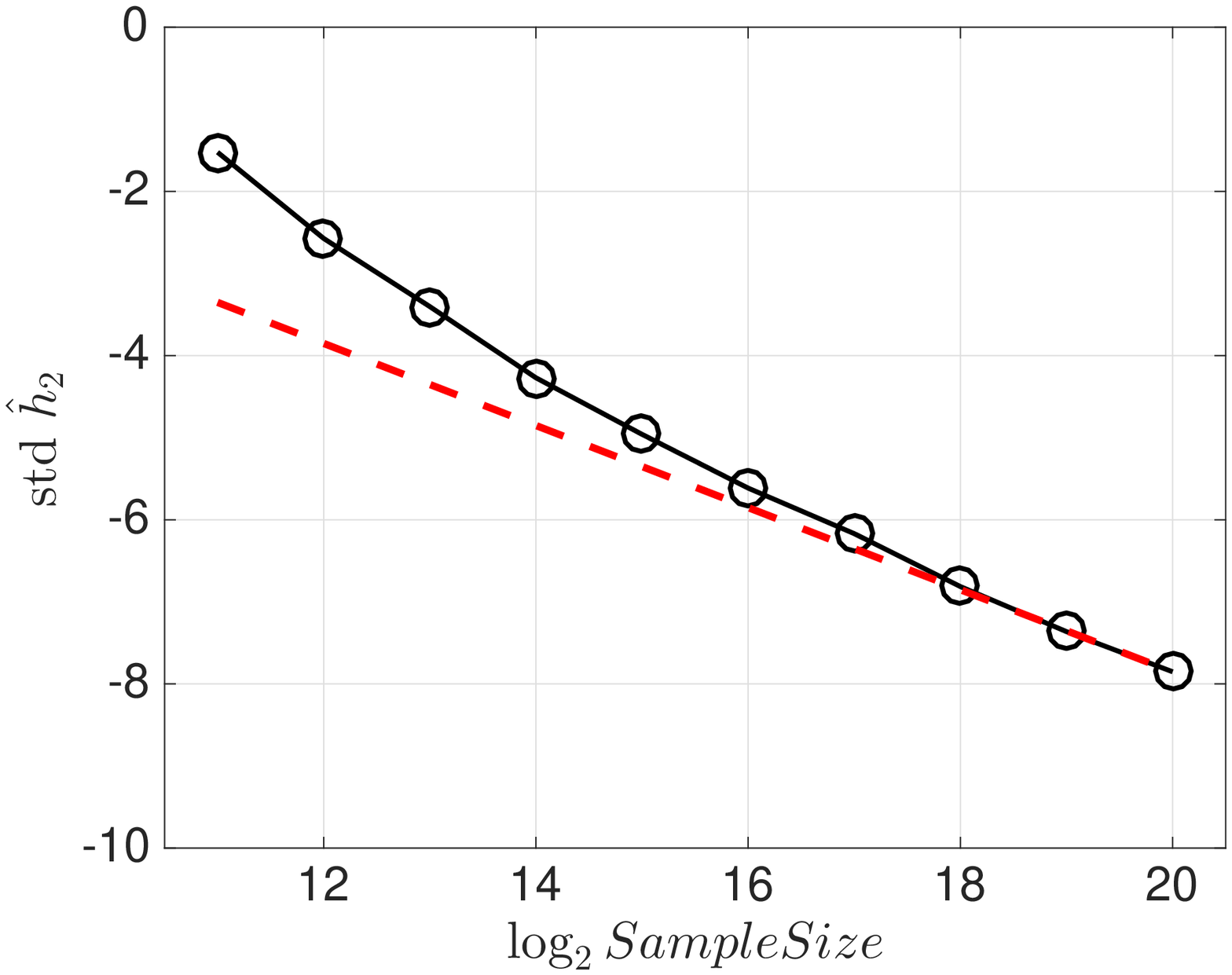}
\includegraphics[width=0.3\linewidth]{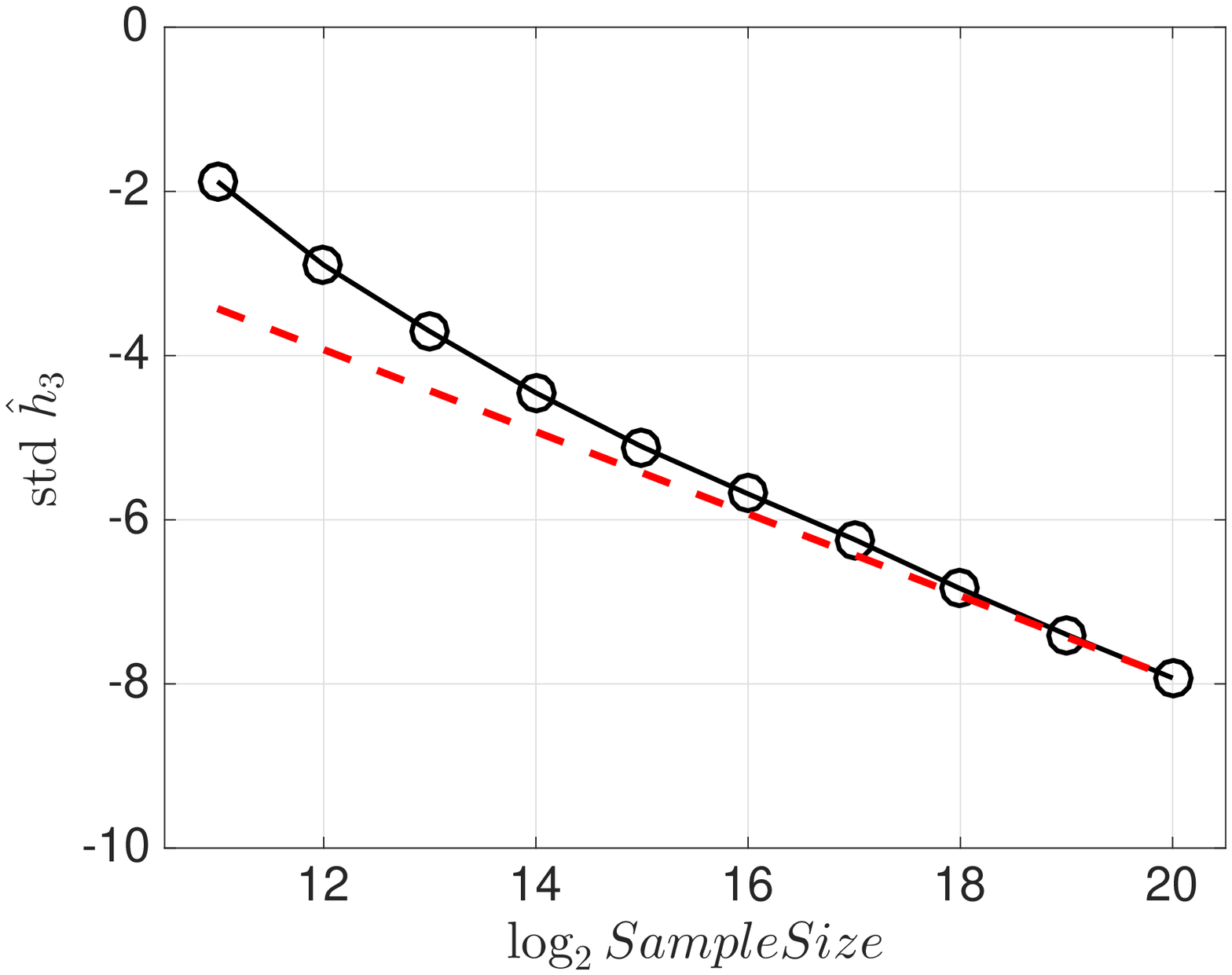}}
\centerline{
\includegraphics[width=0.3\linewidth]{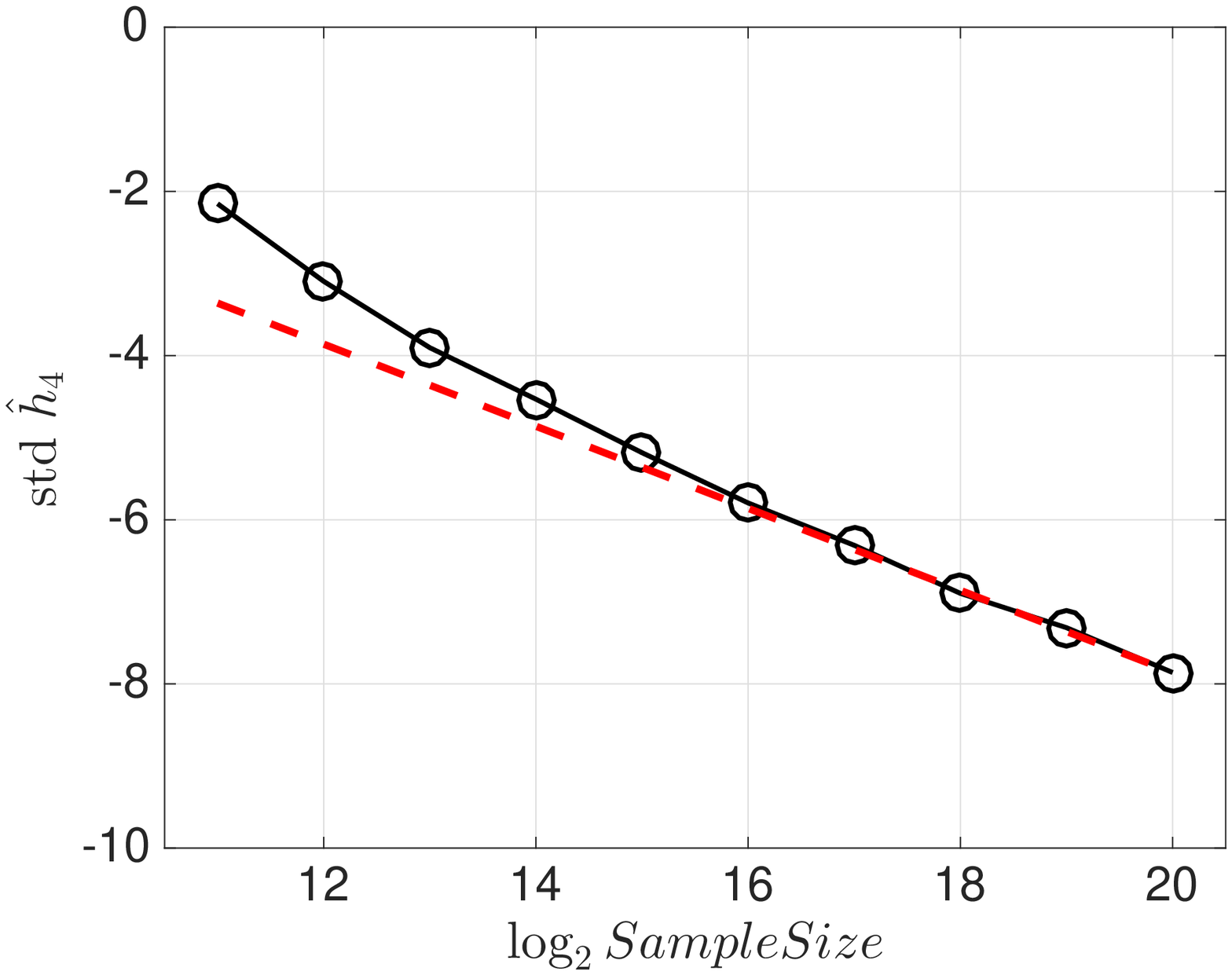}
\includegraphics[width=0.3\linewidth]{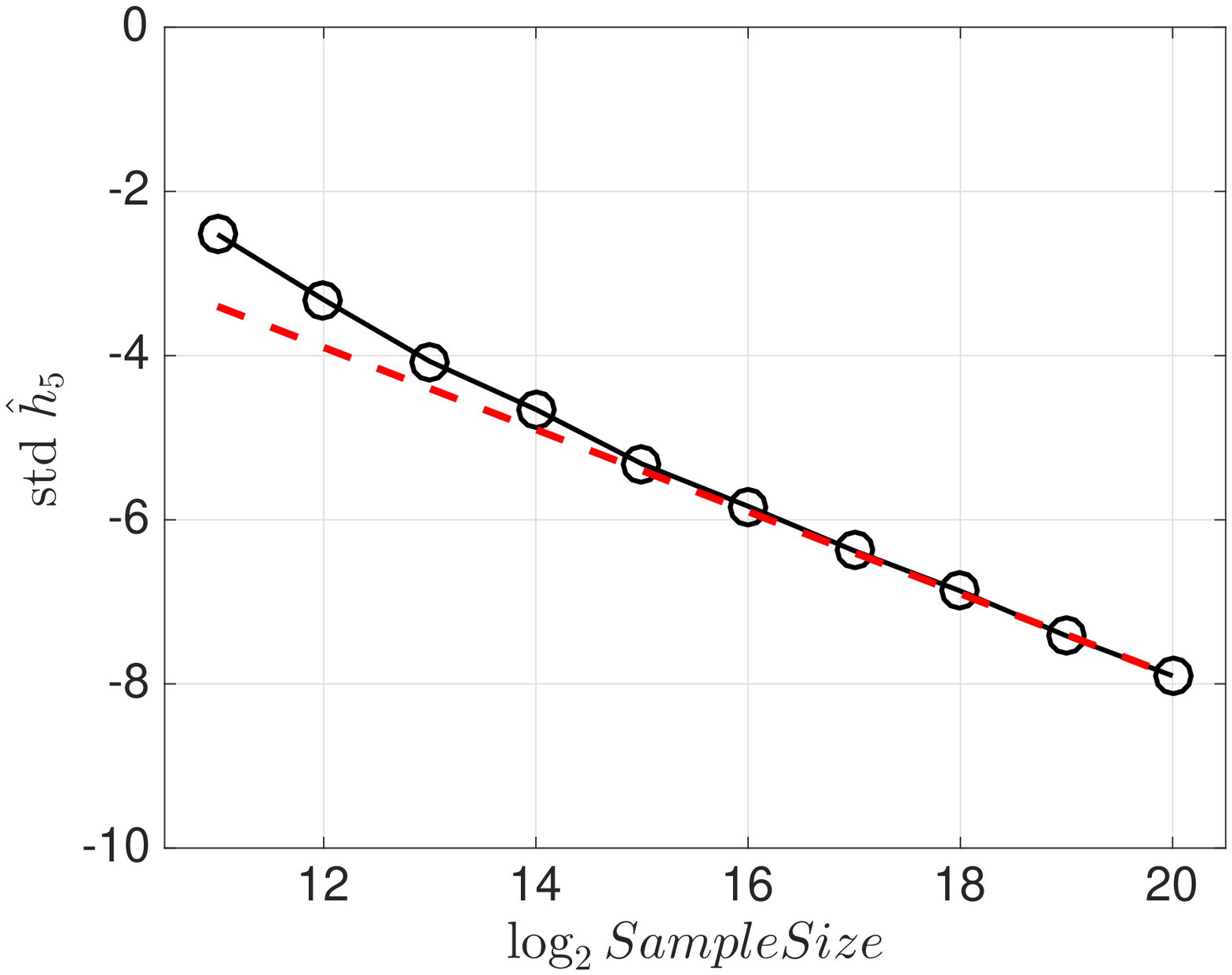}
\includegraphics[width=0.3\linewidth]{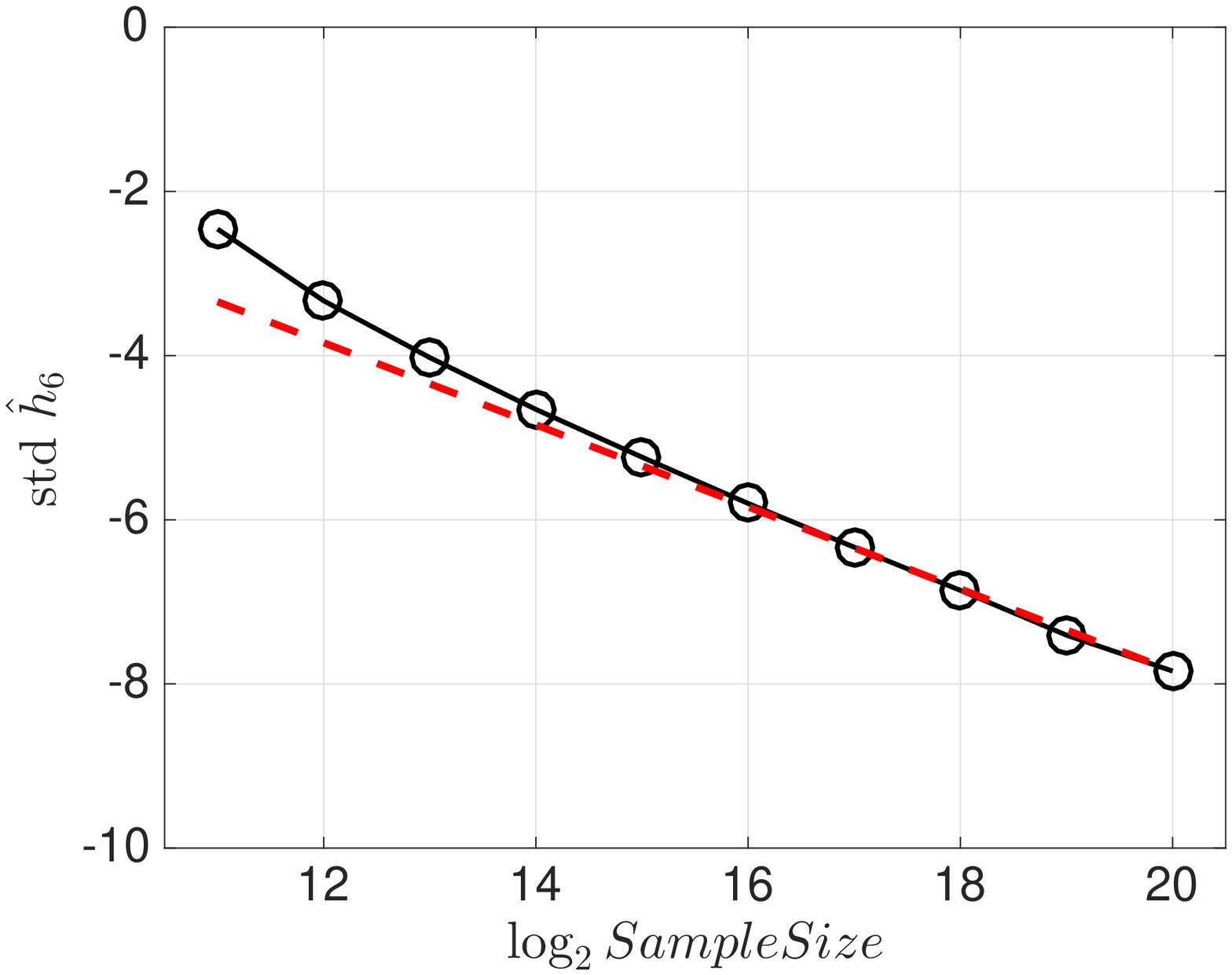}
}
\caption{\label{figbb} {\bf Estimation performance (standard deviation).} ($\log_2$ of ) Standard deviations as functions of ($\log_2$ of the) sample size, for $\widehat{h}_q$, $q= 1,\ldots,6$.
The red dashed line indicates the expected $\nu^{-1/2}$ decrease.
 }
\end{figure*}

\noindent {\bf Bias and standard deviation.} To further assess the estimation performance, in Figure \ref{figb} biases for $ \widehat{h}_{q}$ and $ \widehat{h}^{U}_{q}$, $q = 1,\hdots,6$, are compared as functions of (the $\log_2$ of) the sample size.
The results confirm that the univariate-like estimates $ \widehat{h}^{U}_{q}$ (dashed black lines with $\ast$) are strongly biased, barely departing from the largest Hurst eigenvalue $h_6$. In other words, under an OFBM model, univariate-like data analysis leads practitioners to incorrectly conclude that all components have the same Hurst eigenvalue, i.e., $\widehat{h}^{U}_{q} \simeq h_{q}$, $q=1,\ldots,n$. 

Moreover, biases for the wavelet eigenstructure estimators $ \widehat{h}_{q}$  decrease with sample size for all $q$, as predicted by Theorem \ref{t:asympt_normality_lambda2}. Unsurprisingly, the simulations further show that the accurate estimation of the smaller Hurst eigenvalues is more demanding in terms of data by comparison to larger Hurst eigenvalues. While the estimation of $h_6$ shows negligible bias for a sample size as small as $\nu=2^{10}$, equally accurate estimation of $h_1$ requires $\nu=2^{16}$.

Figure \ref{figbb} further shows that Monte Carlo standard deviations for $\widehat{h}_{q}$ decay as $\nu^{-1/2}$. Interestingly, the amplitude of standard deviations depends neither on each individual value $h_{q}$ nor, globally, on the ensemble of parameters \eqref{e:h1=<...=<h6}. These results constitute two very remarkable features of the proposed estimation procedure, which is strongly reminiscent of what was observed in univariate estimation for FBM (see Veitch and Abry \cite{veitch:abry:1999}).

In addition, Monte Carlo experiments not reported indicate that, surprisingly, biases and standard deviations neither depend (significantly) on the off-diagonal entries of the instantaneous covariance $\bbE B_H(1)B_H(1)^*$ (i.e., on correlations among pre-mixed components), nor on the choice of the Hurst eigenvector matrix $P$. This is another striking feature of the performance of the estimators \eqref{e:hl-hat}.\\

\begin{figure}[th]

\centerline{
\includegraphics[width=0.4\linewidth]{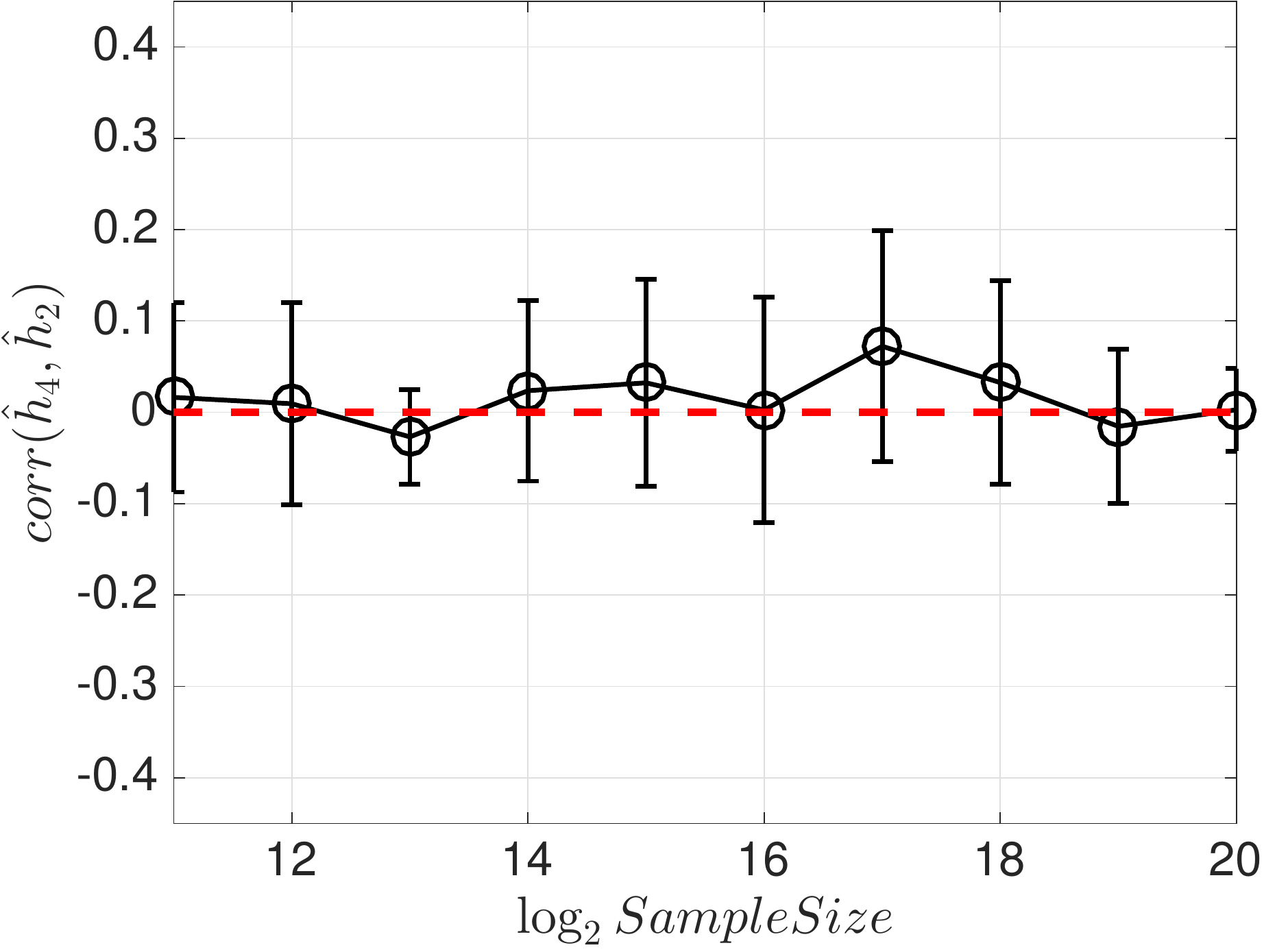}
\includegraphics[width=0.4\linewidth]{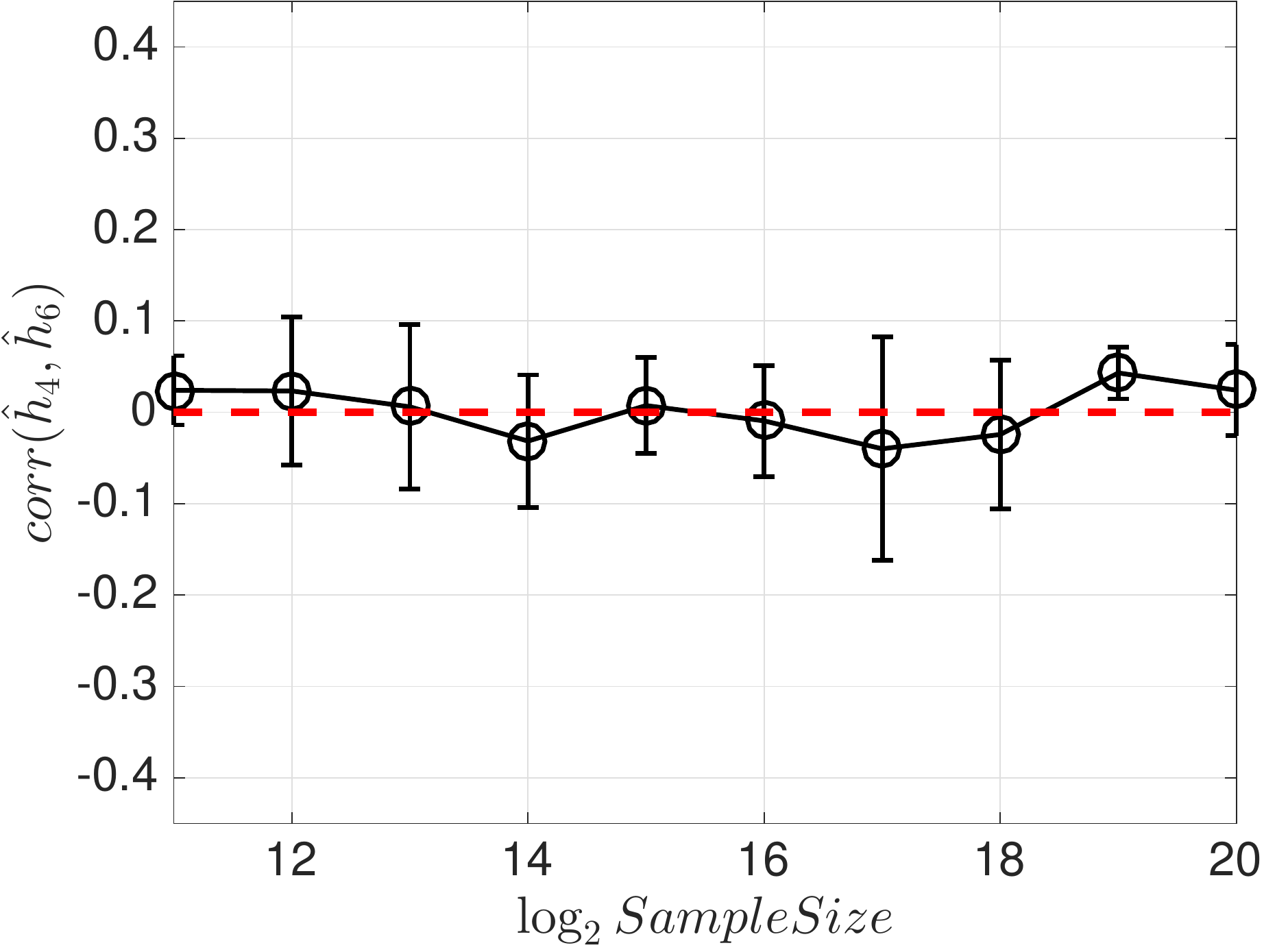}
}
\centerline{
\includegraphics[width=0.4\linewidth]{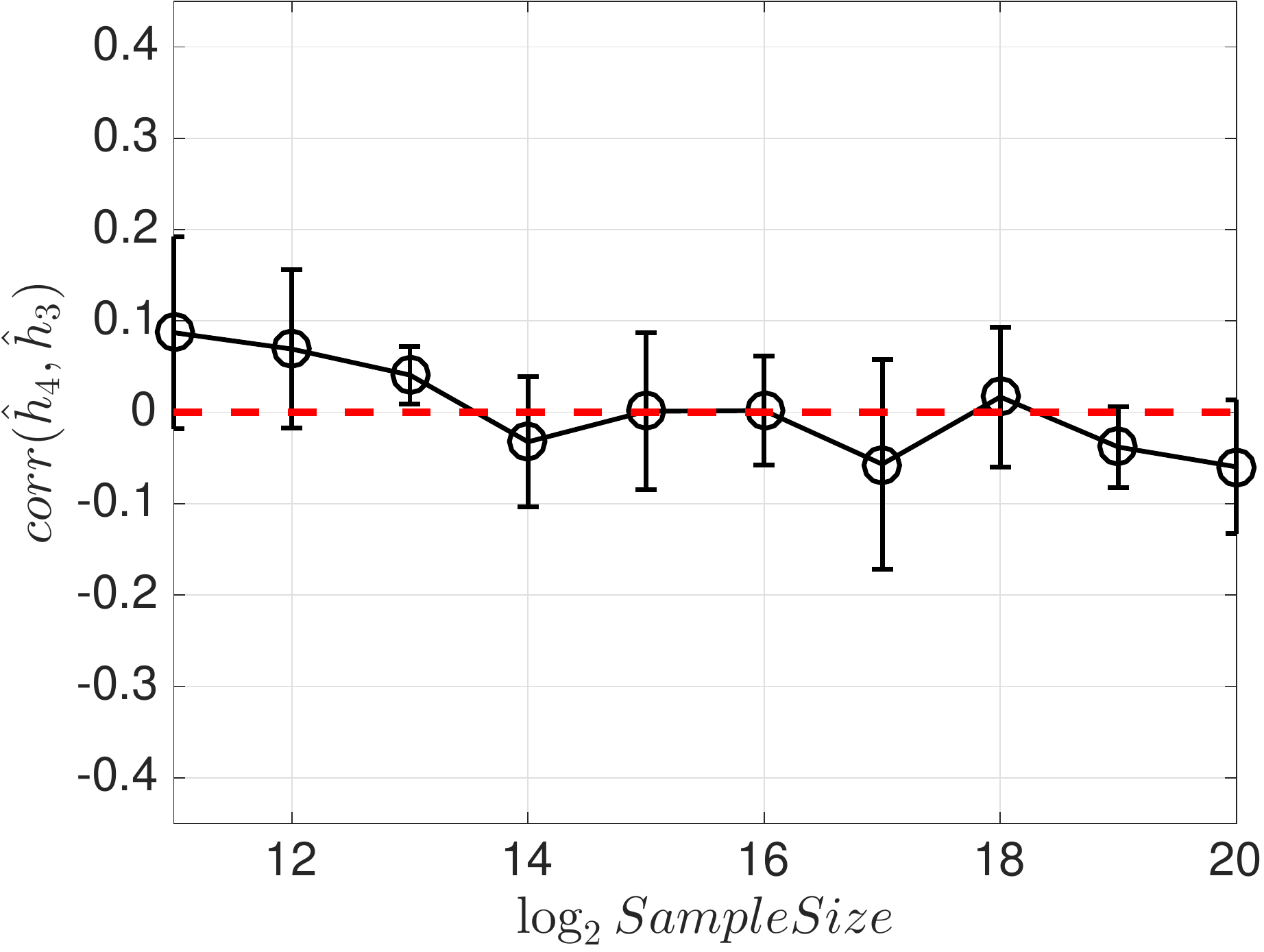}
\includegraphics[width=0.4\linewidth]{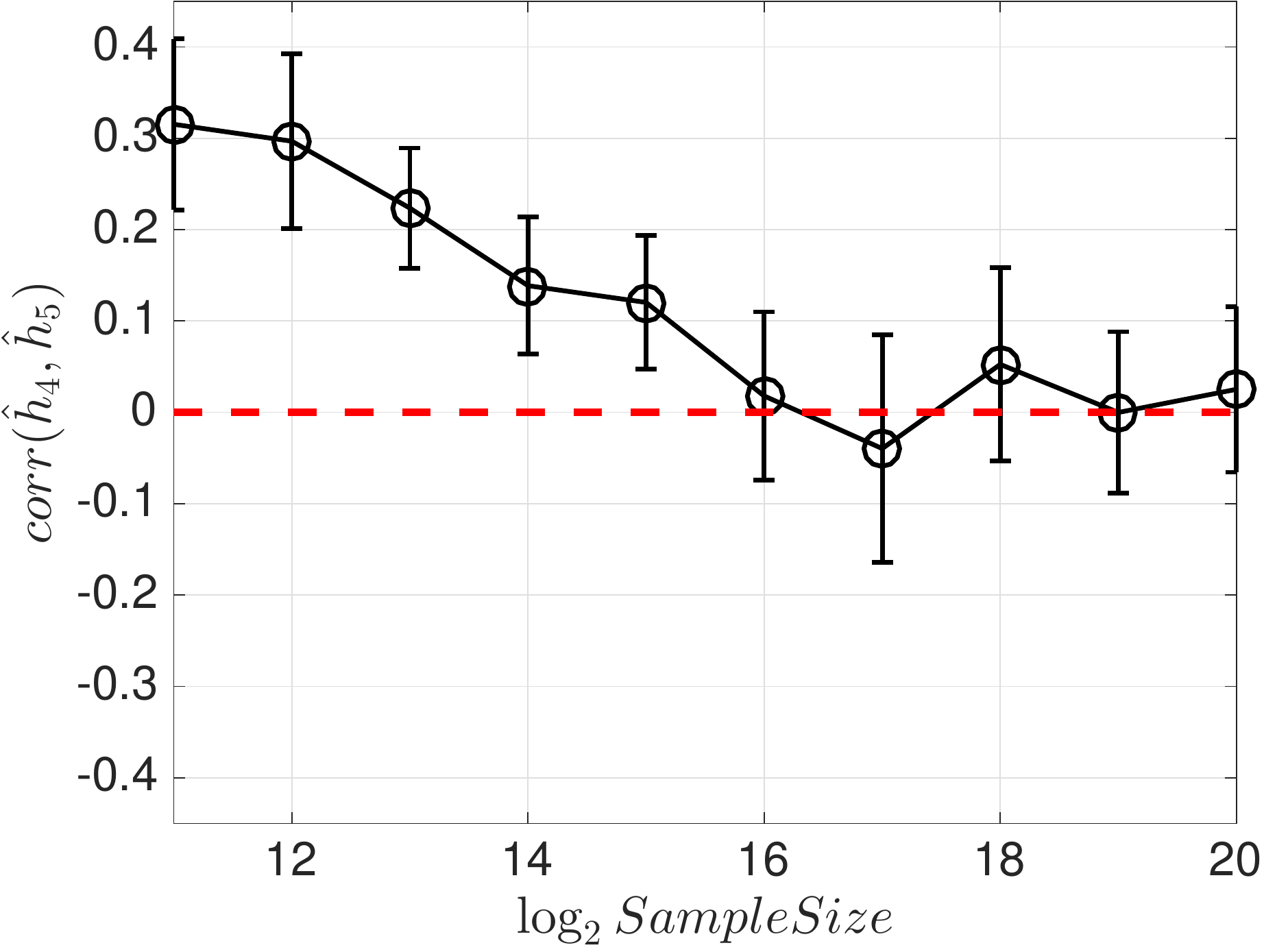}
}
\caption{\label{figd} {\bf Estimation performance (covariance).} Covariance between $\widehat{h}_{q}$ and $\widehat{h}_{q'}$ as a function of the sample size (boostrapped confidence intervals). }
\end{figure}

\noindent {\bf Covariance amongst estimates $ \widehat{h}_{q}$.} Figure \ref{figd} indicates that, asymptotically, the covariances of $\widehat{h}_{q}$ and $\widehat{h}_{q'}$, $q \neq q'$, tend to $0$.
Monte Carlo experiments also consistently showed that $\widehat{h}_{q}$ is generally correlated with $\widehat{h}_{q+1}$ and $\widehat{h}_{q-1}$ (Figure \ref{figd}, bottom plots), with decreasing covariances, while the covariances between $\widehat{h}_{q}$ and $\widehat{h}_{q'}$ with $|q-q'| \geq 2$ are remarkably close to $0$ even for small sample sizes (e.g., Figure \ref{figd}, top plots). These are important facts to be accounted for in practice.\\

\noindent {\bf Asymptotic normality of $ \widehat{h}_{q}$.} Figure \ref{figc} displays the skewness and (excess) kurtosis of the finite sample distribution of the Hurst eigenvalue estimators $\widehat{h}_{q}$. Both measures decrease as the sample size increases. Moreover, the plots provide a measure of the sample sizes needed for an accurate Gaussian approximation to the distribution of each estimator $\widehat{h}_{q}$. In particular, Figure \ref{figc} indicates that normality is reached much faster (i.e., for much smaller sample sizes) for the larger Hurst eigenvalue than for the smaller ones. \\

\begin{figure}[thb]
\centerline{
\includegraphics[width=0.17\linewidth]{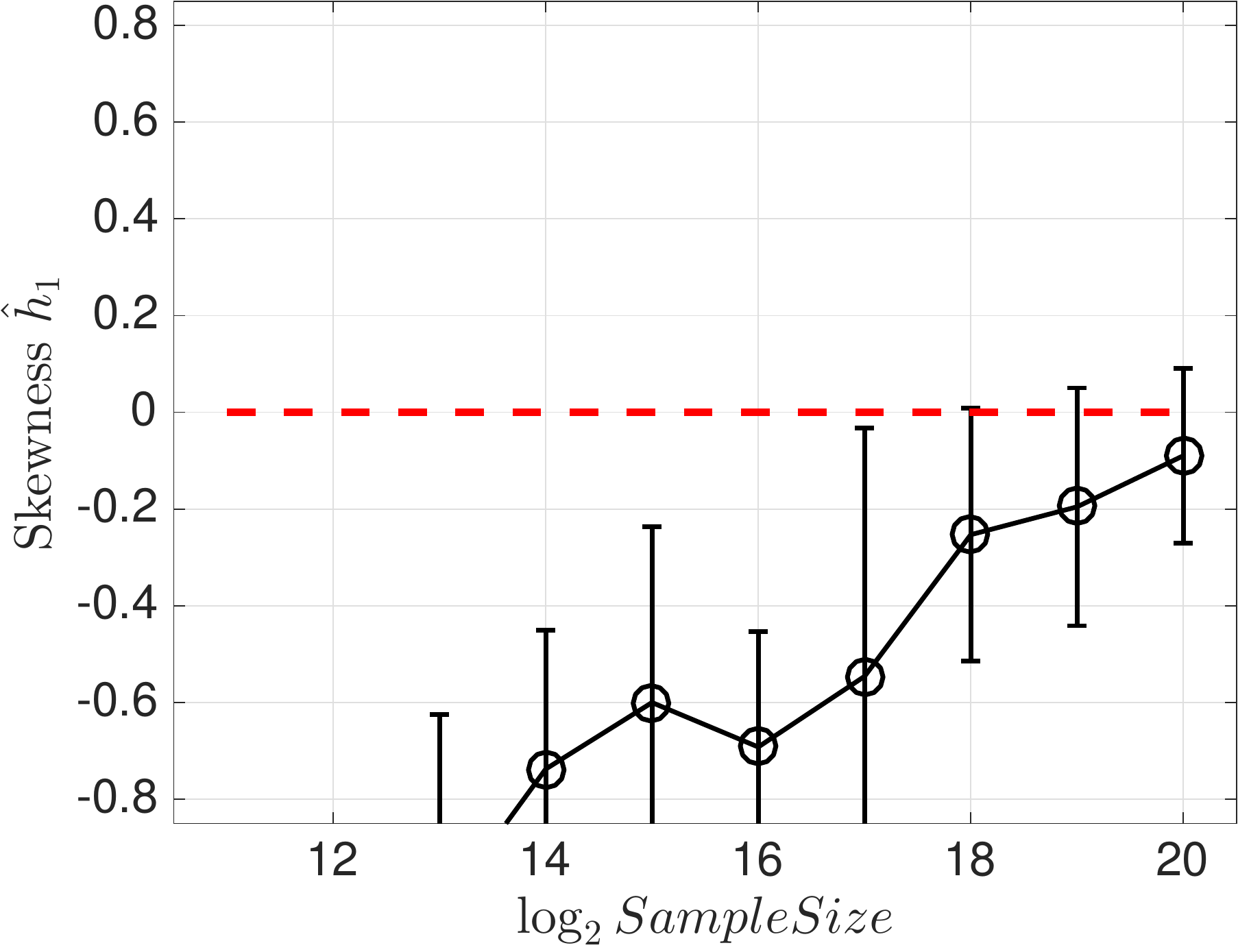}
\includegraphics[width=0.17\linewidth]{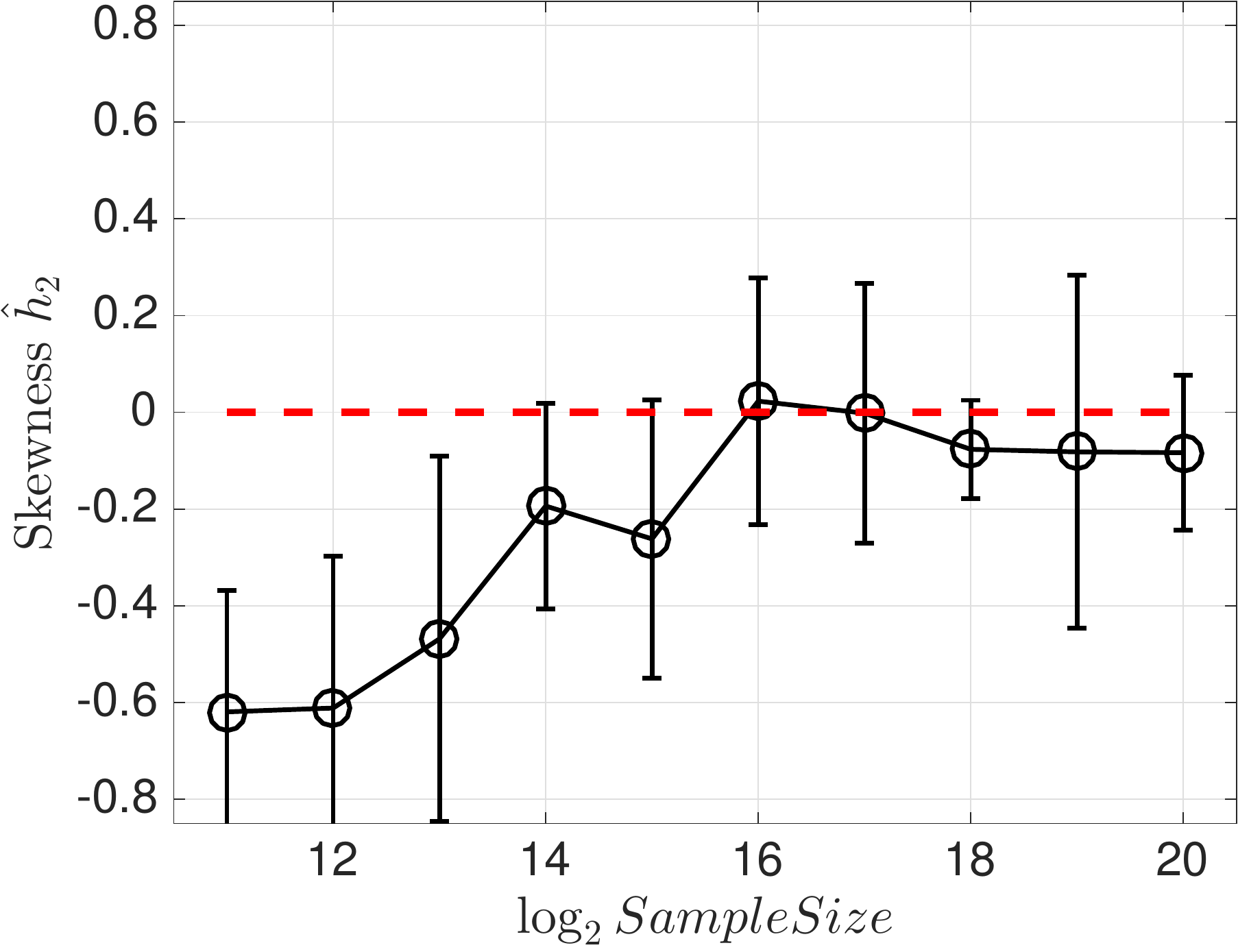}
\includegraphics[width=0.17\linewidth]{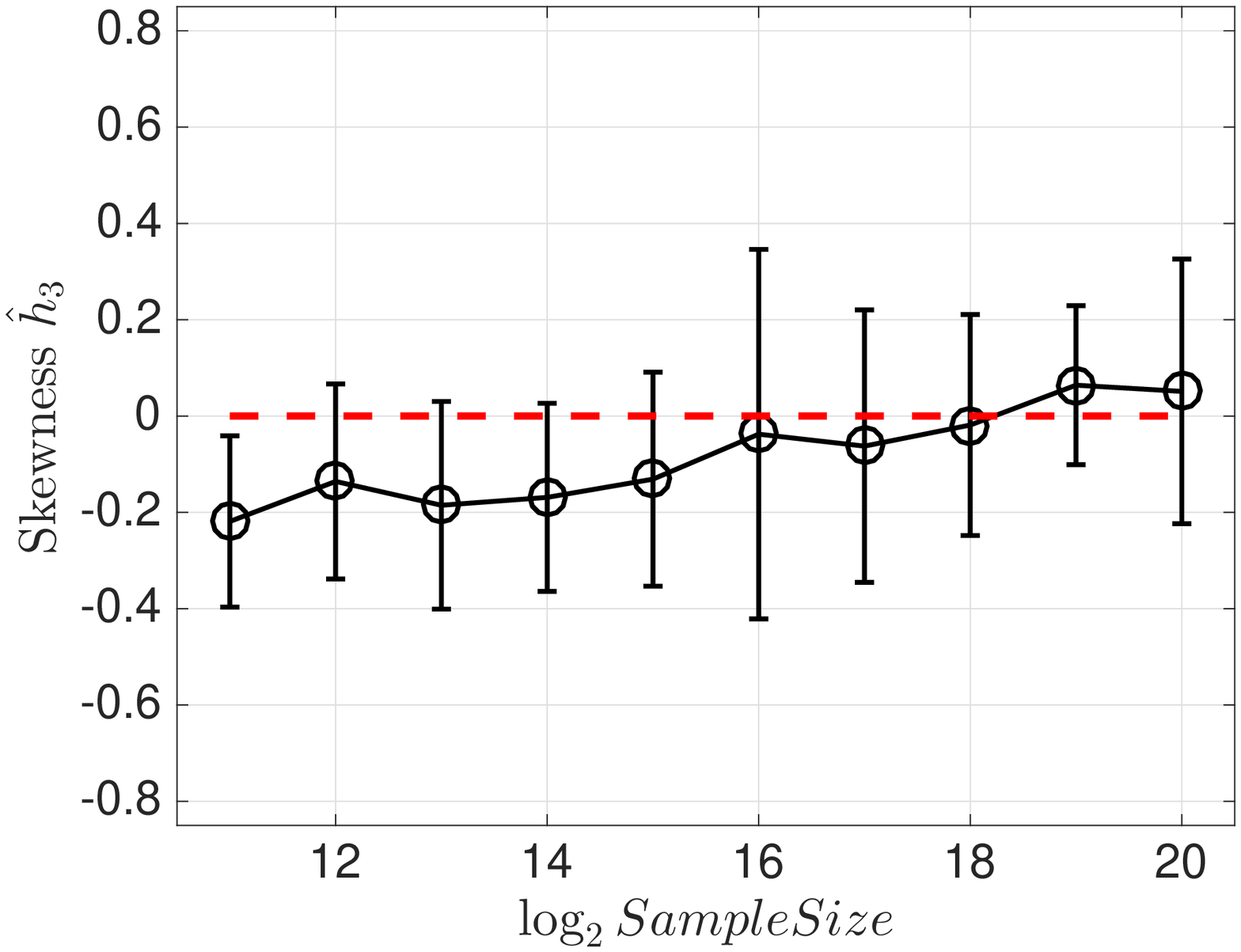}
\includegraphics[width=0.17\linewidth]{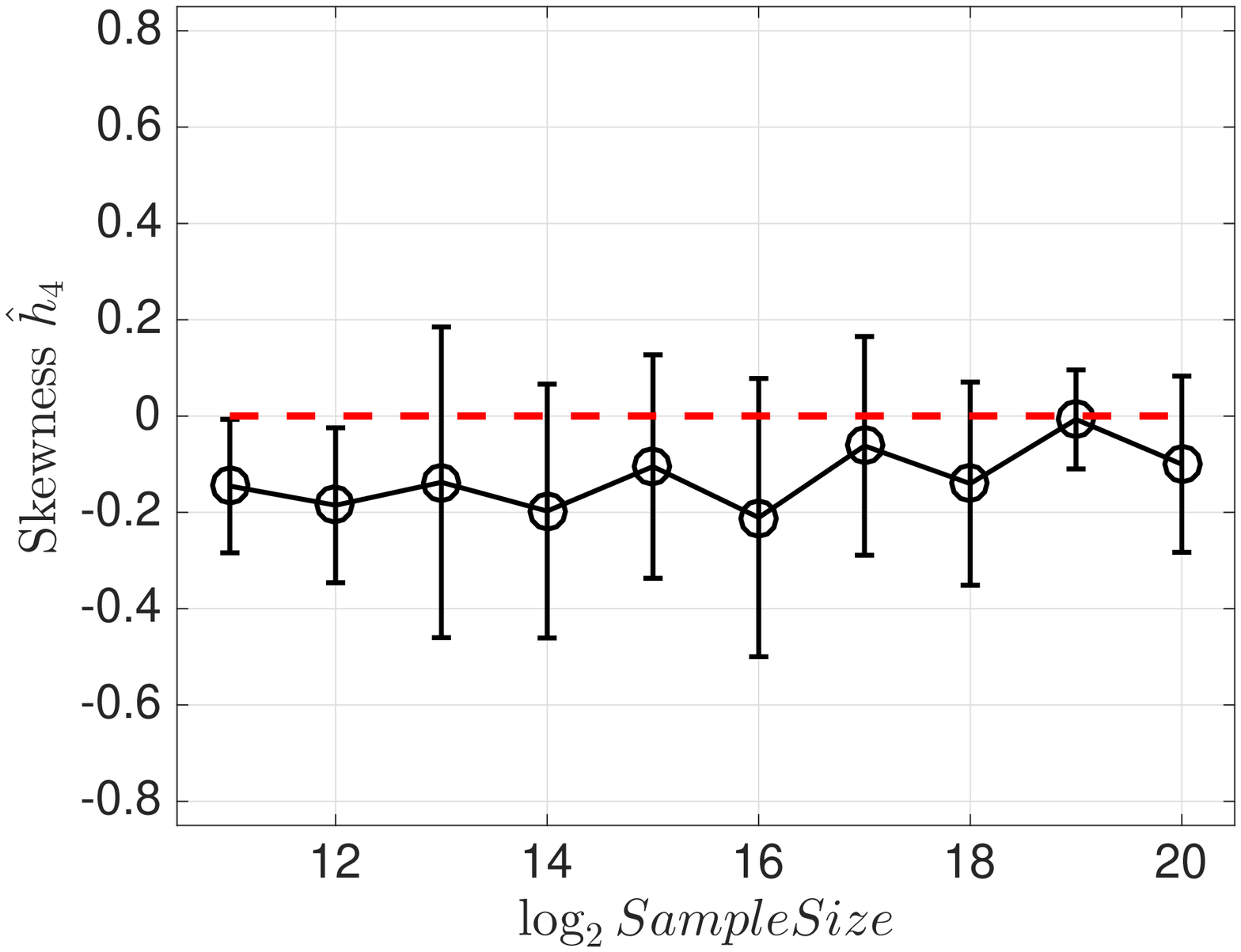}
\includegraphics[width=0.17\linewidth]{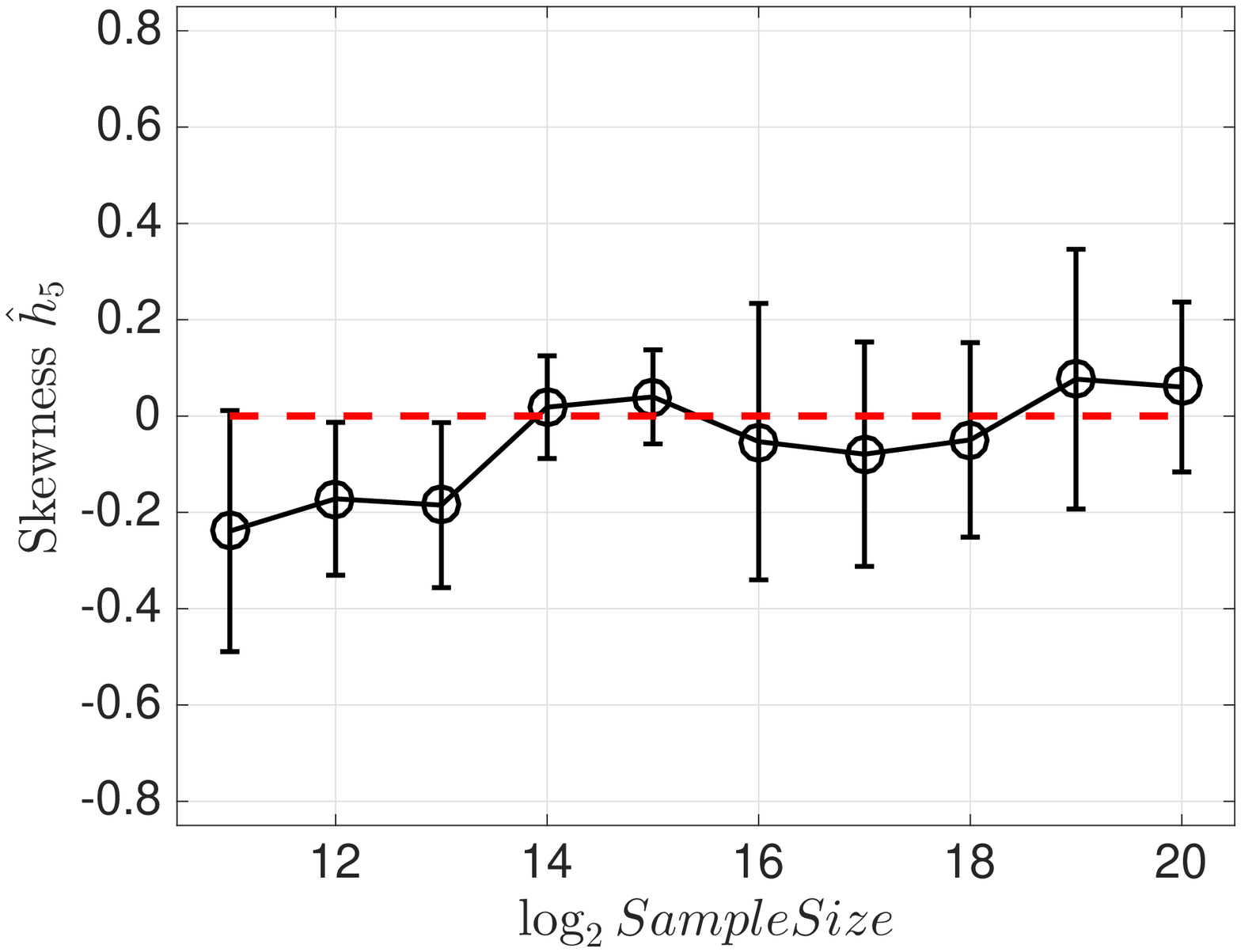}
\includegraphics[width=0.17\linewidth]{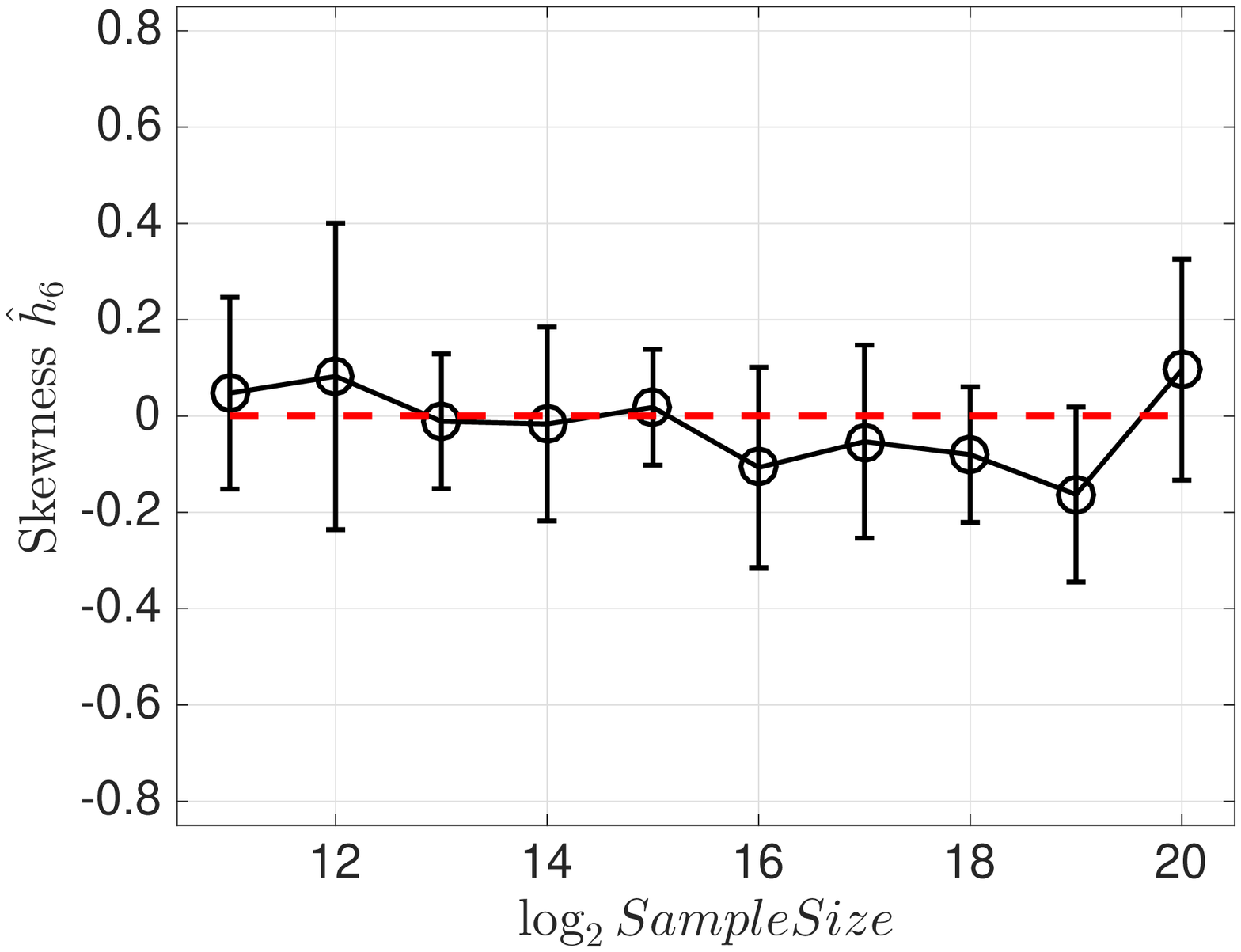}
}
\centerline{
\includegraphics[width=0.17\linewidth]{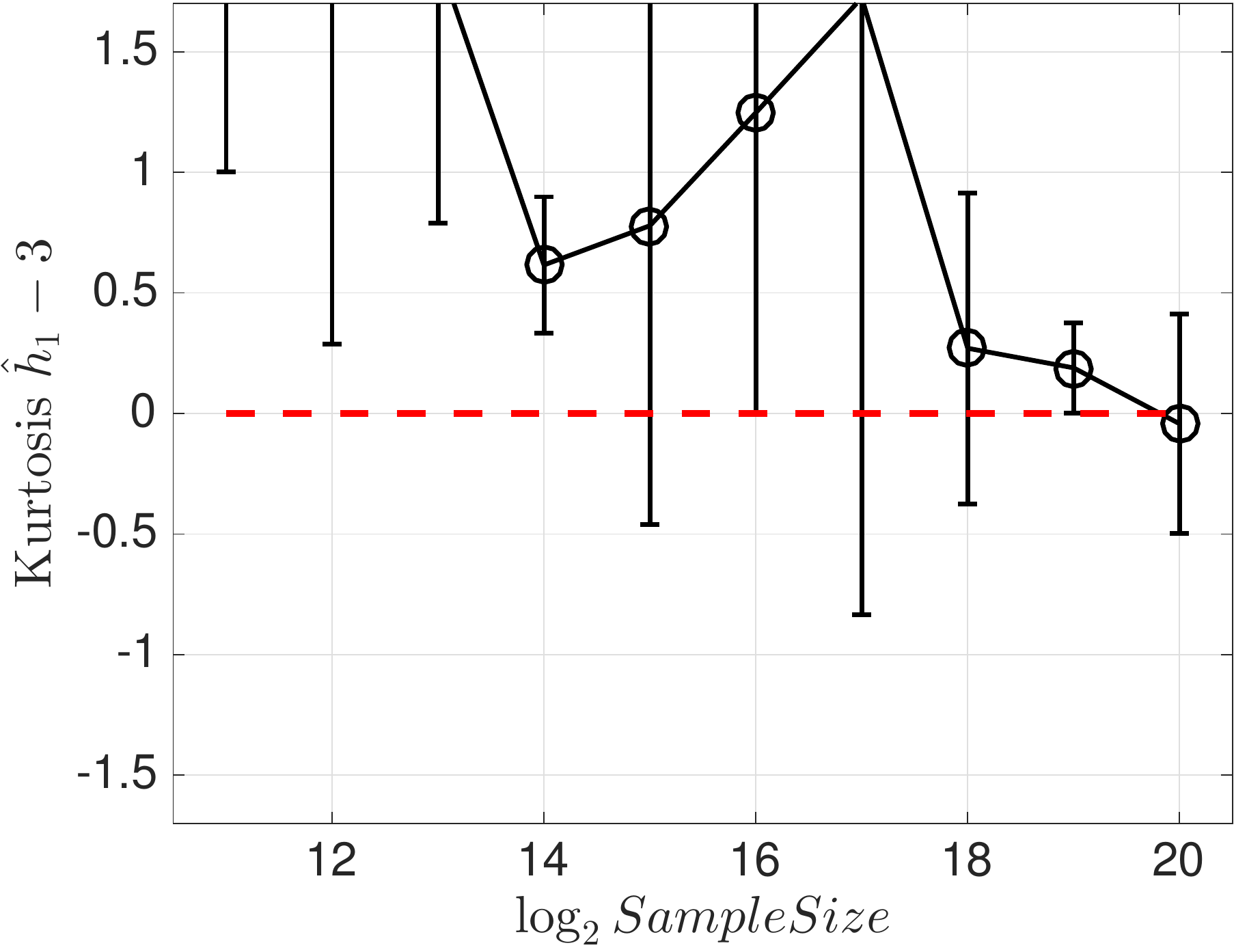}
\includegraphics[width=0.17\linewidth]{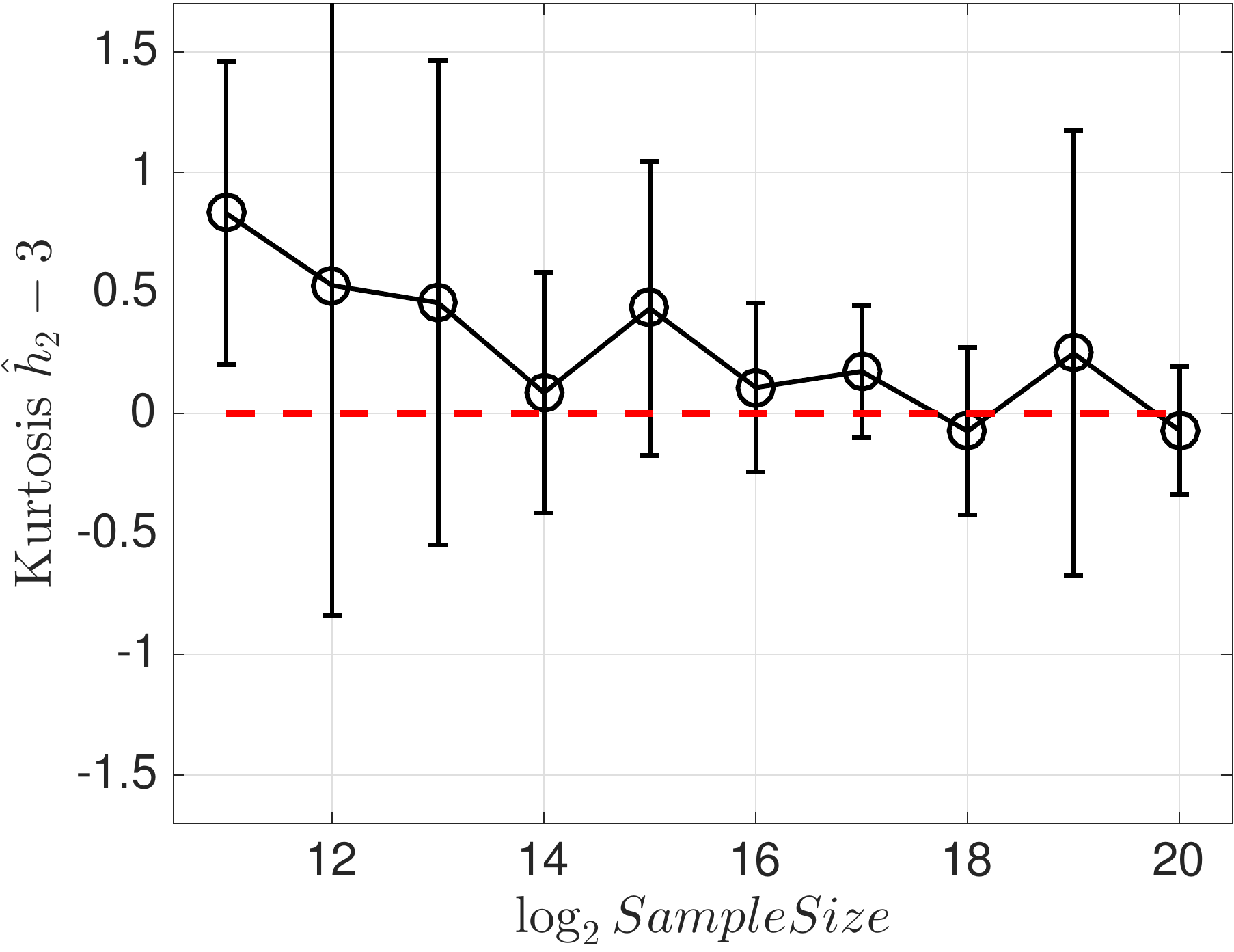}
\includegraphics[width=0.17\linewidth]{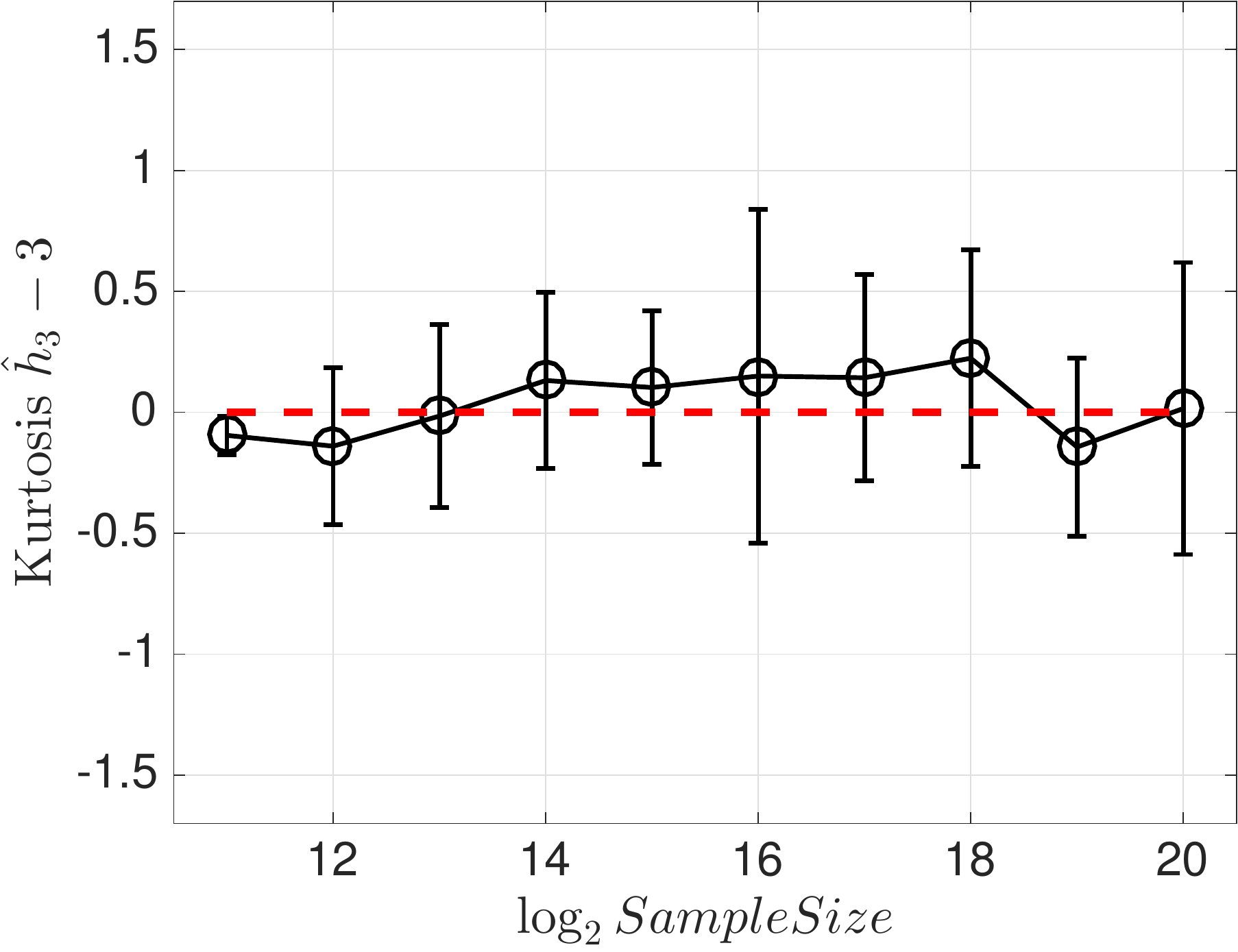}
\includegraphics[width=0.17\linewidth]{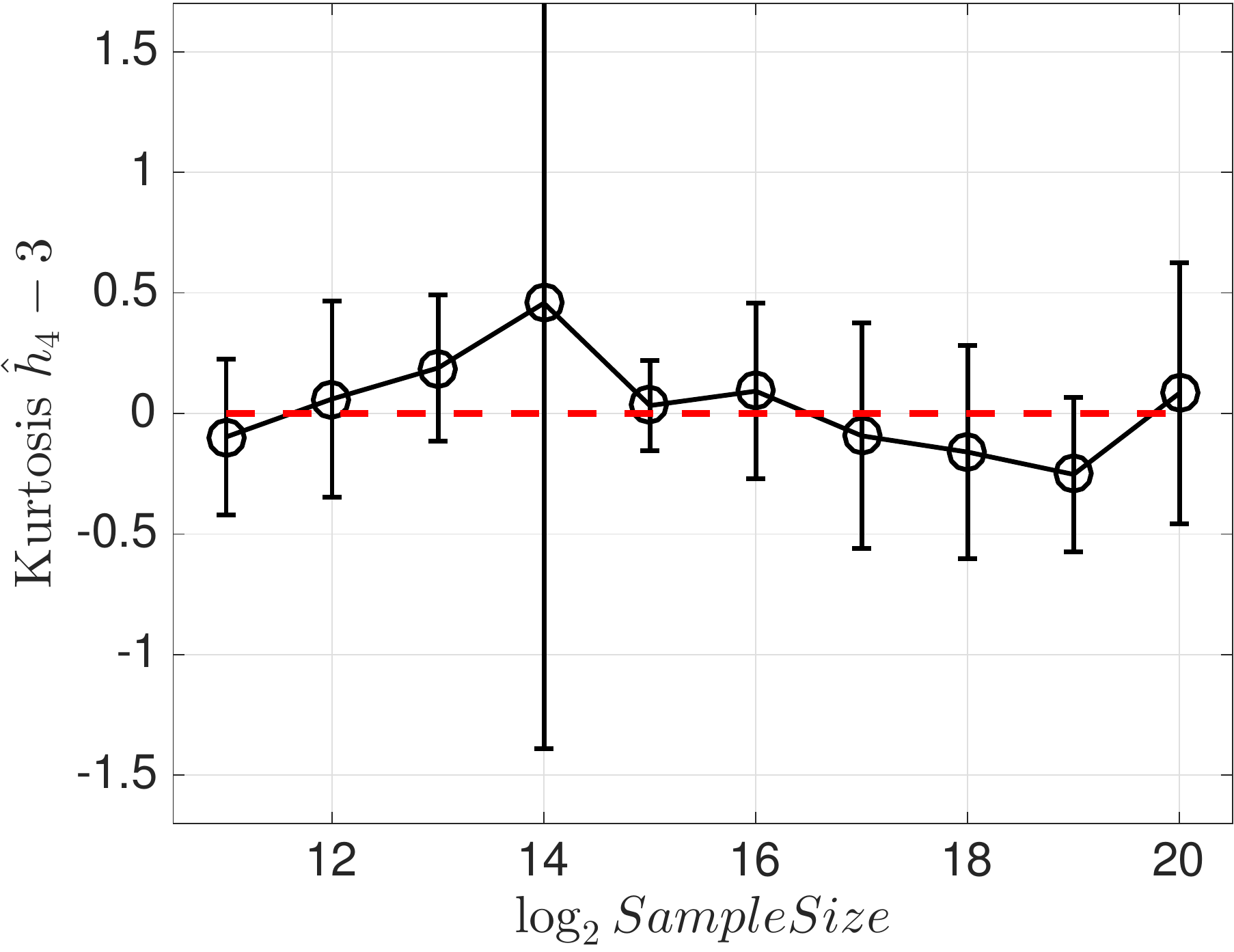}
\includegraphics[width=0.17\linewidth]{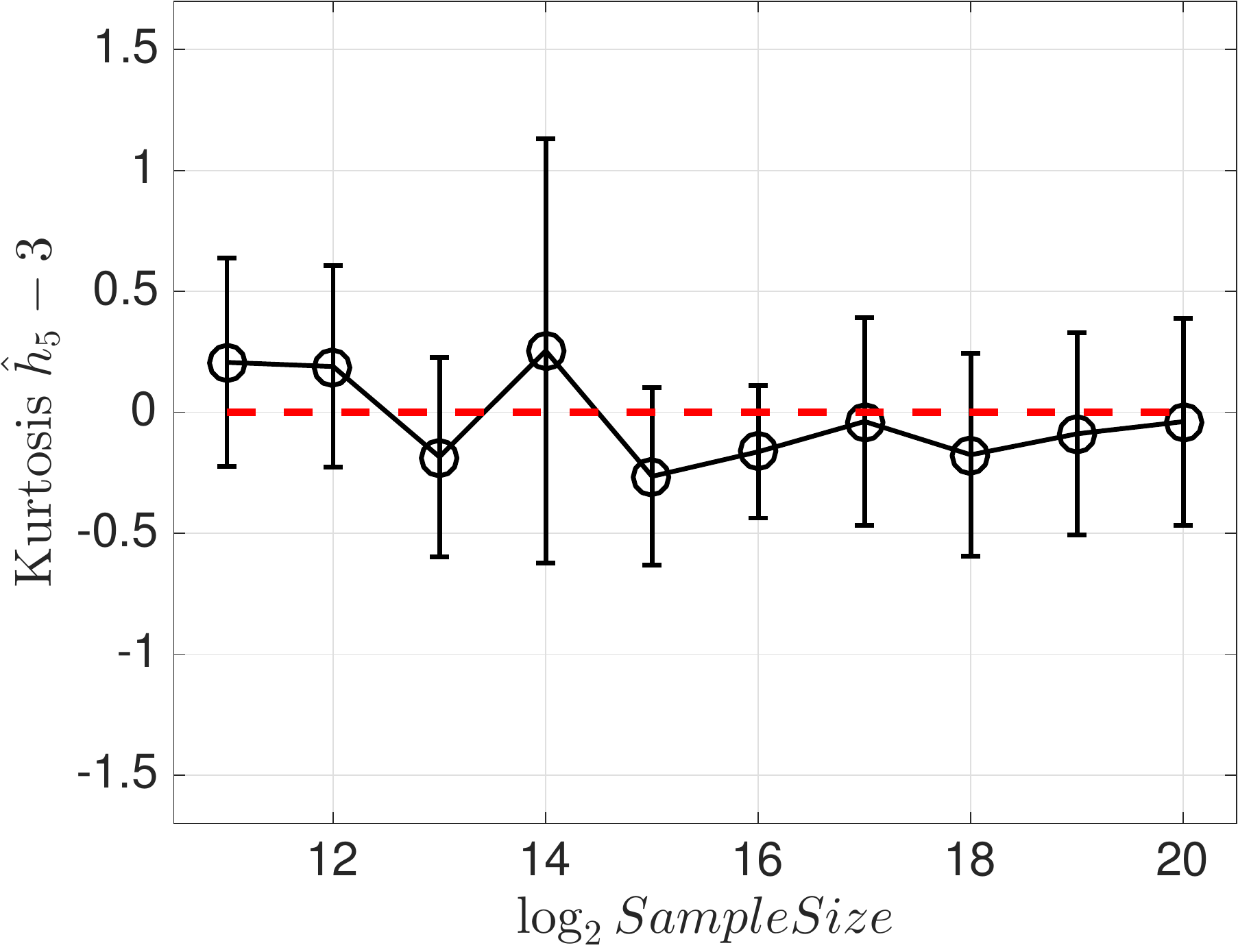}
\includegraphics[width=0.17\linewidth]{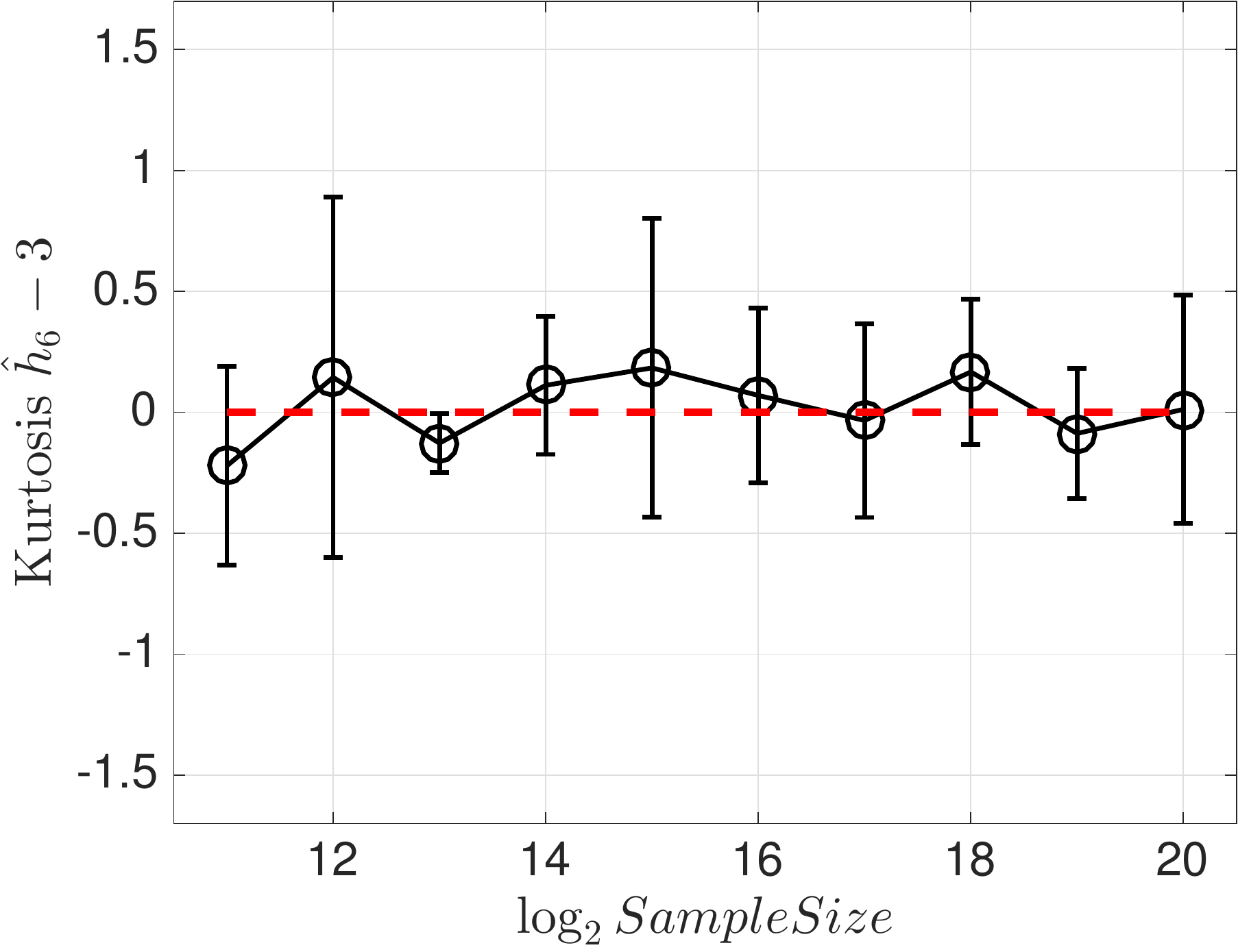}
}
\caption{\label{figc} {\bf Asymptotic normality.} Skewness (top) and kurtosis (bottom) for $\widehat{h}_q$ as functions of the sample size (bootstrapped confidence intervals).}
\end{figure}

\noindent {\bf Scaling range selection for estimation.} In our Monte Carlo studies, the log-regression octave range ($j_1, j_2$) was set a priori.
The choice of octaves $j$ involved in the estimation of Hurst eigenvalues is a way of balancing the bias-variance trade-off. 
On one hand, a large $j_1$ leads to a small bias. However, given the small number of sum terms in the sample wavelet variances \eqref{e:W(2j)}, it also results in a large estimation variance. On the other hand, a small $j_1$ reduces the variance at the price of increased bias. Monte Carlo studies not reported show that small values of $j_1$ lead to an overall better performance in terms of mean squared error, hence the choice $j_1=6$ in the experiments reported above.
The choice of optimal scaling ranges (which may depend on the rank of the Hurst eigenvalue) is a topic for future investigation.


\begin{figure}[thb]
\centerline{
\includegraphics[width=0.40\linewidth]{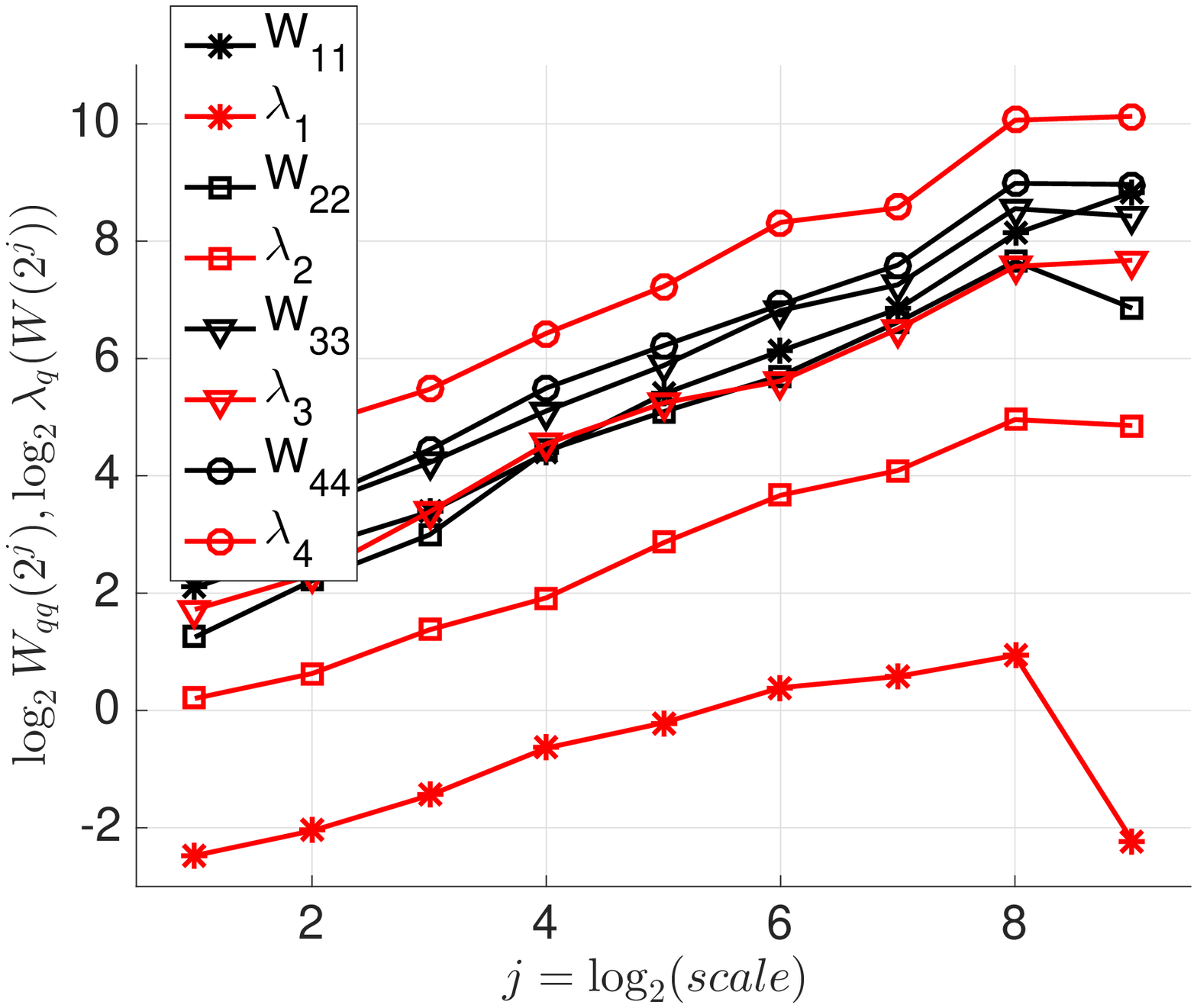}
\includegraphics[width=0.40\linewidth]{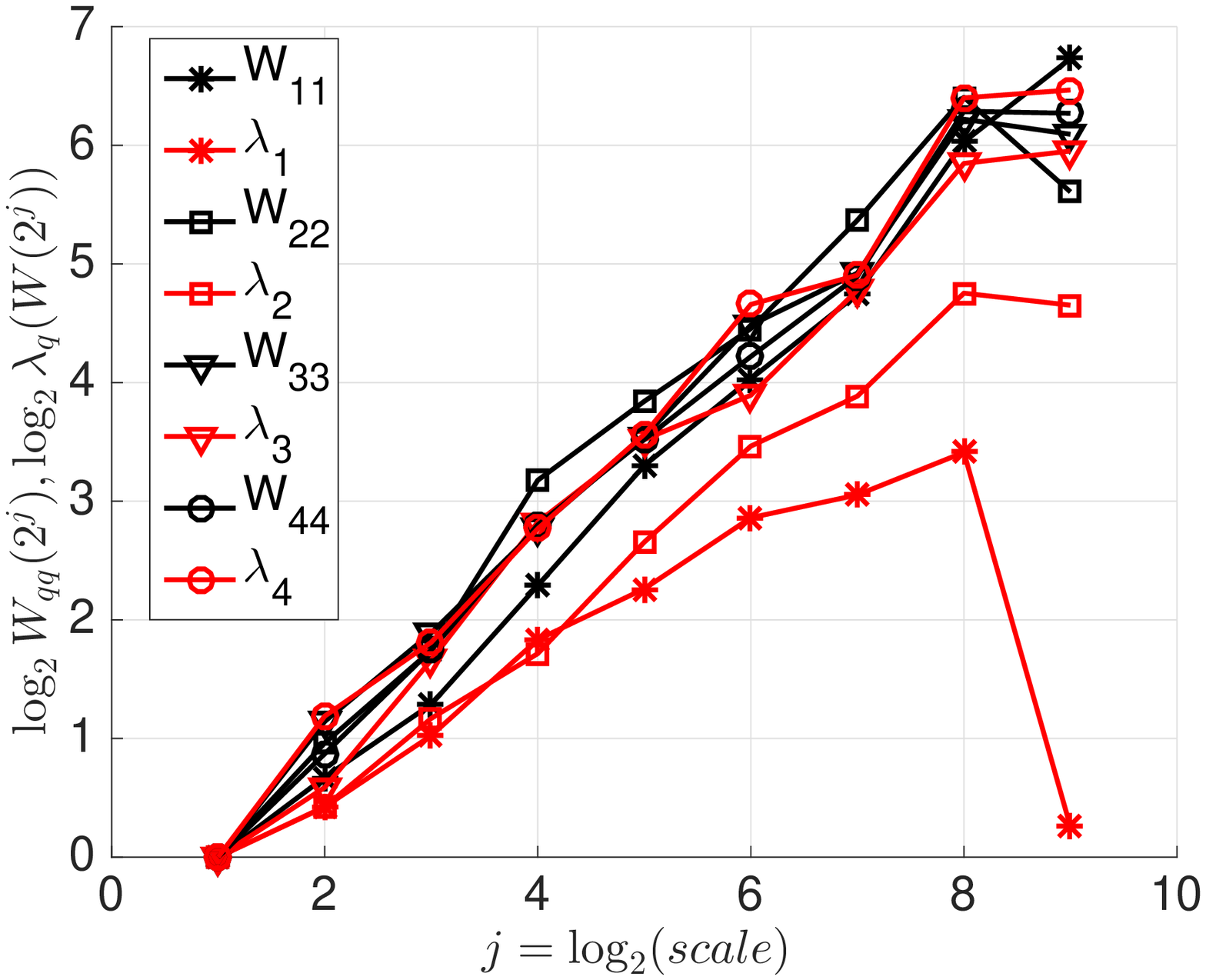}
}
\caption{\label{figInternet} {\bf 4-variate self-similarity analysis of Internet traffic (MAWI trace of June, 15th, 2007).} Functions $\log_2 W(2^j)_{qq}$ (solid black lines) and $\log_2 \lambda_q(2^j)$ (dashed red lines) for each of the 4 traffic components. On the right plot, the 8 functions were set to $0$ at scale $ 2^1$, with scale $2^0=1$ corresponding to $0.25$s. The functions $\log_2 W(2^j)_{qq}$ display linear behavior, which indicates self-similarity. They also have similar slopes, leading to roughly equal $h$ estimates for all 4 times series. By contrast, the functions $\log_2 \lambda_q(2^j)$ also show linear behavior, but with different Hurst eigenvalues. This is evidence of the rich multivariate structure of Internet traffic (see Table \ref{t:hurst} for Hurst eigenvalue estimates).}
\end{figure}

\section{Internet traffic modeling}\label{s:application}

The statistical modeling of Internet traffic is a central task in traffic engineering for the purposes of network design, management, control, security and pricing. Nevertheless, the data has always been modeled as a collection of univariate time series. In this section, we carry out the first study of multivariate self-similarity in Internet traffic data. We use OFBM as a baseline model for (second order) multivariate scaling properties, in the same way that FBM has been applied in the univariate context.

Empirical computer network traffic analysis started in the 1990s and hence can be considered a relatively new scientific field. Yet, the striking properties of Internet traffic data were revealed from the beginning. Standard models of traffic include a Poisson process with independent inter-arrival times or short range (exponentially decaying) autocorrelation structures. Instead, collected data was found to be characterized by significant \emph{burstiness} (strong irregularity over time) as well as slow, power law correlation decay (see Leland et al.\ \cite{leland:taqqu:willinger:wilson:1994}, Paxson and Floyd \cite{paxson:floyd:1995}, Erramilli et al.\ \cite{erramilli:narayan:willinger:1996}, Willinger et al.\ \cite{willinger:taqqu:sherman:wilson:1997}, Abry and Veitch \cite{AbryVeitch98}, Park and Willinger \cite{park:willinger:2000}, Erramilli et al.\ \cite{erramilli:roughan:veitch:willinger:2002}). It was soon recognized that the latter phenomenon, referred to as asymptotic self-similarity or long range dependence (LRD; Beran \cite{beran:1994}), had strong implications for network management due to its dramatic impact on queuing performance (see Norros \cite{norros:1994}, Boxma and Dumas \cite{boxma:dumas:1997}, Boxma and Cohen \cite{boxma:cohen:2000}). This lead to substantial research efforts in the last 20 years (see Willinger et al.\ \cite{willinger:taqqu:erramilli:1996}, Willinger et al.\ \cite{willinger2002scaling} and Fontugne et al.\ \cite{fontugne:abry:fukuda:veitch:cho:borgnat:wendt:2017} for reviews and references therein for details).

Self-similarity in Internet traffic has been widely investigated, but it remains controversial and a number of issues are still open. The data is often modeled in terms of aggregate time series. The latter consist of either IP (Internet Protocol) packet or byte counts on a given link, at a given time resolution $\Delta$. It has long been debated whether self-similarity is rather a property of the packet or byte count time series. Another question is whether traffic should be analyzed globally, with traffic traveling in both directions of the link, or if it should be split into directional traffic. In Dewaele et al.\ \cite{Dewaele2007} and Borgnat et al.\ \cite{borgnat:infocom2009}, these issues are analyzed and commented on in light of self-similarity. In this section, we consider a 4-variate setting, obtained as byte and packet counts, for each direction of the link.

The MAWI archive (Cho et al.\ \cite{cho:mitsuya:kato:2000}) is an ongoing collection of Internet traffic traces, captured on a high-speed, high-capacity backbone that mostly connects Japanese academic institutions to the USA. Anonymized traces are made publicly available at \texttt{http://mawi.wide.ad.jp/mawi/} and \texttt{http://mawi.wide.ad.jp/}, and several of them were kindly prepared for analysis and made available by the authors of Mazel et al.\ \cite{mazel:fontugne:fukuda:2014}. The data consists of 15 minute recordings, collected everyday at 2pm Japanese time.

It is well known in the field of Internet analysis that traffic is constantly affected by the emergence of anomalies. The latter pose significant hurdles to robust and meaningful statistical modeling of traffic. To tackle this issue, the technique of \textit{random projections } was developed. It consists of splitting each traffic series into a collection of subtraces. It has been reported that the median applied to the independent analysis of these subtraces is a robust statistical description of background (anomaly-free) traffic. This is thoroughly documented in Dewaele et al.\ \cite{Dewaele2007}, Borgnat et al.\ \cite{borgnat:infocom2009} and Fontugne et al.\ \cite{fontugne:abry:fukuda:veitch:cho:borgnat:wendt:2017}.

In this work, the random projection procedure yields 16 different subtraces. For each subtrace, the 4 time series consist of byte and packet counts in each direction (Japan to USA and USA to Japan), aggregated at the reference scale $\Delta = 0.25$s.

We analyze the data both by means of univariate-like and multivariate methodologies, based upon, respectively, the main diagonal entries $\log_2 W(2^j)_{\cdot \cdot}$ and the log-eigenvalue functions $\log_2 \lambda_{\bullet}(2^j)$. The median of each function $\log_2 W(2^j)_{qq}$ and $\log_2 \lambda_{q}(2^j)$, $q = 1,2,3,4$, is taken across subtraces to generate a characterization of self-similarity in Internet traces.

Examples of such functions are shown in Figure \ref{figInternet}, left panel. The functions $\log_2 W(2^j)$ clearly display linear behavior, hence indicating self-similarity. They are, however, nearly identical, with similar slopes. Incorrectly, this leads to the conclusion that the 4 times series are characterized by the same Hurst exponent (cf.\ Table~\ref{t:hurst}, top row).

Multivariate analysis also confirms self-similarity by means of the linear behavior of the functions $\log_2 \lambda_{\bullet}(2^j)$. However, the slopes clearly differ, which is evidence for the presence of different Hurst eigenvalues for the 4-variate data (cf.\ Table~\ref{t:hurst}, bottom row). This reveals the rich character of multivariate self-similarity in Internet traffic.

This finding is important in several ways. First, it complements 20 years of self-similarity analysis in Internet traffic and significantly enhances and renews it. Second, multivariate self-similarity modeling may permit revisiting several traffic engineering issues. Notably, it may underpin the construction of new anomaly detection schemes that will fruitfully complement those already available (see Mazel et al.\ \cite{mazel:fontugne:fukuda:2014}).

Results are reported here for one day traces, but equivalent conclusions can be drawn from numerous other traces in the MAWI repository. A longitudinal large-scale study is currently being conducted in collaboration with the teams managing the MAWI repository, aiming both at multivariate self-similarity characterization and at exploring its potential interest in anomaly detection.

\begin{table}[!h]
\begin{center}
\begin{tabular}{lrrrr}\hline
                                   & $\widehat h_1$&$\widehat h_2$&$\widehat h_3$&$\widehat h_4$ \\  \hline  \hline
{univariate-like} &0.85&0.86&0.86 & 0.90\\ \hline
{multivariate} &0.51&0.69&0.82 & 0.86\\ \hline \hline
\end{tabular}
\caption{\label{t:hurst}{\bf Univariate-like versus multivariate self-similarity analysis of Internet traffic.} Estimated Hurst eigenvalues by univariate-like (top row) and multivariate analysis (bottom row), based on the functions $\log_2 W(2^j)_{qq}$ and $\log_2 \lambda_q(2^j)$, respectively, for $q = 1,2,3,4$.}
\end{center}
\end{table}

\section{Conclusion}

In this paper, we construct the first joint estimator of the real parts of the Hurst eigenvalues of $n$-variate OFBM. The procedure consists of a wavelet regression on the log-eigenvalues of the sample wavelet spectrum. The estimator is shown to be consistent for any time reversible OFBM and, under stronger assumptions, also asymptotically normal starting from either continuous or discrete time measurements. Simulation studies establish the finite sample effectiveness of the methodology in terms of bias, mean squared error and asymptotic normality, and illustrate its benefits compared to univariate-like (entrywise) analysis. An application to 4-variate time series of Internet traffic data from the MAWI archive turned up evidence of multivariate self-similarity. Future work includes $(i)$ the quantification of confidence intervals and optimal regression procedures in practice; $(ii)$ the construction of methodology for instances where Hurst eigenvalues display multiplicity strictly between 1 and $n$; $(iii)$ applications in anomaly detection in Internet traffic. In the near future, a {\sc Matlab} toolbox for the estimators proposed in this paper will be made publicly available.

\appendix

\section{Proofs}

In the proofs, whenever convenient we write $a$ instead of $a(\nu)$.

\subsection{Consistency of wavelet log-eigenvalues}

To show Theorem \ref{t:asympt_normality_lambda2}, recall that the Courant-Fischer principle provides a variational characterization of the eigenvalues of a matrix $M \in {\mathcal H}(n,\bbR)$. In other words, it states that, for $q = 1,\hdots,n$,
\begin{equation}\label{e:Courant-Fischer}
\lambda_{q}(M) = \inf_{{\mathcal U}_q} \sup_{u \in {\mathcal U}_q \cap S^{n-1}_{\bbC}}u^*Mu= \sup_{{\mathcal U}_{n-q+1}}\inf_{u \in {\mathcal U}_{n-q+1} \cap S^{n-1}_{\bbC}} u^*Mu,
\end{equation}
where ${\mathcal U}_q$ is an $q$-dimensional subspace of $\bbC^n$ (e.g., Horn and Johnson \cite{horn:johnson:2012}, chapter 4).\\

\noindent {\sc Proof of Theorem \ref{t:consistency}}: The limits \eqref{e:log_lambdaE/2_log_a(n)_distinct Re} are a direct consequence of \eqref{e:log_lambdaE/2_log_a(n)}. We will only show the first limit in \eqref{e:log_lambdaE/2_log_a(n)}, since the second one can be proved by a similar and slightly simpler argument.

We first lay out a few facts that will be used throughout the proof. Note that the nonsingularity of $P$ (see \eqref{e:H_Jordan}) implies that
\begin{equation}\label{e:C1=<|Pu|2=<C2}
C_1 \leq \|v\|^2 = \|Pu\|^2 \leq C_2, \quad u \in S^{n-1}_{\bbC}.
\end{equation}
Under conditions \eqref{e:eigen-assumption}, \eqref{e:full-rank} and \eqref{e:time_revers}, by operator self-similarity the sample wavelet spectrum satisfies the operator scaling relation
\begin{equation}\label{e:W(a(nu)2^j)_is_o.s.s.}
W_a(a2^j) \stackrel{d}= a^{H}W_a(2^j) a^{H^*}
\end{equation}
for $W_a(2^j)$ as in \eqref{e:Wa(a2^j)} (c.f.\ \eqref{e:EW(a(nu)2j)_in_the_case_blindsourcing}). Now define the set $E_{\delta_1,\delta_2} = \{\omega: \delta_1 \leq \lambda_1(W_a(2^j)) \leq \lambda_n(W_a(2^j))\leq \delta_2\}$, $0 < \delta_1 \leq \delta_2$. Note that, by Theorem \ref{t:asymptotic_normality_wavecoef_fixed_scales}, $P(E_{\delta_1,\delta_2})\rightarrow 1$, $\nu \rightarrow \infty$, for some pair $0 < \delta_1 \leq \delta_2$. So, for any small $\varepsilon > 0$,
\begin{equation}\label{e:P(Edelta)->1}
1 - P(E_{\delta_1,\delta_2}) \leq \varepsilon, \quad \nu \geq \nu_0,
\end{equation}
for some $\nu_0 \in \bbN$. For any $u \in S^{n-1}_{\bbC}$, by Lemma \ref{l:S_1=<S_2} applied to $S_1 = \delta_1 a^{H}a^{H^*}$, $S_2 = a^{H}W_a(2^j)a^{H^*}$ and $S_1 = a^{H}W_a(2^j)a^{H^*}$, $S_2 = \delta_2 a^{H}a^{H^*}$,
\begin{equation}\label{e:lambdal(a(n)HW(2j)a(n)H*)_bounds}
\delta_1\hspace{1mm} \lambda_{q}(a^{H} a^{H^*}) \leq \lambda_{q}(a^{H}W_a(2^j) a^{H^*}) \leq \delta_2 \hspace{1mm}\lambda_{q}(a^{H}a^{H^*}), \quad q = 1,\hdots, n,
\end{equation}
for $\omega \in E_{\delta_1,\delta_2}$. Recall that
$$
a^{H}a^{H^*} = P \textnormal{diag}(a^{J_{h_1}},\hdots,a^{J_{h_{n'}}})(P^*P)^{-1}\textnormal{diag}(a^{J_{h_1}},\hdots,a^{J_{h_{n'}}}) P^*.
$$
and set $C'_1 = \lambda_1((P^*P)^{-1})$, $C'_2 = \lambda_n((P^*P)^{-1})$. Now consider Lemma \ref{l:S_1=<S_2} applied to
$$
S_1 = C'_1 P \textnormal{diag}(a^{J_{h_1}}a^{J^*_{h_1}},\hdots,a^{J_{h_{n'}}}a^{J^*_{h_{n'}}})P^*, \quad S_2 = a^{H}a^{H^*},
$$
and
$$
S_1 = a^{H}a^{H^*}, \quad S_2 = C'_2 P \textnormal{diag}(a^{J_{h_1}}a^{J^*_{h_1}},\hdots,a^{J_{h_{n'}}}a^{J^*_{h_{n'}}})P^*.
$$
We obtain the double bound
$$
C'_1 \hspace{1mm}\lambda_{q} \Big(P \textnormal{diag}(a^{J_{h_1}}a^{J^*_{h_1}},\hdots,a^{J_{h_{n'}}}a^{J^*_{h_{n'}}})P^* \Big)  \leq \lambda_{q}(a^{H}a^{H^*})
$$
\begin{equation}\label{e:lambdal(a(n)Ha(n)H*)_bounds}
\leq C'_2 \hspace{1mm}\lambda_{q} \Big(P \textnormal{diag}(a^{J_{h_1}}a^{J^*_{h_1}},\hdots,a^{J_{h_{n'}}}a^{J^*_{h_{n'}}})P^* \Big), \quad q = 1,\hdots,n.
\end{equation}
However, in view of \eqref{e:C1=<|Pu|2=<C2}, we can write
$$
C''_1 \inf_{{\mathcal U}_{q}} \sup_{u \in {\mathcal U}_{q} \cap S^{n-1}_{\bbC}}\Big\{\frac{u^*P}{\|u^*P\|} \textnormal{diag}(a^{J_{h_1}}a^{J^*_{h_1}},\hdots,a^{J_{h_{n'}}}a^{J^*_{h_{n'}}})\frac{P^*u}{\|u^*P\|} \Big\}
$$
$$
\leq \lambda_{q} \Big(P \textnormal{diag}(a^{J_{h_1}}a^{J^*_{h_1}},\hdots,a^{J_{h_{n'}}}a^{J^*_{h_{n'}}})P^* \Big)
$$
\begin{equation}\label{e:bound_lambdal_no_B}
\leq C''_2 \inf_{{\mathcal U}_{q}} \sup_{u \in {\mathcal U}_{q} \cap S^{n-1}_{\bbC}}\Big\{\frac{u^*P}{\|u^*P\|} \textnormal{diag}(a^{J_{h_1}}a^{J^*_{h_1}},\hdots,a^{J_{h_{n'}}}a^{J^*_{h_{n'}}})\frac{P^*u}{\|u^*P\|}\Big\}.
\end{equation}
By expressions \eqref{e:bound_lambdal_no_B} and \eqref{e:z^Jlambda},
$$
0 < C'''_1 a^{2 \Re h_{q'}}  \leq C''_1 \inf_{{\mathcal U}_{q}} \sup_{u \in {\mathcal U}_{q} \cap S^{n-1}_{\bbC}}\Big\{u^*\textnormal{diag}(a^{J_{h_1}}a^{J^*_{h_1}},\hdots,a^{J_{h_{n'}}}a^{J^*_{h_{n'}}})u\Big\}
$$
$$
\leq \lambda_{q}(P \textnormal{diag}(a^{J_{h_1}}a^{J^*_{h_1}},\hdots,a^{J_{h_{n'}}}a^{J^*_{h_{n'}}})P^*)
$$
$$
\leq C''_2 \inf_{{\mathcal U}_{q}} \sup_{u \in {\mathcal U}_{q} \cap S^{n-1}_{\bbC}}\Big\{u^* \textnormal{diag}(a^{J_{h_1}}a^{J^*_{h_1}},\hdots,a^{J_{h_{n'}}}a^{J^*_{h_{n'}}})u\Big\}
$$
\begin{equation}\label{e:lambdal(Pdiag(a(n)^h1,...,a(n)^hn)P*)}
= C'''_2 a^{2 \Re h_{q'}} (1 + o_P(\log^{2(n-1)} a(\nu)))
\end{equation}
for $\omega \in E_{\delta_1, \delta_2}$. In the first inequality in \eqref{e:lambdal(Pdiag(a(n)^h1,...,a(n)^hn)P*)}, we use the fact that $\det [ a(\nu)^{-2 \Re h_{\cdot}}a(\nu)^{J_{h_\cdot}}a(\nu)^{J^*_{h_\cdot}}] = 1$, i.e., $C'''_1$ is a strictly positive constant. The second equality in \eqref{e:lambdal(Pdiag(a(n)^h1,...,a(n)^hn)P*)} holds because no logarithmic term appears in one of the main diagonal blocks $a^{J_{h_{\cdot}}}a^{J^*_{h_{\cdot}}}$ with power greater than $2(n-1)$. By \eqref{e:lambdal(a(n)Ha(n)H*)_bounds}, \eqref{e:lambdal(Pdiag(a(n)^h1,...,a(n)^hn)P*)} and taking logs in \eqref{e:lambdal(a(n)HW(2j)a(n)H*)_bounds}, in view of \eqref{e:P(Edelta)->1} we arrive at the consistency relation in \eqref{e:log_lambdaE/2_log_a(n)}. $\Box$\\

\subsection{Asymptotic normality of wavelet log-eigenvalues}

Recall that, throughout this section, we work under the stronger assumption \eqref{e:H_h1<...<hn}. For notational simplicity, we write
\begin{equation}\label{e:Bhat,B}
\widehat{B}_a(2^j) = \Big( \hspace{1mm}\widehat{b}_{ii'}\hspace{1mm} \Big)_{i,i'= 1,\hdots,n}, \quad
B(2^j) = \Big( \hspace{1mm}b_{ii'}\hspace{1mm} \Big)_{i,i'= 1,\hdots,n}
\end{equation}
(see \eqref{e:EW(a(nu)2j)_in_the_case_blindsourcing}). We now establish Proposition \ref{p:convergence}. \\

\noindent {\sc Proof of Proposition \ref{p:convergence}:} In this proof, we will use the Courant-Fischer principle \eqref{e:Courant-Fischer} as applied to real spaces.

We start off with the eigenvalue $\lambda_n(W_a(a2^j))$, whose behavior is the easiest to characterize. From expression \eqref{e:W_a,2j}, note that
$$
0 \leq \frac{\lambda_n(W_a(a2^j))}{a^{2h_n}} = \sup_{u \in S^{n-1}} u^* \frac{W_a(a2^j)}{a^{2h_n}}u \stackrel{P}\rightarrow \sup_{u \in S^{n-1}} u^* P \textnormal{diag}(0,\hdots,0,1) B(2^j)\textnormal{diag}(0,\hdots,0,1)P^* u
$$
\begin{equation}\label{e:u3(nu)->u3}
= b_{nn} \sup_{u \in S^{n-1}} \langle p_{\cdot,n}, u \rangle^2 = b_{nn} \|p_{\cdot,n} \|^2 = b_{nn} > 0.
\end{equation}
Recall that $u_n(\nu) \in S^{n-1}$ denotes an eigenvector of $W_a(a2^j)$ associated with $\lambda_n(W_a(a2^j))$. For every $\nu \in \bbN$, $u_n(\nu) \in \textnormal{argmax}_{u \in S^{n-1}} u^* \frac{W_a(a2^j)}{a^{2h_n}}u$ a.s., and the largest eigenvalue of $\frac{W_a(a2^j)}{a^{2h_n}}$ is the only one not converging to zero. Therefore, \eqref{e:u1,u2,u3_conv} holds, and so does \eqref{e:lambdai/a^(2hi)_conv} for $q = n$ and
$$
\xi_n(2^j) = b_{nn} = b_{nn}(2^j).
$$
Moreover, for $q = n$, statements \eqref{e:u1,u2,u3_conv} and \eqref{e:inner*scaling=o(1)} are equivalent.

Turning to the remaining eigenvalues, in regard to $(iii)$, statement \eqref{e:u1,u2,u3_conv} is a consequence of \eqref{e:|<p3,u2>|a^(h3-h2)=O_P(1)} by considering $q = 1,2,\hdots,n-1$, sequentially. To show $(i)$, fix $q \leq n- 1$ and rewrite
$$
\frac{W_a(a2^j)}{a^{2h_{q}}}= P \Big( \widehat{b}_{i i'}a^{h_{i} - h_{q}}a^{h_{i'} - h_{q}} \Big)_{i,i' = 1,\hdots,n}P^*
$$
$$
= P\left(\begin{array}{ccccccc}
\widehat{b}_{11}a^{2(h_1 - h_{q})} &    \hdots     &   \widehat{b}_{1, q-1}a^{h_{1} - h_{q}}a^{h_{q-1} - h_{q}} &  \widehat{b}_{1, q}a^{h_{1} - h_{q}}  &    0    & \hdots &   0 \\
   \vdots                              &    \ddots     &      \vdots                                                            &          \vdots                                 &  \vdots & \ddots &   \vdots \\
                                     &                 &   \widehat{b}_{q-1, q-1}a^{2(h_{q-1} - h_{q})}           &  \widehat{b}_{q-1, q}a^{h_{q-1} - h_{q}}   &   0     & \hdots &   0 \\
   \widehat{b}_{1,q}a^{h_{1} - h_{q}}  &   \hdots   &     \widehat{b}_{q-1, q}a^{h_{q-1} - h_{q}}                &             \mathbf{0}                                &    0     & \hdots &   0 \\
0                                            &   \hdots   &     0                                                                   &              0                                &    0     & \hdots &   0 \\
\vdots                                        &   \ddots   &     \vdots                                                             &              \vdots                           &   \vdots     & \ddots &   \vdots \\
0                                            &   \hdots   &     0                                                                   &              0                                &    0     & \hdots &   0 \\
\end{array}\right)P^*
$$
$$
+ P \left(\begin{array}{ccccccc}
0     &  \hdots &      0       &            0                  & \widehat{b}_{1, q+1}a^{h_{1} - h_{q}}a^{h_{q+1} - h_{q}}          &     \hdots   & \widehat{b}_{1, n}a^{h_{1} - h_{q}}a^{h_{n} - h_{q}} \\
\vdots &  \ddots     &      \vdots  &           \vdots              &                 \vdots                                                     &     \ddots   &             \vdots                                          \\
0 &  \hdots   &      0       &            0                  & \widehat{b}_{q-1, q+1}a^{h_{q-1} - h_{q}}a^{h_{q+1} - h_{q}} &     \hdots   &  \widehat{b}_{q-1, n}a^{h_{q-1} - h_{q}}a^{h_{n} -h_{q}}\\
0 &  \hdots &   0      & \mathbf{\widehat{b}_{q q} }      & \widehat{b}_{q, q+1}a^{h_{q+1} - h_{q}}                         &    \hdots    & \widehat{b}_{q n} a^{h_n - h_{q}}                            \\
\bullet &  \hdots &   \bullet      & \bullet                       & \widehat{b}_{q +1, q+1} a^{2(h_{q+1} - h_{q})}                  &    \hdots    & \widehat{b}_{q +1, n} a^{h_{q +1} - h_{q}}a^{h_n -h_{q}}\\
\bullet &  \hdots &   \bullet      & \bullet                       & \vdots                                                                     &    \ddots    &     \vdots                                                    \\
\bullet &  \hdots &   \bullet      & \bullet                         & \bullet                                                                     &    \hdots    & \widehat{b}_{n n} a^{2(h_{n} - h_{q})}                  \\
\end{array}\right)P^*
$$
\begin{equation}\label{e:W/a^(2hi0)}
=: P\widehat{{\mathbf S}}_{\nu,q-1} P^* + P\widehat{{\mathbf T}}_{\nu,n -q +1} P^*.
\end{equation}
In \eqref{e:W/a^(2hi0)}, each $\bullet$ entry is generally not identically zero and can be obtained by symmetry, and in both matrices on the right-hand side of \eqref{e:W/a^(2hi0)}, entry $(q,q)$ appears in boldface for ease of visualization. By Weyl's inequality,
\begin{equation}\label{e:Weyl}
\lambda_{q}(P\widehat{{\mathbf S}}_{\nu,q-1} P^*+P\widehat{{\mathbf T}}_{\nu,n -q +1}  P^*) \leq \lambda_n(P\widehat{{\mathbf S}}_{\nu,q-1}P^*)+\lambda_{q}(P\widehat{{\mathbf T}}_{\nu,n -q +1}P^*)
\end{equation}
(Horn and Johnson \cite{horn:johnson:2012}, Theorem 4.3.1, p.\ 239). Since $P\widehat{{\mathbf S}}_{\nu,q-1}P^* \stackrel{P}\rightarrow \textbf{0}$, then \begin{equation}\label{e:lambda3(PAP*)->0_in_prob}
\lambda_n(P\widehat{{\mathbf S}}_{\nu,q-1} P^*) \stackrel{P}\rightarrow 0, \quad \nu \rightarrow \infty.
\end{equation}
%

Now consider the second term on the right-hand side of \eqref{e:Weyl}. Define the matrix
\begin{equation}\label{e:Uhat}
\widehat{{\mathbf U}}_{\nu,n -q +1} = \left(\begin{array}{ccccccc}
0     &  \hdots &      0       &            0                  &           0                                                   &     \hdots                     &                       0 \\
\vdots &  \ddots     &      \vdots  &           \vdots              &                 \vdots                                                     &     \ddots   &             \vdots          \\
0 &  \hdots   &      0       &            0                  &                         0                                         &     \hdots                    &                         0 \\
0 &  \hdots &   0      & \widehat{b}_{q q}       & \widehat{b}_{q, q+1}a^{h_{q+1} - h_{q}}                         &    \hdots    & \widehat{b}_{q n} a^{h_n - h_{q}}                            \\
0      &  \hdots &   0      & \bullet                       & \widehat{b}_{q +1, q+1} a^{2(h_{q+1} - h_{q})}                  &    \hdots    & \widehat{b}_{q +1, n} a^{h_{q +1} - h_{q}}a^{h_n -h_{q}}\\
\vdots &  \ddots &   \vdots     & \bullet                       & \vdots                                                                     &    \ddots    &     \vdots                                                    \\
0      &  \hdots &   0      & \bullet                         & \bullet                                                                     &    \hdots    & \widehat{b}_{n n} a^{2(h_{n} - h_{q})}                  \\
\end{array}\right).
\end{equation}
Let
\begin{equation}\label{e:u-bf(nu),u-bf'(nu)}
\textnormal{${\mathbf u}_{q}(\nu)$ and ${\mathbf u}'_{q}(\nu)$}
\end{equation}
be unit eigenvectors associated with $\lambda_{q}(P\widehat{{\mathbf T}}_{\nu,n -q +1} P^*)$ and $\lambda_{q}(P\widehat{{\mathbf U}}_{\nu,n -q +1} P^*)$, respectively. As a consequence of \eqref{e:|<p3,u2>|a^(h3-h2)=O_P(1)} in Lemma \ref{l:<p3,u2>a(nu)^(h_3-h_2)=OP(1)} applied to $({\mathbf W}_{\nu}/a^{2h_q} = )P\widehat{{\mathbf T}}_{\nu,n -q +1} P^*$ and $({\mathbf W}_{\nu}/a^{2h_q} = )P\widehat{{\mathbf U}}_{\nu,n -q +1} P^*$, for any $\gamma > 0$ there exists $\eta_{\gamma} > 0 $ such that
\begin{equation}\label{e:P(|<p3,u2>a^(h3-h2)|=<eta)>=1-xi}
P(C_{\gamma,\nu} ) \geq 1 - \gamma, \quad \nu \in \bbN,
\end{equation}
where
\begin{equation}\label{e:Cgamma,nu}
C_{\gamma,\nu} = \Big\{ \max_{i=q+1,\hdots,n}\Big(\max\{|\langle p_{\cdot,i}, {\mathbf u}_{q}(\nu)\rangle| a^{h_i - h_{q}}, |\langle p_{\cdot,i}, {\mathbf u}'_{q}(\nu)\rangle | a^{h_i - h_{q}} \}\Big)\leq \eta_{\gamma} \Big\}.
\end{equation}
Moreover,
$$
\widehat{b}_{i,q} \langle p_{\cdot,i}, {\mathbf u}_{q}(\nu)\rangle a^{h_i - h_{q}} \stackrel{P}\rightarrow 0, \quad i = 1,\hdots,q -1.
$$
Therefore, for some constant $C > 0$, with probability going to 1,
$$
\Big|\sum^{q-1}_{i= 1} \sum^{n}_{i'= q + 1}  \widehat{b}_{i i'}\langle p_{\cdot,i}, {\mathbf u}_{q}(\nu) \rangle a^{h_i-h_{q}}\langle p_{\cdot,i'},{\mathbf u}_{q}(\nu) \rangle a^{h_{i'}-h_{q}} \Big| \leq \frac{C}{a^{h_q - h_{q-1}}}.
$$
Hence, in the set $C_{\gamma,\nu}$,
$$
\lambda_{q}(P\widehat{{\mathbf T}}_{\nu,n- q + 1} P^*) = \sup_{{\mathcal U}_{n - q+1}}\inf_{u \in {\mathcal U}_{n-q+1} \cap S^{n-1}} u^* P \widehat{{\mathbf T}}_{\nu,n- q + 1} P^* u
$$
$$
= \sup_{{\mathcal U}_{n-q+1}}\inf_{\substack{u \in {\mathcal U}_{n-q+1} \cap S^{n-1} \\ \max_{i=q+1,\hdots,n} |\langle p_{\cdot,i}, u\rangle a^{h_i - h_{q}}| \leq \eta_{\gamma}}} u^* P \widehat{{\mathbf T}}_{\nu,n- q + 1} P^* u
$$
$$
= \sup_{{\mathcal U}_{n-q+1}}\inf_{\substack{u \in {\mathcal U}_{n-q+1} \cap S^{n-1} \\ \max_{i=q+1,\hdots,n} |\langle p_{\cdot,i}, u \rangle a^{h_i - h_{q}}| \leq \eta_{\gamma}} }
\Big[ \sum^{n}_{i=q}\widehat{b}_{ii}\langle p_{\cdot,i},u \rangle^2 a^{2(h_i-h_{q})}
$$
$$
+ 2 \sum^{q-1}_{i= 1} \sum^{n}_{i'= q + 1}  \widehat{b}_{i i'}\langle p_{\cdot,i}, u \rangle a^{h_i-h_{q}}\langle p_{\cdot,i'},u \rangle a^{h_{i'}-h_{q}}
+ 2 \sum_{q \leq i < i' \leq n }  \widehat{b}_{i i'}\langle p_{\cdot,i},u\rangle a^{h_i-h_{q}}\langle p_{\cdot,i'},u \rangle a^{h_{i'}-h_{q}} \Big]
$$
$$
\leq \sup_{{\mathcal U}_{n-q+1}}\inf_{\substack{u \in {\mathcal U}_{n-q+1} \cap S^{n-1} \\ \max_{i=q+1,\hdots,n} |\langle p_{\cdot,i}, u \rangle a^{h_i - h_{q}}| \leq \eta_{\gamma}} } \Big[ \sum^{n}_{i=q}\widehat{b}_{ii}\langle p_{\cdot,i},u \rangle^2 a^{2(h_i-h_{q})}
$$
$$
+ \frac{(C + o_P(1))}{a^{h_q - h_{q-1}}}
+ 2 \sum_{q \leq i < i' \leq n }  \widehat{b}_{i i'}\langle p_{\cdot,i},u \rangle a^{h_i-h_{q}}\langle p_{\cdot,i'},u \rangle a^{h_{i'}-h_{q}} \Big]
$$
\begin{equation}\label{e:lambda2(PTP*)_bound}
= \sup_{{\mathcal U}_{n-q+1}}\inf_{u \in {\mathcal U}_{n-q+1} \cap S^{n-1}} u^* P \widehat{{\mathbf U}}_{\nu,n - q + 1} P^* u  + \frac{(C + o_P(1))}{a^{h_q - h_{q-1}}},
\end{equation}
where the inequality holds for large enough $\nu$ and the last equality is a consequence of \eqref{e:Cgamma,nu}.

Turning to the matrix $P\widehat{{\mathbf U}}_{\nu,n - q + 1}P^* \in {\mathcal  H}_{\geq 0}(n,\bbR)$, it is clear that
$$
\Big\{v \in \bbR^n: v \in \{p_{\cdot,q},p_{\cdot,q+1}, \hdots,p_{\cdot,n}\}^{\perp} \Big\}
$$
is the real $(q-1)$-dimensional eigenspace associated with the zero eigenvalues of $P\widehat{{\mathbf U}}_{\nu,n - q + 1} P^*$, i.e., with $\lambda_i(P\widehat{{\mathbf U}}_{\nu,n - q + 1} P^*)$, $i = 1,\hdots,q - 1$. Therefore,
\begin{equation}\label{e:lambda2(PBP*)}
\lambda_{q}(P\widehat{{\mathbf U}}_{\nu,n - q + 1} P^*) = \inf_{u \in \textnormal{span}\{p_{\cdot,q},\hdots,p_{\cdot,n}\} \cap S^{n-1}}u^* P\widehat{{\mathbf U}}_{\nu,n - q + 1} P^* u.
\end{equation}
Let
$$
\widehat{{\mathbf x}}_{q,*}(2^j) = \widehat{{\mathbf x}}_{q,*} = (\widehat{x}_{q+1,*},\hdots,\widehat{x}_{n,*}), \quad {\mathbf x}_{q,*}(2^j) = {\mathbf x}_{q,*}
$$
be the global minima of the functions $\widehat{g}_{\nu,q,j}$ and $g_{q,j}$ as in \eqref{e:x*_sol} and \eqref{e:x*_sol_determ}, respectively. Consider a sequence of vectors
\begin{equation}\label{e:w(nu)_in_span(pi0,...,pn)}
w(\nu) \in \textnormal{span}\{p_{\cdot,q},\hdots,p_{\cdot,n}\}  \cap S^{n-1}
\end{equation}
such that
$$
\langle p_{\cdot,i},w(\nu)\rangle = \frac{\widehat{x}_{i,*}}{a^{h_i - h_{q}}} \in (-1,1), \quad i = q+1,\hdots,n,
$$
which is possible for large enough $\nu$. In particular, the distance between $w(\nu)$ and the subspace $\{p_{\cdot,q+1},\hdots,p_{\cdot,n}\}^{\perp}$ goes to zero. This implies that, without loss of generality, we can choose the sequence $w(\nu)$ so that
\begin{equation}
w(\nu) \stackrel{P}\rightarrow u_{q} \in \textnormal{span}\{p_{\cdot,q},\hdots,p_{\cdot,n}\} \cap \{p_{\cdot,q+1},\hdots,p_{\cdot,n}\}^{\perp} \cap S^{n-1},
\end{equation}
where $u_{q}$ is given by \eqref{e:u1,u2,u3_conv}. Let $\{u_{q}(\nu)\}_{\nu \in \bbN}$ be a sequence of eigenvectors (of $W_a(a2^j)$) as in \eqref{e:u1,u2,u3_conv}. Then, by \eqref{e:lambda2(PBP*)} and \eqref{e:w(nu)_in_span(pi0,...,pn)},
$$
\lambda_{q}(P\widehat{{\mathbf U}}_{\nu,n - q + 1} P^*) \leq w^*(\nu) P\widehat{{\mathbf U}}_{\nu,n - q + 1} P^* w(\nu)
$$
$$
= \sum^{n}_{i=q} \widehat{b}_{ii}\langle p_{\cdot,i},w(\nu)\rangle^2 a^{2(h_i - h_q)}
+ 2 \sum_{q \leq i < i' \leq n}\widehat{b}_{ii'} \langle p_{\cdot,i},w(\nu)\rangle a^{h_i - h_{q}}\langle p_{\cdot,i'},w(\nu)\rangle a^{h_{i'} - h_{q}}
$$
$$
=\widehat{b}_{q q}\langle p_{\cdot,q},u_{q}(\nu)\rangle^2 + \sum^{n}_{i=q+1} \widehat{b}_{ii}\langle p_{\cdot,i},w(\nu)\rangle^2 a^{2(h_i - h_{q})}
+ 2 \sum_{q +1\leq i \leq n}\widehat{b}_{q i} \langle p_{\cdot,q},u_{q}(\nu)\rangle \langle p_{\cdot,i},w(\nu)\rangle a^{h_{i} - h_{q}}
$$
$$
+ 2 \sum_{q +1\leq i < i' \leq n}\widehat{b}_{ii'} \langle p_{\cdot,i},w(\nu)\rangle a^{h_i - h_{q}}\langle p_{\cdot,i'},w(\nu)\rangle a^{h_{i'} - h_{q}}
$$
$$
+ \Big\{\widehat{b}_{q q}\Big[\langle p_{\cdot,q},w(\nu)\rangle^2 - \langle p_{\cdot,q},u_{q}(\nu)\rangle^2 \Big] + 2 \sum_{q +1\leq i \leq n}\widehat{b}_{q i} \Big[\langle p_{\cdot,q},w(\nu)\rangle- \langle p_{\cdot,q},u_{q}(\nu)\rangle \Big]\langle p_{\cdot,i},w(\nu)\rangle a^{h_{i} - h_{q}} \Big\}
$$
\begin{equation}\label{e:lambda2(P(A+B)P*)_upper_bound}
= \widehat{g}_{\nu,q,j} (\widehat{{\mathbf x}}_{q,*}) + o_P(1).
\end{equation}
On the other hand, since $\widehat{x}_{q,*}$ is the global minimum of the function $\widehat{g}_{\nu,j,q}$,
\begin{equation}\label{e:infu*PUP*u}
\widehat{g}_{\nu,q,j} (\widehat{{\mathbf x}}_{q,*}) \leq
\lambda_{q}\Big( \frac{W_a(a2^j)}{a^{2h_q}}\Big)=\lambda_q(P(\widehat{{\mathbf S}}_{\nu,q -1} +\widehat{{\mathbf T}}_{\nu,n-q + 1} )P^*).
\end{equation}
From \eqref{e:Weyl}, \eqref{e:lambda3(PAP*)->0_in_prob}, \eqref{e:lambda2(PTP*)_bound}, \eqref{e:lambda2(P(A+B)P*)_upper_bound} and \eqref{e:infu*PUP*u},
\begin{equation}\label{e:g-nu(delta-*)=<lambda2=<oP(1)+g-nu(delta*)}
\widehat{g}_{\nu,q,j} (\widehat{{\mathbf x}}_{q,*})\leq \lambda_{q}(P(\widehat{{\mathbf S}}_{\nu,q - 1} +\widehat{{\mathbf T}}_{\nu,n- q + 1} )P^*)
\leq  \widehat{g}_{\nu,q,j} (\widehat{{\mathbf x}}_{q,*}) + o_P(1)
\end{equation}
in the set $C_{\xi,\nu}$, where
$$
\widehat{g}_{\nu,q,j} (\widehat{{\mathbf x}}_{q,*})  \stackrel{P}\rightarrow g_{q,j}({\mathbf x}_{q,*}), \quad \nu \rightarrow \infty.
$$
Consequently, for any $\varepsilon > 0$ and large enough $\nu$,
\begin{equation}\label{e:lambda2->g(delta*)}
P(|\lambda_{q}(P(\widehat{{\mathbf S}}_{\nu,q - 1} +\widehat{{\mathbf T}}_{\nu,n - q +1} )P^*) - g_{q,j}({\mathbf x}_{q,*})| \geq  \varepsilon ) \leq \gamma.
\end{equation}
Since $\gamma > 0$ is arbitrary,
\begin{equation}\label{e:lambdai0->xi0*_in_prob}
\lambda_{q}(P(\widehat{{\mathbf S}}_{\nu,q - 1} +\widehat{{\mathbf T}}_{\nu,n - q +1} )P^*) \stackrel{P}\rightarrow g_{q,j}({\mathbf x}_{q,*}), \quad \nu \rightarrow \infty.
\end{equation}
This establishes \eqref{e:lambdai/a^(2hi)_conv} for $i = q$ with
\begin{equation}\label{e:ksi_i0=g_i0(x*)}
\xi_{q}(2^j) = g_{q,j}({\mathbf x}_{q,*}) = g_{q,j}({\mathbf x}_{q,*}(2^j)).
\end{equation}

To show $(iv)$, consider any subsequence $\nu' \in \bbN'$ of
$$
\Big\{ \Big(\langle p_{\cdot,q+1},u_{q}(\nu) \rangle a^{h_{q+1} - h_{q}}, \hdots, \langle p_{\cdot,n},u_{q}(\nu) \rangle a^{h_n - h_{q}} \Big)\Big\}_{\nu \in \bbN}.
$$
We will show that there is a further subsequence $\nu'' \in \bbN''$ such that
\begin{equation}\label{e:<p3,u2>(nu'')a^(h3-h2)->delta*}
\Big(\langle p_{\cdot,q+1},u_{q}(\nu'') \rangle a^{h_{q+1} - h_{q}}, \hdots, \langle p_{\cdot,n},u_{q}(\nu'') \rangle a^{h_n - h_{q}} \Big) \stackrel{P}\rightarrow {\mathbf x}_{q,*}, \quad \nu'' \rightarrow \infty,
\end{equation}
where ${\mathbf x}_{q,*}= {\mathbf x}_{q,*}(2^j)$ is given by \eqref{e:x*_sol_determ}. This, in turn, implies \eqref{e:inner*scaling=o(1)}.

In fact, \eqref{e:lambdai/a^(2hi)_conv} and \eqref{e:ksi_i0=g_i0(x*)} imply that there is a further subsequence $\nu'' \in \bbN''$ such that
\begin{equation}\label{e:lambda_i0(nu'')_conv}
\frac{\lambda_{q}(W_a(a2^j))}{a^{2 h_{q}}} \rightarrow g_{q,j}({\mathbf x}_{q,*}) \quad \textnormal{a.s.}, \quad \nu'' \rightarrow \infty.
\end{equation}
Let $\{u_{q}(\nu)\}_{\nu \in \bbN}$ be a sequence of eigenvectors (of $W_a(a2^j)$) satisfying \eqref{e:u1,u2,u3_conv}. The subsequence $\{ \langle p_{\cdot,q+1},u_{q}(\nu'') \rangle a^{h_{q+1} - h_{q}}, \hdots, \langle p_{\cdot,n},u_{q}(\nu'') \rangle a^{h_n - h_{q}} \}_{\nu''  \in \bbN'' }
\subseteq \bbR^{n - q}$ is bounded a.s., which can be shown by an adaptation of the proof of Lemma \ref{l:<p3,u2>a(nu)^(h_3-h_2)=OP(1)}. Therefore, we may assume without loss of generality that there is some ${\mathbf x}_{q,**} = {\mathbf x}_{q,**}(\omega) \in \bbR$ such that
$$
\Big(\langle p_{\cdot,q+1},u_{q}(\nu'') \rangle a^{h_{q+1} - h_{q}}, \hdots, \langle p_{\cdot,n},u_{q}(\nu'') \rangle a^{h_n - h_{q}} \Big)\rightarrow {\mathbf x}_{q,**} \quad \textnormal{a.s.}, \quad \nu'' \rightarrow \infty.
$$
Consequently,
$$
u_{q}(\nu'')^* \frac{W_a(a2^j)}{a^{2 h_{q}}}u_{q}(\nu'')
$$
$$
\rightarrow b_{q q} \langle p_{\cdot,q},u_{q}\rangle^2  + \sum^{n}_{i=q + 1} b_{ii}x_{i,**} + 2 \langle p_{\cdot,q},u_{q}\rangle  \sum^{n}_{i = q + 1}b_{q i} x_{i,**}
+ 2 \sum_{q \leq i < i' \leq n} b_{ii'} x_{i,**} x_{i',**}
$$
$$
= g_{q,j}({\mathbf x}_{q,**}) \quad \textnormal{a.s.}, \quad \nu'' \rightarrow \infty.
$$
In view of \eqref{e:lambda_i0(nu'')_conv}, $g_{q,j}({\mathbf x}_{q,*}) = g_{q,j}({\mathbf x}_{q,**}) $. Since ${\mathbf x}_{q,*}$ is the unique global minimum of $g_{q,j}$,
$$
{\mathbf x}_{q,**} = {\mathbf x}_{q,*}.
$$
This shows \eqref{e:<p3,u2>(nu'')a^(h3-h2)->delta*} (and thus, also \eqref{e:inner*scaling=o(1)}).

It only remains to show ($ii$). First recall that the limiting matrix $B(2^j)$ satisfies the entrywise scaling relation \eqref{e:B(2^j)_entrywise_scaling}.
Therefore, the function $g_{q,j}$ in \eqref{e:g-i0} can be rewritten as
$$
g_{q,j}(x_{q +1}, \hdots, x_n)
$$
$$
= b_{q q}(2^j) \langle p_{\cdot,q},u_{q}\rangle^2
+ \sum^{n}_{i=q + 1}b_{ii}(2^j)x^{2}_{i} + 2 \langle p_{\cdot,q},u_{q}\rangle \sum^{n}_{i= q + 1}b_{q ,i}(2^j) x_i
+ 2 \sum_{q + 1 \leq i < i' \leq n} b_{ii'}(2^j)x_{i}x_{i'}
$$
$$
= 2^{j2h_{q }} \Big\{b_{q q}(1) \langle p_{\cdot,q},u_{q}\rangle^2
+ \sum^{n}_{i=q + 1}2^{j2(h_{i}-h_{q})} b_{ii}(1)x^{2}_{i} + 2 \langle p_{\cdot,q},u_{q}\rangle \sum^{n}_{i= q + 1}2^{j(h_i - h_{q})}b_{qi}(1) x_i
$$
$$
+ 2 \sum_{q + 1 \leq i < i' \leq n} 2^{j(h_i - h_{q})} 2^{j(h_{i'} - h_{q})}b_{ii'}(1)x_{i}x_{i'} \Big\}
$$
$$
= 2^{j2h_{q }} \Big\{b_{q q}(1) \langle p_{\cdot,q},u_{q}\rangle^2
+ \sum^{n}_{i=q + 1}b_{ii}(1)y^{2}_{i} + 2 \langle p_{\cdot,q},u_{q}\rangle \sum^{n}_{i= q + 1}b_{qi}(1) y_i
+ 2 \sum_{q + 1 \leq i < i' \leq n} b_{ii'}(1)y_{i}y_{i'} \Big\}
$$
$$
= 2^{j2h_{q }} g_{q,0}(y_{q +1}, \hdots, y_n),
$$
where
\begin{equation}\label{e:yi=2^j(h-hi0)xi}
y_i := 2^{j(h_i - h_{q})}x_i, \quad i = q + 1,\hdots,n.
\end{equation}
Since the relation \eqref{e:yi=2^j(h-hi0)xi} is isomorphic, minimizing the function $g_{q,j}$ over $\bbR^{n-q + 1}$ is equivalent to minimizing the function $g_{q,0}$ again over $\bbR^{n-q + 1}$, where the latter function does not depend on $j$. Since $\xi_{q}(2^j)$ and $\xi_{q}(1)$ correspond to the values attained by the functions $g_{q,j}$ and $g_{q,0}$ at their minima, respectively, relation \eqref{e:xi-i0_scales} holds. $\Box$\\

\begin{lemma}\label{l:gtilde_g_min}
Fix $q \in \{1,\hdots,n-1\}$ and let $\{u_{q}(\nu)\}_{\nu \in \bbN}$, $u_{q}$ be, respectively, a sequence of eigenvectors associated with $\lambda_{q}(W_a(a(\nu) 2^j))$ and its limit in probability as in \eqref{e:u1,u2,u3_conv}. Let $\widehat{g}_{\nu,q,j}, g_{q,j}: \bbR^{n - q} \rightarrow \bbR$ be the random and deterministic functions, respectively, defined by
$$
\widehat{g}_{\nu,q,j}(x_{q +1}, \hdots, x_n)
 = \widehat{b}_{q q}(2^j) \langle p_{\cdot,q},u_{q}(\nu)\rangle^2
+ \sum^{n}_{i=q + 1}\widehat{b}_{ii}(2^j)x^{2}_{i}
$$
\begin{equation}\label{e:g-tilde-nu,i0}
+ 2 \langle p_{\cdot,q},u_{q}(\nu)\rangle \sum^{n}_{i= q + 1}\widehat{b}_{q i}(2^j)x_i
+ 2 \sum_{q + 1 \leq i < i' \leq n} \widehat{b}_{ii'}(2^j)x_{i}x_{i'} + \widehat{r}_{\nu,q,j}(x_{q+1},\hdots,x_n)
\end{equation}
and
$$
g_{q,j}(x_{q +1}, \hdots, x_n) = b_{q q}(2^j) \langle p_{\cdot,q},u_{q}\rangle^2
+ \sum^{n}_{i=q + 1}b_{ii}(2^j)x^{2}_{i}
$$
\begin{equation}\label{e:g-i0}
+ 2 \langle p_{\cdot,q},u_{q}\rangle \sum^{n}_{i= q + 1}b_{q i}(2^j)x_i
+ 2 \sum_{q + 1 \leq i < i' \leq n} b_{ii'}(2^j)x_{i}x_{i'},
\end{equation}
where the residual function in \eqref{e:g-tilde-nu,i0} is given by
$$
\widehat{r}_{\nu,q,j}(x_{q+1},\hdots,x_n) = \sum^{q-1}_{i=1}\widehat{b}_{ii}(2^j)\langle p_{\cdot,i},u_q(\nu)\rangle^2 a(\nu)^{2(h_i - h_q)}
$$
$$
+ 2 \langle p_{\cdot,q},u_q(\nu)\rangle\sum^{q-1}_{i=1}\widehat{b}_{iq}(2^j)\langle p_{\cdot,i},u_q(\nu)\rangle a(\nu)^{h_i - h_q}
+ 2 \sum^{q-1}_{i=1}\sum^{n}_{i'=q+1}\widehat{b}_{ii'}(2^j)\langle p_{\cdot,i},u_q(\nu)\rangle a(\nu)^{h_i - h_q}x_{i'}.
$$
Then, each function $\widehat{g}_{\nu,q,j}$ and $g_{q,j}$ has a unique global minimum.
\end{lemma}
\begin{proof}
We only establish the claim for $\widehat{g}_{\nu,q,j}$, since the argument for $g_{q,j}$ is essentially identical. We will drop the factor $2^j$ for notational simplicity.

The first order conditions for the minimization of $\widehat{g}_{\nu,q,j}$ give the matrix system
$$
\Big( \hspace{1mm}\widehat{b}_{ii'}\hspace{1mm}\Big)_{i,i'= q+1,\hdots,n} (x_{q+1},\hdots,x_n)^*
$$
\begin{equation}\label{e:system_x*_sol}
= - (\widehat{b}_{q,q + 1},\hdots,\widehat{b}_{q,n})^* \langle p_{\cdot,q},u_{q}(\nu)\rangle - \nabla^*_{q+1,\hdots,n}\widehat{r}_{\nu,q,j}(x_{q+1},\hdots,x_n),
\end{equation}
where $\nabla_{q+1,\hdots,n}$ denotes the gradient with respect to the vector $x_{q+1},\hdots,x_n$. Note that $\nabla^*_{q+1,\hdots,n}\widehat{r}_{\nu,q,j}(x_{q+1},\hdots,x_n)$ is constant. Since the matrix $\Big( \hspace{1mm}\widehat{b}_{ii'}\hspace{1mm}\Big)_{i,i'= q+1,\hdots,n}$ is nonsingular a.s., a solution $\widehat{{\mathbf x}}_{q,*}$ to \eqref{e:system_x*_sol} always exists. Moreover, the Hessian matrix is given by
$$
\Big(\hspace{0.5mm} \frac{\partial^2}{\partial x_{i}\partial x_{i'}}\widehat{g}_{\nu,q,j}\hspace{0.5mm}\Big)_{i,i'= q+1,\hdots,n} = 2 \Big(\hspace{1mm} \widehat{b}_{ii'}\hspace{1mm}\Big)_{i,i'= q+1,\hdots,n}
$$
which is symmetric positive definite a.s. Therefore, the solution to \eqref{e:system_x*_sol} is the unique global minimum of $\widehat{g}_{\nu,q,j}$. $\Box$\\
\end{proof}

In proofs, the global minima of $\widehat{g}_{\nu,q,j}$ and $g_{q,j}$ provided in Lemma \ref{l:gtilde_g_min} will be denoted by
\begin{equation}\label{e:x*_sol}
\widehat{{\mathbf x}}_{q,*}(2^j) =\widehat{{\mathbf x}}_{q,*} = (\widehat{x}_{q+1,*},\hdots,\widehat{x}_{n,*}) \quad \textnormal{a.s.}
\end{equation}
and
\begin{equation}\label{e:x*_sol_determ}
{\mathbf x}_{q,*}(2^j) = {\mathbf x}_{q,*} = (x_{q+1,*},\hdots, x_{n,*}),
\end{equation}
respectively.\\

\begin{lemma}\label{l:<p3,u2>a(nu)^(h_3-h_2)=OP(1)}
Let $\{\textbf{B}_\nu\}_{\nu \in \bbN}$,  $\textbf{B}_\nu = \Big(\widehat{b}_{i_1 i_2}\Big)_{i_1,i_2=1,\hdots,n}$, be a sequence of symmetric, and not necessarily positive semidefinite, random matrices such that
$\textbf{B}_\nu \stackrel{P}\rightarrow \textbf{B}$, as $\nu \rightarrow \infty$, where $\textbf{B} = \Big(b_{i_1 i_2}\Big)_{i_1,i_2=1,\hdots,n}$ is deterministic. Consider $h_1,\hdots,h_n$ and $P$ as in \eqref{e:H_h1<...<hn}. In addition, for a fixed $q \in \{ 1,\hdots,n-1 \}$, assume that
\begin{equation}\label{e:bii>0,i=q+1,...n}
b_{i i} > 0, \quad  i = q +1,\hdots,n.
\end{equation}
Let
$$
\textbf{W}_\nu
= P \textnormal{diag}(a(\nu)^{h_1},\hdots,a(\nu)^{h_n})\left(\begin{array}{ccc}
\widehat{b}_{11} & \hdots & \widehat{b}_{1n}\\
\vdots  & \ddots & \vdots \\
\widehat{b}_{1n} & \hdots & \widehat{b}_{nn}\\
\end{array}\right)\textnormal{diag}(a(\nu)^{h_1},\hdots,a(\nu)^{h_n})P^*
$$
and let $u_{q}(v)$ be a unit eigenvector associated with $\lambda_{q}(\textbf{W}_\nu)$. Then,
\begin{equation}\label{e:|<p3,u2>|a^(h3-h2)=O_P(1)}
\Big\{ |\langle p_{\cdot,i},u_{q}(\nu) \rangle| a(\nu)^{h_i-h_{q}} \Big\}_{\nu \in \bbN} = O_P(1), \quad i = q + 1,\hdots,n.
\end{equation}
\end{lemma}
\begin{proof} Rewrite
$$
\bbR \ni \frac{\lambda_{q}(\textbf{W}_\nu)}{a^{2h_{q}}} = u^*_{q}(\nu)\frac{\textbf{W}_\nu}{a^{2h_{q}}}u_{q}(\nu)
$$
$$
= u^*_{q}(\nu)P \textnormal{diag}(a^{h_1-h_{q}},\hdots,1,\hdots,a^{h_n-h_{q}})\textbf{B}_\nu\textnormal{diag}(a^{h_1-h_{q}},\hdots,1,\hdots,a^{h_n-h_{q}}) P^*u_{q}(\nu)
$$
$$
= \sum^{n}_{i=1}\widehat{b}_{ii}\langle p_{\cdot,i},u_{i}(\nu)\rangle^2 a^{2(h_i-h_{q})}
+ 2 \sum_{i < i'}  \widehat{b}_{i i'}\langle p_{\cdot,i},u_{q}(\nu)\rangle a^{h_i-h_{q}}\langle p_{\cdot,i'},u_{q}(\nu)\rangle a^{h_{i'}-h_{q}}
$$
$$
= \inf_{{\mathcal U}_{q}} \sup_{u \in {\mathcal U}_{q} \cap S^{n-1}} u^* \frac{\textbf{W}_\nu}{a^{2h_{q}}}u \leq \sup_{u \in \{p_{\cdot,q + 1}, \hdots, p_{\cdot,n}\}^{\perp} \cap S^{n-1}} u^* \frac{\textbf{W}_\nu}{a^{2h_{q}}}u
$$
\begin{equation}\label{e:lambda2<C}
=: w^*_q(\nu) \frac{\textbf{W}_\nu}{a^{2h_{q}}}w_q(\nu) = O_P(1),
\end{equation}
where the last equality is a consequence of the fact that
$$
|\langle p_{\cdot,i},w_{q}(\nu)\rangle| \hspace{0.5mm}a^{h_i-h_{q}} =
\left\{\begin{array}{cc}
o_P(1), & i = 1,\hdots,q-1;\\
0, & i = q + 1,\hdots,n.
\end{array}\right.
$$
We claim that, as a consequence of \eqref{e:lambda2<C}, \eqref{e:|<p3,u2>|a^(h3-h2)=O_P(1)} holds for any $i = q + 1,\hdots,n$. By contradiction, suppose that for some $i_1 \in \{q + 1,\hdots,n\}$ there exists $\varepsilon_0 > 0 $ such that, for $m \in \bbN$ and a subsequence $\nu'= \nu'(m) \in \bbN'$,
$$
P\Big(\langle p_{\cdot,i_1},u_{q}(\nu')\rangle^2 a(\nu')^{2(h_{i_1}-h_{q})} > m\Big) \geq \varepsilon_0, \quad m \rightarrow \infty.
$$
Therefore,
$$
P\Big(\sum^{n}_{i = q+1}\widehat{b}_{ii}\langle p_{\cdot,i},u_{q}(\nu')\rangle^2 a(\nu')^{2(h_{i}-h_{q})} > m \widehat{b}_{i_1 i_1} \Big) \geq \varepsilon_0, \quad m \rightarrow \infty.
$$
By \eqref{e:bii>0,i=q+1,...n}, with non-vanishing probability, for every $m \in \bbN$ and $\nu' = \nu'(m)$, we can rewrite the left-hand side of \eqref{e:lambda2<C} as
$$
\frac{\lambda_{q}(\textbf{W}_{\nu'})}{a(\nu')^{2h_{q}}} = \sum^{n}_{i = q+1}\widehat{b}_{i i}\langle p_{\cdot,i},u_{q}(\nu')\rangle^2 a(\nu')^{2(h_{i}-h_{q})} (1 + o_P(1)) > m b_{i_1 i_1}  \hspace{1mm}(1 + o_P(1))> 0,
$$
since $\widehat{b}_{i i}  \stackrel{P}\rightarrow b_{i i} > 0$, $\nu' \rightarrow \infty$, $i = q+1,\hdots,n$. Therefore, $\frac{\lambda_{q}(\textbf{W}_{\nu'})}{a(\nu')^{2h_{q}}}$ is not bounded in probability from above, which contradicts \eqref{e:lambda2<C}. Thus, \eqref{e:|<p3,u2>|a^(h3-h2)=O_P(1)} holds for any $i = q + 1, \hdots,n$, as claimed. $\Box$\\
\end{proof}

We are now in a position to prove Theorem \ref{t:asympt_normality_lambda2}.\\

\noindent{\sc Proof of Theorem \ref{t:asympt_normality_lambda2}}: Fix $j$. For $q = 1,\hdots,n$, define the sequence of $\bbR$-valued functions $\{f_{\nu,q}(B )\}_{\nu \in \bbN}$, where
$$
{\mathcal H}_{\geq 0}(n,\bbR) \ni B \mapsto f_{\nu,q}(B)
$$
\begin{equation}\label{e:fv}
= \log \lambda_{q}\Big(\frac{P \textnormal{diag}(a^{h_1},\hdots, a^{h_n}) B  \textnormal{diag}(a^{h_1},\hdots, a^{h_n})P^* }{a^{2h_q}}\Big).
\end{equation}
By Lemma 4.1 in Abry and Didier \cite{abry:didier:2017},
\begin{equation}\label{e:sqrt(K)(B^-B)->N(0,sigma^2)}
\bbR^{n(n+1)/2} \ni \sqrt{K_{a,j}} \hspace{1mm}\textnormal{vec}_{{\mathcal S}}(\widehat{B}_{a}(2^j) - B(2^j)) \stackrel{d}\rightarrow {\mathcal N}(0,\Sigma_B(j)), \quad \nu \rightarrow \infty,
\end{equation}
where $\Sigma_B(j) \in {\mathcal  H}_{>0}(n(n+1)/2,\bbR)$. In particular,
\begin{equation}\label{e:Bhat->B}
\widehat{B}_{a}(2^j) \stackrel{P}\rightarrow B(2^j).
\end{equation}
Recall that, for any $M \in {\mathcal H}(n,\bbR)$, the differential of a simple eigenvalue $\lambda_{q}(M)$, $q = 1,\hdots,n$, exists in a vicinity of $M$ and is given by
\begin{equation}\label{e:dlambdal}
d\lambda_{q}(M) = {\mathbf u}_{q}(\nu)^* \hspace{0.5mm}\{d M \} \hspace{0.5mm}{\mathbf u}_{q}(\nu),
\end{equation}
where ${\mathbf u}_{q}(\nu)$ is a unit eigenvector of $M$ associated with $\lambda_{q}(M)$ (Magnus \cite{magnus:1985}, p.\ 182, Theorem 1). By Proposition \ref{p:convergence}, $(i)$, except for
\begin{equation}\label{e:EW_has_pairwise_distinct_eigens}
\lambda_q\Big( \frac{\bbE W_a(a(\nu)2^j)}{a^{2h_q}} \Big)= \lambda_q\Big(\frac{P \textnormal{diag}(a^{h_1},\hdots, a^{h_n}) B(2^j)  \textnormal{diag}(a^{h_1},\hdots, a^{h_n})P^*}{a^{2h_q}}\Big),
\end{equation}
all eigenvalues of the matrix $\bbE W_a(a(\nu)2^j)/a^{2h_q}$ either go to zero or blow up. Therefore, \eqref{e:EW_has_pairwise_distinct_eigens} is a simple eigenvalue for large enough $\nu$. Therefore, also for large $\nu$, by \eqref{e:dlambdal} the derivative of the function $f_{\nu,q}$ in \eqref{e:fv} exists in a vicinity ${\mathcal O}$ of $B(2^j)$ in ${\mathcal H}_{> 0}(n,\bbR)$. For any $B \in {\mathcal O}$, an application of Proposition \ref{p:mean_value_theorem} yields
\begin{equation}\label{e:Taylor_determ}
f_{\nu,q}(B) - f_{\nu,q}(B(2^j)) = \sum^{n}_{i_1 = 1}\sum^{n}_{i_2 = 1}\frac{\partial}{\partial b_{i_1,i_2}}f_{\nu,q}(\breve{B}) \hspace{1mm} \pi_{i_1,i_2}(B - B(2^j))
\end{equation}
for some matrix $\breve{B} \in {\mathcal  H}_{>0}(n,\bbR)$ lying in a segment connecting $B$ and $B(2^j)$ across ${\mathcal  H}_{>0}(n,\bbR)$ (see \eqref{e:path_in_matrix_space} and \eqref{e:H(1)-H(0)=H'(varsigma)}). Define the event
$$
A = \Big\{ \omega:  \breve{B}_{a}(2^j) \in {\mathcal O} \Big\}.
$$
By \eqref{e:Bhat->B},
\begin{equation}\label{e:W_has_pairwise_distinct_eigens}
P (A) \rightarrow 1, \quad \nu \rightarrow \infty.
\end{equation}
By \eqref{e:Taylor_determ}, for large enough $\nu$ and in the set $A$, the expansion
\begin{equation}\label{e:Taylor}
f_{\nu,q}(\widehat{B}_{a} (2^j)) - f_{\nu,q}(B(2^j)) = \sum^{n}_{i_1 = 1}\sum^{n}_{i_2 = 1}\frac{\partial}{\partial b_{i_1,i_2}}f_{\nu,q}(\breve{B}_{a} (2^j)) \hspace{1mm} \pi_{i_1,i_2}(\widehat{B}_{a}(2^j) - B(2^j))
\end{equation}
holds for some matrix $\breve{B}_{a}(2^j)$ lying in a segment connecting $\widehat{B}_{a}(2^j)$ and $B(2^j)$ across ${\mathcal  H}_{>0}(n,\bbR)$.
So, fix $q \in \{2,\hdots,n-1\}$ and let
\begin{equation}\label{e:W_a,2j}
\breve{W}(a2^j) := P \textnormal{diag}(a^{h_1},\hdots,a^{h_n})\breve{B}_a(2^j)\textnormal{diag}(a^{h_1},\hdots,a^{h_n})P^*,
\end{equation}
$$
\lambda_{q}(\breve{W}(a2^j)):= \inf_{{\mathcal U}_{q} }\sup_{u \in {\mathcal U}_{q} \cap S^{n-1}}u^* \breve{W}(a2^j)u.
$$
Consider the matrix
\begin{equation}\label{e:df/dbi1,i2}
\Big\{\frac{\partial}{\partial b_{i_1,i_2}}f_{\nu,q}(\breve{B}_{a}(2^j))\Big\}_{i_1,i_2 = 1,\hdots,n} = \Big\{\frac{a^{2 h_q}}{\lambda_{q}(\breve{W}(a2^j))}\hspace{1mm}\frac{\partial}{\partial b_{i_1,i_2}}\lambda_{q}\Big( \frac{\breve{W}(a2^j)}{a^{2h_q}}\Big)\Big\}_{i_1,i_2 = 1,\hdots,n}.
\end{equation}
where the differential of the eigenvalue $\lambda_{q}(\breve{W}(a2^j)/a^{2h_q})$ is given by expression \eqref{e:dlambdal} with $\breve{W}(a2^j)/a^{2h_q}$ in place of $M$ and $u_{q}(\nu)$ denoting a unit eigenvector of $\breve{W}(a2^j)/a^{2h_q}$ associated with its $q$-th eigenvalue. Then, each entry of the matrix \eqref{e:df/dbi1,i2} can be rewritten as
\begin{equation}\label{e:derivative_log_eigen}
\bbR \ni \frac{1}{a^{-2h_{q}}\lambda_{q}(\breve{W}(a2^j))}\Big(a^{-2h_{q}} u^{*}_{q}(\nu)\Big\{\frac{\partial}{\partial b_{i_1,i_2}}\breve{W}(a2^j) \Big\}u_{q}(\nu)\Big), \quad i_1,i_2 = 1,\hdots,n.
\end{equation}
To establish the limit in probability of \eqref{e:derivative_log_eigen}, note that, by an analogous argument for proving Proposition \ref{p:convergence}, all the claims in the latter proposition hold for the matrix $\breve{W}(a2^j)$ as in \eqref{e:W_a,2j} in place of $W_a(a2^j)$. So, write
\begin{equation}\label{e:d/dbBtilde_i1,i2=1i1,i2}
M(n,\bbR) \ni \frac{\partial}{\partial b_{i_1,i_2}}\breve{B}_{a} (2^j) = {\boldsymbol 1}_{i_1,i_2},  \quad  i_1,i_2 = 1,\hdots,n,
\end{equation}
where ${\boldsymbol 1}_{i_1,i_2}$ is a matrix with 1 on entry $(i_1,i_2)$ and zeroes elsewhere. Therefore, for $1 \leq i_1 \leq i_2 \leq n$, we can pick the sequence $u_{q}(\nu)$ as to obtain, from \eqref{e:W_a,2j} and \eqref{e:df/dbi1,i2},
$$
a^{-2h_{q}} u^*_q(\nu)\Big\{\frac{\partial}{\partial b_{i_1,i_2}}\breve{W}(a2^j)\Big\}u_q(\nu)
$$
$$
= u^*_q(\nu)P \textnormal{diag}(a^{h_1 - h_{q}},\hdots,1,\hdots,a^{h_n - h_{q}}){\boldsymbol 1}_{i_1,i_2}
\textnormal{diag}(a^{h_1 - h_{q}},\hdots,1,\hdots,a^{h_n - h_{q}})P^*u_q(\nu)
$$
\begin{equation}\label{e:limit_derivative_eigenvalue}
=\langle p_{\cdot,i_1},u_{q}(\nu)\rangle a^{h_{i_1} - h_{q}} \langle p_{\cdot,i_2},u_{q}(\nu)\rangle a^{h_{i_2} - h_{q}}
\stackrel{P}\rightarrow \left\{\begin{array}{cc}
0, & i_1 < q;\\
\langle p_{\cdot,q},u_{q}\rangle^2 & i_1 = q = i_2;\\
\langle p_{\cdot,q},u_{q}\rangle x_{i_2,*}, & i_1 = q < i_2;\\
x_{i_1,*}x_{i_2,*}, &  q < i_1,\\
\end{array}\right.
\end{equation}
for entries $x_{i,*}$ (depending on $q$), $i = q+1,\hdots,n$, of the vector $\textbf{x}_{q,*}(2^j)$ as given by expression \eqref{e:inner*scaling=o(1)} in Proposition \ref{p:convergence}. Then, by \eqref{e:limit_derivative_eigenvalue} and \eqref{e:lambdai/a^(2hi)_conv} in Proposition \ref{p:convergence}, expression \eqref{e:derivative_log_eigen} converges in probability to the matrix
\begin{equation}\label{e:limit_derivative}
\frac{1}{\xi_q(2^j)}\times
\left(\begin{array}{cccccc}
0 & \hdots & 0 & 0 & \hdots & 0 \\
\vdots & \ddots & \vdots & \vdots & \ddots & \vdots \\
0 & \hdots & \langle p_{\cdot,q}, u_{q} \rangle^2 & \langle p_{\cdot,q}, u_{q} \rangle x_{q+1,*} & \hdots & \langle p_{\cdot,q} , u_{q}\rangle x_{n,*}\\
0 & \hdots & \langle p_{\cdot,q}, u_{q} \rangle x_{q+1,*} &  x^2_{q+1,*} & \hdots & x_{q+1,*}x_{n,*}\\
\vdots & \ddots & \vdots & \vdots & \ddots & \vdots \\
0 & \hdots & \langle p_{\cdot,q} , u_{q}\rangle x_{n,*} &  x_{q+1,*}x_{n,*} & \hdots & x^2_{n,*}\\
\end{array}\right),
\end{equation}
where $\xi_q(2^j) > 0$. Turning back to \eqref{e:Taylor}, expression \eqref{e:limit_derivative} and Theorem \ref{t:asymptotic_normality_wavecoef_fixed_scales} imply that
$$
\bbR \ni \sum^{n}_{i_1 = 1}\sum^{n}_{i_2 = 1}\frac{\partial}{\partial b_{i_1,i_2}}f_{\nu,q}(\breve{B}(2^j)) \sqrt{K_{a,j}} \hspace{1mm}\pi_{i_1,i_2}(\widehat{B}_{a}(2^j) - B(2^j))  \stackrel{d}\rightarrow  {\mathcal N}(0,\sigma^2_{q}(j)),
$$
as $\nu \rightarrow \infty$, where $\sigma^2_{q}(j) > 0$ as a consequence of the fact that $\Sigma_B(j)$ in \eqref{e:sqrt(K)(B^-B)->N(0,sigma^2)} has full rank.

The behavior of the remaining terms
$$
\sqrt{K_{a,j}}( \log \lambda_{q}(W_a(a(\nu)2^{j}))  - \log \lambda_{q}(\bbE W_a(a(\nu)2^{j})) ) , \quad q = 1,n,
$$
can be established by a similar argument starting from \eqref{e:Taylor} and applying Proposition \ref{p:convergence}. Since, for $j = j_1,\hdots,j_2$ and $q = 1,\hdots,n$, the asymptotic normality of each individual log-eigenvalue results from the factor \eqref{e:sqrt(K)(B^-B)->N(0,sigma^2)}, then the limiting distribution is a $n$-variate normal, as claimed. This shows \eqref{e:asympt_normality_lambda2}. $\Box$\\

\subsection{Asymptotic theory for the wavelet eigenvalue regression estimator}

The following lemma is used in the proof of Corollary \ref{c:h^q_asymptotically_normal}.
\begin{lemma}\label{l:|lambdaq(EW)-xiq(2^j)|_bound}
Fix $j \in \{j_1,\hdots,j_2\}$. Then, for some $C > 0$ that does not depend on $j$,
\begin{equation}\label{e:|lambdaq(EW)-xiq(2^j)|_bound}
\Big|\frac{\lambda_q(\bbE W_a(a(\nu)2^j))}{a(\nu)^{2h_q}} - \xi_q(2^j)\Big| \leq \frac{C}{a(\nu)^{ \min_{1 \leq q_1 < q_2 \leq n} (h_{q_2} - h_{q_1})}}, \quad q = 1,\hdots,n,
\end{equation}
for large enough $\nu \in \bbN$.
\end{lemma}
\begin{proof}
Since the argument is similar to that for proving Proposition \ref{p:convergence}, $(i)$, we only write it out in dimension $n = 3$ and for $q = 2$. In the following bounds, the generic constant $C > 0$ does not depend on $j$ since we can always take the maximum over $j = j_1,\hdots,j_2$.

For notational simplicity, write $\Big( b_{ii'} \Big)_{i,i'= 1,2,3} = \Big(b_{ii'}(2^j)\Big)_{i,i'= 1,2,3}$ as in \eqref{e:Bhat,B}. Let $\{u_q(\nu)\}_{\nu \in \bbN} \subseteq S^2$ be the sequence of eigenvectors of $\bbE W_a(a2^j)$ associated with the eigenvalue $\lambda_q(\bbE W_a(a2^j))$. By Proposition \ref{p:convergence}, we can assume that, for $q = 1,2,3$, relation \eqref{e:u1,u2,u3_conv} holds and that
\begin{equation}\label{e:<p_q,u_q>_>_0}
\langle p_{\cdot,q},u_q\rangle > 0.
\end{equation}
We can further assume, without loss of generality, that
$$
p_{\cdot,3} = e_3, \quad u_1 = e_1, \quad u_2 = e_2, \quad p_{\cdot,2} \perp e_1.
$$
From expression \eqref{e:inner*scaling=o(1)},
\begin{equation}\label{e:|<e2,u1(nu)>|a^(h2-h1)_|<e3,u1(nu)>|a^(h3-h1)}
|\langle p_{\cdot,2},u_1(\nu)\rangle |a^{h_2 - h_1} = O(1), \quad |\langle e_3,u_1(\nu)\rangle| a^{h_3 - h_1} = O(1)
\end{equation}
and
\begin{equation}\label{e:|<e3,u2(nu)>|/a^(h3-h2)=O(1)}
|\langle e_3,u_2(\nu)\rangle|a^{h_3 - h_2} = O(1).
\end{equation}
Consider the decomposition
\begin{equation}\label{e:p2=sum_betai_ei}
p_{\cdot,2} = \sum^{3}_{i=1}\beta_i e_i,  \quad \textnormal{for some triple $\beta_i$, $i=1,2,3$}.
\end{equation}
Then,
$$
\langle p_{\cdot,2},u_1(\nu) \rangle = \beta_2 \langle e_2,u_1(\nu)\rangle + \beta_3 \langle e_3,u_1(\nu)\rangle
$$
and \eqref{e:|<e2,u1(nu)>|a^(h2-h1)_|<e3,u1(nu)>|a^(h3-h1)} imply that
\begin{equation}\label{e:|<e2,u1(nu)>|_bound}
|\langle e_2,u_1(\nu)\rangle | \leq \frac{C}{a^{h_2 - h_1}}.
\end{equation}
Hence, by \eqref{e:|<e2,u1(nu)>|a^(h2-h1)_|<e3,u1(nu)>|a^(h3-h1)}, \eqref{e:|<e3,u2(nu)>|/a^(h3-h2)=O(1)}, \eqref{e:|<e2,u1(nu)>|_bound} and the orthogonality relation $\langle u_1(\nu),u_2(\nu)\rangle = 0$,
$$
|\langle e_1,u_2(\nu)\rangle| = \frac{1}{|\langle e_1,u_1(\nu)\rangle|} \Big|\langle e_2,u_1(\nu)\rangle \langle e_2,u_2(\nu)\rangle
+\langle e_3,u_1(\nu)\rangle \langle e_3,u_2(\nu)\rangle\Big|
$$
\begin{equation}\label{e:|<e1,u2(nu)>|=<a^(h2-h1)}
= \frac{1}{|1 - o(1)|} \Big|O\Big(\frac{1}{a^{h_2-h_1}}\Big) (1 - o(1))
+O\Big(\frac{1}{a^{h_3 - h_1}}\Big)O\Big(\frac{1}{a^{h_3 - h_2}}\Big)\Big| \leq \frac{C}{a^{h_2 - h_1}}.
\end{equation}
In view of \eqref{e:|<e3,u2(nu)>|/a^(h3-h2)=O(1)} and \eqref{e:|<e1,u2(nu)>|=<a^(h2-h1)}, the unit norm relation $\|u_2(\nu)\|^{2} = 1$ implies that
\begin{equation}\label{e:1-<e2,u2(nu)>^2=O(1/a^(2(h2-h1)))}
1 - \langle e_{2},u_2(\nu)\rangle^2 = O\Big(\frac{1}{a^{2\min\{h_2 - h_1,h_3 - h_2\}}}\Big).
\end{equation}
Consider the function
\begin{equation}\label{e:varrho_nu,2,j(x)}
\varrho_{\nu,2,j}(x) = b_{22}\langle p_{\cdot,2},u_2(\nu)\rangle^2 + b_{33}x^2 + 2 b_{23} \langle p_{\cdot,2},u_2(\nu)\rangle x + r_{\nu,2,j}(x)
\end{equation}
and its limiting counterpart
$$
g_{2,j}(x) = b_{22}\langle p_{\cdot,2},u_2 \rangle^2
+ b_{33}x^2 + 2 b_{23} \langle p_{\cdot,2},u_2\rangle x
$$
(see \eqref{e:g-i0}), where the residual function in \eqref{e:varrho_nu,2,j(x)} is given by
$$
r_{\nu,2,j}(x) = b_{11}\langle p_{\cdot,1},u_2(\nu) \rangle^2 a^{2(h_1-h_2)} + 2 b_{12}\langle p_{\cdot,1},u_2(\nu) \rangle a^{h_1-h_2} \langle p_{\cdot,2},u_2(\nu) \rangle
$$
$$
+ 2 b_{13}\langle p_{\cdot,1},u_2(\nu) \rangle a^{h_1-h_2} x.
$$
For $\nu \in \bbN$, let
\begin{equation}\label{e:y*,3=argmin}
x_{*}(\nu) = \textnormal{argmin}_{x \in \bbR} \hspace{1mm} \varrho_{\nu,2,j}(x), \quad x_{*} = \textnormal{argmin}_{x \in \bbR} g_{2,j}(x)
\end{equation}
be the unique global minima of $\varrho_{\nu,2,j}(\cdot)$ and $g_{2,j}(\cdot)$, respectively, which can be expressed as
\begin{equation}\label{e:y*,3(nu)_system}
x_{*}(\nu) = - \frac{1}{b_{33}} \Big(b_{23}\langle p_{\cdot,2},u_2(\nu)\rangle +  b_{13}\langle p_{\cdot,1},u_2(\nu) \rangle a^{h_1-h_2}\Big), \quad x_{*} = - \frac{b_{23}}{b_{33}} \langle p_{\cdot,2},u_2\rangle.
\end{equation}
Note that the sequence $\{x_{*}(\nu)\}_{\nu \in \bbN}$ in \eqref{e:y*,3(nu)_system} converges to the solution $x_*$ of the limiting system. Let
\begin{equation}\label{e:w(nu)_in_span(p2,e3)}
\{{\mathbf w}(\nu)\}_{\nu \in \bbN} \subseteq \textnormal{span}\{p_{\cdot,2},p_{\cdot,3}\} \cap S^2 = \textnormal{span}\{p_{\cdot,2},e_{3}\} \cap S^2
\end{equation}
be a sequence such that
\begin{equation}\label{e:<e3,w(nu)>=C/a^(h3-h2)}
\langle p_{\cdot,3}, {\mathbf w}(\nu)\rangle = \langle e_{3}, {\mathbf w}(\nu)\rangle = \frac{x_{*}(\nu)}{a^{h_3 - h_2}} \in (-1,1),
\end{equation}
which is possible for large enough $\nu$. From the unit norm relation
$$
1 = \|{\mathbf w}(\nu)\|^2 = \langle e_2, {\mathbf w}(\nu)\rangle^2 + \langle e_{3}, {\mathbf w}(\nu)\rangle^2,
$$
we obtain
\begin{equation}\label{e:1-<e2,w(nu)>^2=C/a^(2(h3-h2))}
1 - \langle e_2, {\mathbf w}(\nu)\rangle^2 = \frac{x^2_{*}(\nu)}{a^{2(h_3 - h_2)}}.
\end{equation}
By \eqref{e:1-<e2,u2(nu)>^2=O(1/a^(2(h2-h1)))}, \eqref{e:1-<e2,w(nu)>^2=C/a^(2(h3-h2))} and the mean value theorem applied to the function $f(x) = \sqrt{x}$ under condition \eqref{e:<p_q,u_q>_>_0},
$$
\Big| \langle e_2, {\mathbf w}(\nu)\rangle - \langle e_2, u_2(\nu)\rangle \Big| = \Big|\Big(\langle e_2, {\mathbf w}(\nu)\rangle -1\Big) - \Big(\langle e_2, u_2(\nu)\rangle-1\Big)\Big|
$$
\begin{equation}\label{e:|<e2,w(nu)>-<e2,u2(nu)>|_bound}
= \Big|f'(\theta_1(\nu)) \Big(\langle e_2, {\mathbf w}(\nu)\rangle^2 -1\Big) - f'(\theta_2(\nu))\Big(\langle e_2, u_2(\nu)\rangle^2 -1\Big)\Big|
 \leq \frac{C}{a^{2\min\{h_2 - h_1,h_3 - h_2\}}}
\end{equation}
for bounded sequences $\{\theta_1(\nu)\}_{\nu \in \bbN}$ and $\{\theta_2(\nu)\}_{\nu \in \bbN}$. By \eqref{e:|<e3,u2(nu)>|/a^(h3-h2)=O(1)}, \eqref{e:p2=sum_betai_ei}, \eqref{e:|<e1,u2(nu)>|=<a^(h2-h1)}, \eqref{e:<e3,w(nu)>=C/a^(h3-h2)} and \eqref{e:|<e2,w(nu)>-<e2,u2(nu)>|_bound},
$$
\Big|\langle p_{\cdot,2},{\mathbf w}(\nu)\rangle - \langle p_{\cdot,2},u_2(\nu)\rangle \Big|
$$
$$
= \Big| \beta_2 \langle e_2,{\mathbf w}(\nu) - u_2(\nu)\rangle + \sum_{i=1,3} e_i \langle p_{\cdot,i},{\mathbf w}(\nu)\rangle
- \sum_{i=1,3} e_i \langle p_{\cdot,i},u_2(\nu)\rangle\Big|
$$
\begin{equation}\label{e:|<p2,w(nu)>-<p2,u2(nu)>|_bound}
\leq \frac{C}{a^{2 \min\{h_2 - h_1,h_3 - h_2\}}} + \frac{C'}{a^{h_3 - h_2}}+ \Big( \frac{C''}{a^{h_2 - h_1}} + \frac{C'''}{a^{h_3 - h_2}}\Big)
\leq  \frac{C}{a^{\min\{h_2 - h_1,h_3 - h_2\}}}.
\end{equation}
Therefore, by the mean value theorem applied to the function $f(x) = x^2$,
$$
\Big|\langle p_{\cdot,2},{\mathbf w}(\nu)\rangle^2 - \langle p_{\cdot,2},u_2(\nu)\rangle^2 \Big|
$$
\begin{equation}\label{e:|<p2,w(nu)>2-<p2,u2(nu)>2|}
= \Big|f'(\theta_3(\nu)) \Big(\langle p_{\cdot,2},{\mathbf w}(\nu)\rangle - \langle p_{\cdot,2},u_2(\nu) \rangle \Big) \Big|\leq \frac{C}{a^{\min\{h_2 - h_1,h_3 - h_2\}}}
\end{equation}
for some bounded sequence $\{\theta_3(\nu)\}_{\nu \in \bbN}$. In addition, by \eqref{e:1-<e2,u2(nu)>^2=O(1/a^(2(h2-h1)))} and a similar reasoning,
\begin{equation}\label{e:|<p2,u2(nu)>-<p2,u2>|_bound}
\Big|\langle p_{\cdot,2},u_2(\nu)\rangle -  \langle p_{\cdot,2},u_2\rangle\Big| \leq \frac{C}{a^{\min_{1 \leq q_1 < q_2 \leq 3}(h_{q_2} - h_{q_1})}},
\end{equation}
whence
\begin{equation}\label{e:|x*(nu)-x*|_bound}
|x_*(\nu) -  x_*| \leq \frac{C}{a^{\min_{1 \leq q_1 < q_2 \leq 3}(h_{q_2} - h_{q_1})}}.
\end{equation}

%
On the other hand, define the matrices
$$
{\mathbf S}_{\nu,1}
= \left(\begin{array}{ccc}
b_{11}a^{2(h_1 - h_2)} & b_{12}a^{h_1 - h_2} & 0 \\
b_{12}a^{h_1 - h_2} & 0 & 0 \\
0 & 0 & 0
\end{array}\right), \quad
{\mathbf T}_{\nu,2}
= \left(\begin{array}{ccc}
0 & 0 & b_{13}a^{h_1 - h_2} a^{h_3 - h_2} \\
0 & b_{22} & b_{23}a^{h_3 - h_2}\\
a^{h_3 - h_2}  & b_{23}a^{h_3 - h_2} & b_{33} a^{2(h_3 - h_2)}
\end{array}\right),
$$
$$
{\mathbf U}_{\nu,2}
= \left(\begin{array}{ccc}
0 & 0 & 0 \\
0 & b_{22} & b_{23}a^{h_3 - h_2}\\
0 & b_{23}a^{h_3 - h_2} & b_{33} a^{2(h_3 - h_2)}
\end{array}\right)
$$
and note that
$$
\lambda_2(P( {\mathbf S}_{\nu,1}+ {\mathbf T}_{\nu,2})P^*) = \frac{\lambda_2(\bbE W_a(a2^j))}{a^{2h_2}}
$$
(c.f.\ expressions \eqref{e:W/a^(2hi0)} and \eqref{e:Uhat}). By adapting the argument leading to \eqref{e:infu*PUP*u}, by Weyl's inequality (see \eqref{e:Weyl}), and by a simple adaptation of the proof of \eqref{e:lambda2(PTP*)_bound},
$$
\varrho_{\nu,2,j}(x_*(\nu)) \leq \lambda_2(P( {\mathbf S}_{\nu,1}+ {\mathbf T}_{\nu,2})P^*)
$$
$$
\leq \lambda_3(P {\mathbf S}_{\nu,1}P^*)+ \lambda_2(P{\mathbf T}_{\nu,2}P^*) \leq \frac{C}{a^{h_2 - h_1}} + \frac{C'}{a^{h_3 - h_2}} +  \lambda_2(P{\mathbf U}_{\nu,2}P^*)
$$
$$
\leq \frac{C}{a^{\min_{1 \leq q_1 < q_2 \leq 3}(h_{q_2} - h_{q_1})}}  + b_{22}\langle p_{\cdot,2},{\mathbf w}(\nu)\rangle^2 + b_{33}x^2_{*}(\nu) + 2 b_{23}\langle p_{\cdot,2},{\mathbf w}(\nu)\rangle x_*(\nu)
$$
$$
= \frac{C}{a^{\min_{1 \leq q_1 < q_2 \leq 3}(h_{q_2} - h_{q_1})}} +  b_{22}\langle p_{\cdot,2},u_2(\nu)\rangle^2 + b_{33}x^2_{*}(\nu) + 2 b_{23}\langle p_{\cdot,2},u_2(\nu)\rangle x_*(\nu)
$$
$$
+ b_{22}\Big(
\langle p_{\cdot,2},{\mathbf w}(\nu)\rangle^2  - \langle p_{\cdot,2},u_2(\nu)\rangle^2 \Big)
+ 2 b_{23}\Big( \langle p_{\cdot,2},{\mathbf w}(\nu)\rangle  - \langle p_{\cdot,2},u_2(\nu)\rangle \Big)x_*(\nu)
$$
$$
\leq \frac{C}{a^{\min_{1 \leq q_1 < q_2 \leq 3}(h_{q_2} - h_{q_1})}} +  \varrho_{\nu,2,j}(x_*(\nu)),
$$
where the last inequality is a consequence of the bounds \eqref{e:|<p2,w(nu)>-<p2,u2(nu)>|_bound} and \eqref{e:|<p2,w(nu)>2-<p2,u2(nu)>2|}. Moreover, by \eqref{e:|<p2,u2(nu)>-<p2,u2>|_bound}
and \eqref{e:|x*(nu)-x*|_bound},
$$
\Big| \varrho_{\nu,2,j}(x_*(\nu)) - \xi_{2}(2^j) \Big| = \Big|\varrho_{\nu,2,j}(x_*(\nu)) - g_{2,j}(x_*)\Big|
$$
$$
= \Big| \Big\{b_{22}\langle p_{\cdot,2},u_2(\nu)\rangle^2 + b_{33}x^2_*(\nu) + 2 b_{23}\langle p_{\cdot,2},u_2(\nu)\rangle x_*(\nu) \Big\}
$$
$$
- \Big\{b_{22}\langle p_{\cdot,2},u_2\rangle^2 + b_{33}x^2_* + 2 b_{23}\langle p_{\cdot,2},u_2\rangle x_*\Big\}\Big| \leq \frac{C}{a^{\min_{1 \leq q_1 < q_2 \leq 3}(h_{q_2} - h_{q_1})}}.
$$
Consequently,
$$
\Big| \frac{\lambda_2(\bbE W_a(a2^j))}{a^{2h_2}} - \xi_{2}(2^j) \Big| \leq \Big| \frac{\lambda_2(\bbE W_a(a2^j))}{a^{2h_2}} - \varrho_{\nu,2,j}(x_*(\nu)) \Big|+ \Big| \varrho_{\nu,2,j}(x_*(\nu)) - \xi_{2}(2^j) \Big|
$$
$$
\leq \frac{C}{a^{\min_{1 \leq q_1 < q_2 \leq 3}(h_{q_2} - h_{q_1})}}.
$$
Hence, \eqref{e:|lambdaq(EW)-xiq(2^j)|_bound} holds for $q = 2$. $\Box$\\
\end{proof}

\noindent {\sc Proof of Corollary \ref{c:h^q_asymptotically_normal}}: We begin by showing $(i)$. In the argument for proving Theorem \ref{t:consistency}, replace $a^{H}$ with $(a2^j)^{H}$. We arrive at the double bound
\begin{equation}\label{e:lambdaq_double_bound_in_prob}
C_1 (a2^j)^{2\Re h_{q'}} \leq \lambda_q(W_a(a2^j)) \leq C_2 (a2^j)^{2\Re h_{q'}} (1 + o_P(\log^{2(n-1)}a), \quad q' = 1,\hdots,n',
\end{equation}
for constants $C_1$, $C_2 > 0$ that do not depending on $j$, where \eqref{e:lambdaq_double_bound_in_prob} holds with probability arbitrarily close to 1. Therefore,
$$
\log_2 C_1 + 2\Re h_{q'}(\log_2 a +  j)  \leq \log_2 \lambda_q(W_a(a2^j)) \leq \log_2 C_2 + 2\Re h_{q'}(\log_2 a ( 1 + o_{P}(1)) +  j) .
$$
Claim \eqref{e:h^q_consistent} is now a consequence of \eqref{e:sum_wj=0,sum_jwj=1} and Theorem \ref{t:consistency}.

Next, we show $(ii)$. For a fixed $q = 1,\hdots,n$, the left-hand side of \eqref{e:h^q_asymptotically_normal} can be recast as
$$
\sqrt{\frac{\nu}{a}} \sum^{j_2}_{j = j_1} \frac{w_j}{2} \Big(\log_2 \lambda_q(W_a(a2^j) - \log_2 \lambda_q(\bbE W_a(a 2^j))\Big)
$$
\begin{equation}\label{e:h^q-hq_three_terms}
+ \sqrt{\frac{\nu}{a}} \sum^{j_2}_{j=j_1} \frac{w_j}{2} \Big(\log_2 \lambda_q(\bbE W_a(a 2^j)) - \log_2 \xi_q(a(\nu)2^j)\Big)
+ \sqrt{\frac{\nu}{a}} \Big( \sum^{j_2}_{j = j_1} \frac{w_j}{2}  \log_2 \xi_q(a(\nu)2^j)  - h_q \Big).
\end{equation}
Note that by \eqref{e:xi-i0_scales} in Proposition \ref{p:convergence}, the function $\xi_q(\cdot)$ satisfies the scaling relation $\xi_q(a(\nu)2^j) = (a(\nu)2^{j})^{2h_q} \xi_q(1)$. Therefore, by property \eqref{e:sum_wj=0,sum_jwj=1}, the third term in the sum \eqref{e:h^q-hq_three_terms} is zero. In turn, by the mean value theorem and \eqref{e:|lambdaq(EW)-xiq(2^j)|_bound} in Lemma \ref{l:|lambdaq(EW)-xiq(2^j)|_bound}, the second term in the sum \eqref{e:h^q-hq_three_terms} is bounded by
$$
\sqrt{\frac{\nu}{a}} \sum^{j_2}_{j=j_1} \frac{|w_j |}{2} \frac{C}{a^{ \min_{1 \leq q_1 < q_2 \leq n} (h_{q_2} - h_{q_1})}}
\leq C'\sqrt{\frac{\nu}{a^{ 1 + 2 \min_{1 \leq q_1 < q_2 \leq n} (h_{q_2} - h_{q_1})}}} \rightarrow 0, \quad \nu \rightarrow \infty,
$$
where the limit is a consequence of condition \eqref{e:a(nu)/J->infty}. Therefore, we can rewrite \eqref{e:h^q_asymptotically_normal} as
$$
\sum^{j_2}_{j = j_1} \frac{2^{j/2 - 1}w_j}{ \log 2} \sqrt{K_{a,j}} \Big(\log \lambda_q(W_a(a2^j)) - \log \lambda_q(\bbE W_a(a2^j))\Big) + o(1),
$$
and the weak limit \eqref{e:h^q_asymptotically_normal} follows from Theorem \ref{t:asympt_normality_lambda2}. In the limiting variance in \eqref{e:h^q_asymptotically_normal}, the weight matrix $M \in M(n,mn,\bbR)$ is given by
\begin{equation}\label{e:weight_matrix_M}
M = \Big( \frac{2^{j_1/2}w_{j_1}}{ \log 2} I_n ; \hspace{1mm}\frac{2^{j_1+1/2}w_{j_1+1}}{ \log 2} I_n; \hspace{1mm} \hdots \hspace{1mm}; \hspace{1mm} \frac{2^{j_2/2}w_{j_2}}{ \log 2} I_n\Big),
\end{equation}
where $I_n \in M(n,\bbR)$ is an identity matrix and $m$ is as in \eqref{e:j1<...<j2_m=j2-j1+1}. $\Box$\\

\noindent {\sc Proof of Proposition \ref{p:h1=...=hn}} Fix $j \in \bbN$. For an OFBM under assumptions \eqref{e:h1=...=hn} and \eqref{e:time_revers}, the wavelet variance at octave $j$ is given by
$$
\bbE W_a(2^j) = C \int_{\bbR} |x|^{-2h+1} \frac{|\widehat{\psi}(2^j x)|^2}{x^2}dx \hspace{1mm}AA^*
$$
for some constant $C > 0$ (see Abry and Didier \cite{abry:didier:2017}, expression (3.2)). Hence, by condition \eqref{e:AA*_has pairwise distinct eigenvalues}, all eigenvalues of $\bbE W(2^j)$ are simple. For $q = 1,\hdots,n$, let
$$
f_q (B) = \log \lambda_q (P B P^*), \quad B \in {\mathcal H}_{\geq 0}(n,\bbR).
$$
Then, we can rewrite
$$
\log \lambda_{q}(W_a(a2^j)) - \log \lambda_{q}(\bbE W_a(a2^j)) = \log \lambda_{q}(W_a(2^j)) - \log \lambda_{q}(\bbE W_a(2^j))
$$
$$
= f_q (\widehat{B}_a(2^j)) - f_q( B(2^j)),
$$
where the derivative \eqref{e:dlambdal} is well-defined in some vicinity ${\mathcal O}$ of $B(2^j)$ in ${\mathcal H}_{> 0}(n,\bbR)$. Therefore, the weak limit \eqref{e:asympt_normality_lambda2} is a consequence of the Delta method (Taylor expansion). In addition, since
$$
\lambda_{q}(\bbE W_a(a2^j)) = (a2^j)^{2h}\lambda_{q}(\bbE W_a(1)),
$$
under \eqref{e:h1=...=hn}, the weak limit \eqref{e:h^q_asymptotically_normal} also holds. $\Box$\\

\subsection{Asymptotic theory for discrete time measurements}


Define the complex-valued random matrix
\begin{equation}\label{e:Btilde-nu}
\widetilde{B}_{\nu}(2^j)= P^{-1}\frac{1}{K_{a,j}}\sum^{K_{a,j}}_{k=1} \widetilde{D}_\nu (2^j, k)\widetilde{D}_\nu (2^j, k)^*  (P^*)^{-1},\quad P \in GL(n,\bbC).
\end{equation}
The following lemma can be proved by following the same steps of the proof of Lemma C.2 in Abry and Didier \cite{abry:didier:2017:supplementary}. In its proof, we make use of the condition that
$\frac{\nu}{a(\nu)^{1+ 2 \hspace{0.5mm}\Re h_1}}\rightarrow 0$, $\nu \rightarrow \infty$ (see \eqref{e:a(nu)/J->infty}).

\begin{lemma}\label{l:vecBtilde(2^j)_asympt}
Under the assumptions (OFBM1--2, 3$'$, 4), let $\widetilde{B}_{\nu}(2^j)$, $\widehat{B}_a(2^j)$ be as in \eqref{e:Btilde-nu} and \eqref{e:Bhat_a(2^j)_B(2^j)}. Then,
\begin{equation}\label{e:sqrt(Kaj)(Btilde-Bhat)}
\norm{ \sqrt{K_{a,j}}(\widetilde{B}_{\nu}(2^j) - \widehat{B}_a(2^j)) }_{L^1(P)} \rightarrow 0, \quad \nu \rightarrow \infty.
\end{equation}
\end{lemma}

Next, we show Theorem \ref{t:asympt_log_a(nu)_discrete}.\\

\noindent {\sc Proof of Theorem \ref{t:asympt_log_a(nu)_discrete}}: In regard to $(a)$, note that, by Lemma \ref{l:vecBtilde(2^j)_asympt},
\begin{equation}\label{e:Btilde->B}
\widetilde{B}_{\nu}(2^j) \stackrel{P}\rightarrow B(2^j).
\end{equation}
Therefore, \eqref{e:log_lambdaE/2_log_a(n)_discrete} and \eqref{e:log_lambdaE/2_log_a(n)_distinct Re_discrete} can be shown by the same argument for establishing Theorem \ref{t:consistency}.

To show $(c)$, rewrite the left-hand side of \eqref{e:asympt_normality_lambda2_discrete}  as
\begin{equation}\label{e:sqrt(K)(log_lambdaq(Wtilde)-log_lambdaq(W))}
\sqrt{K_{a,j}} \Big( \log \lambda_{q}(\widetilde{W}(a2^j)) - \log \lambda_{q}(W_a(a2^j)) \Big) +
\sqrt{K_{a,j}} \Big( \log \lambda_{q}(W_a(a2^j)) - \log \lambda_{q}(\bbE W_a(a2^j)) \Big).
\end{equation}
Define the event
$$
\widetilde{A} = \Big\{ \omega:  \widetilde{B}_{\nu}(2^j), \widehat{B}_{a}(2^j) \in {\mathcal O} \Big\}.
$$
In view of \eqref{e:Btilde->B}, $P(\widetilde{A}) \rightarrow 1$ as $\nu \rightarrow \infty$. Therefore, by replacing $\bbE W_a(a(\nu)2^j)$ with $W_a(a(\nu)2^j)$ and $W_a(a(\nu)2^j)$ with $\widetilde{W}(a(\nu)2^j)$, we can use the same argument leading to \eqref{e:Taylor} to arrive at
\begin{equation}\label{e:Taylor_discrete}
f_{\nu,q}(B) - f_{\nu,q}(\widehat{B}_a(2^j)) = \sum^{n}_{i_1 = 1}\sum^{n}_{i_2 = 1}\frac{\partial}{\partial b_{i_1,i_2}}f_{\nu,q}(\widetilde{B}) \hspace{1mm} \pi_{i_1,i_2}(B - \widehat{B}_a(2^j)).
\end{equation}
The expansion \eqref{e:Taylor_discrete} holds in the set $\widetilde{A}$ for any $B \in {\mathcal O}$ and for some matrix $\widetilde{B}$ lying in a segment connecting $B$ and $\widehat{B}_a(2^j)$ across ${\mathcal  H}_{>0}(n,\bbR)$. By \eqref{e:sqrt(Kaj)(Btilde-Bhat)},
\begin{equation}\label{e:sqrt(K)pi(Btilde-Bhat)->0_L1}
\sqrt{K_{a,j}} \hspace{1mm}  \pi_{i_1,i_2}(\widetilde{B}_{\nu}(2^j) - \widehat{B}_a(2^j)) \stackrel{L^1(P)}\longrightarrow 0, \quad i_1,i_2 = 1,\hdots,n.
\end{equation}
Moreover, by following the argument of the proof of Proposition \ref{p:convergence},
\begin{equation}\label{e:eigen/law_converges_discrete}
\frac{\lambda_{q}(\widetilde{W}(a2^j))}{a^{2h_{q}}}\stackrel{P}\rightarrow \xi_{q}(2^j),
\end{equation}
and for a sequence of eigenvectors $\{u_{q}(\nu)\}_{\nu \in \bbN}$ of $\widetilde{W}(a2^j)$,
\begin{equation}\label{e:angle*law_converges_discrete}
\Big(\langle p_{\cdot,q+1},u_{q}(\nu) \rangle a^{h_{q+1} - h_{q}}, \hdots, \langle p_{\cdot,n},u_{q}(\nu) \rangle a^{h_n - h_{q}} \Big) \stackrel{P}\rightarrow {\mathbf x}_{q,*}, \quad \nu \rightarrow \infty,
\end{equation}
where ${\mathbf x}_{q,*}$ is given by \eqref{e:x*_sol_determ}. By \eqref{e:Taylor_discrete}, \eqref{e:sqrt(K)pi(Btilde-Bhat)->0_L1}, \eqref{e:eigen/law_converges_discrete} and \eqref{e:angle*law_converges_discrete},
$$
\sqrt{K_{a,j}} \Big( \log \lambda_{q}(\widetilde{W}(a2^j)) - \log \lambda_{q}(W_a(a2^j)) \Big) = \sqrt{K_{a,j}} \hspace{1mm} \Big(f_{\nu,q}(\widetilde{B}_{\nu}(2^j)) - f_{\nu,q}(\widehat{B}_a(2^j)) \Big) \stackrel{L^1(P)}\rightarrow 0
$$
as $\nu \rightarrow \infty$. Hence, by \eqref{e:sqrt(K)(log_lambdaq(Wtilde)-log_lambdaq(W))} and Theorem \ref{t:asympt_normality_lambda2}, \eqref{e:asympt_normality_lambda2_discrete} holds. Moreover, by a similar argument, statement $(e)$ also holds.

As in the proof of Corollary \ref{c:h^q_asymptotically_normal}, statement $(b)$ is a consequence of the proof of statement $(a)$, and statement $(d)$ is a consequence of statement $(c)$. Moreover, in light of the proof of Proposition \ref{p:h1=...=hn}, under conditions \eqref{e:h1=...=hn} and \eqref{e:AA*_has pairwise distinct eigenvalues} the argument for showing $(c)$ still holds. Hence, so does statement $(f)$. $\Box$\\

\section{Auxiliary results}

\subsection{Theorem \ref{t:consistency}}

The following two basic lemmas are used the proof of Theorem \ref{t:consistency} and are stated without proof.
\begin{lemma}\label{l:bound_eigen_W}
Let $M \in {\mathcal H}_{\geq 0}(n,\bbC)$ and suppose its eigenvalues are ordered $0 \leq \lambda_1(M) \leq \hdots \leq \lambda_n(M)$. Then,
$$
\lambda_1(M) \hspace{0.5mm}v^* v \leq v^* M v \leq \lambda_n(M) \hspace{0.5mm}v^* v, \quad v \in \bbC^n.
$$
\end{lemma}

\begin{lemma}\label{l:S_1=<S_2}
Let $S_1, S_2 \in {\mathcal H}_{\geq 0}(n,\bbC)$. If $u^*S_1 u \leq u^*S_2 u$, $u \in S^{n-1}_{\bbC}$, then
$$
\lambda_{q}(S_1) \leq \lambda_{q}(S_2), \quad q = 1,\hdots,n.
$$
\end{lemma}

\subsection{Matrix calculus}

In this section, we retrieve some results from Magnus and Neudecker \cite{magnus:neudecker:2007} to produce a mean value theorem for scalar-valued functions with matrix arguments.

Let $\vecoper$ be the operator that piles up the columns of a matrix, namely,
$$
\vecoper(A) = \left(\begin{array}{c}
a_{\cdot,1}\\
\vdots\\
a_{\cdot,n}
\end{array}\right), \quad A \in M(m,n,\bbR).
$$
Define the function
$$
F: {\mathbf S} \rightarrow \bbR^{m \times p}, \quad {\mathbf S} \subseteq \bbR^{n \times q},
$$
differentiable at a point $S \in \textnormal{int}\hspace{0.5mm}{\mathbf S}$. We define the \textit{Jacobian matrix} of $F$ at the matrix $S$ by
$$
DF(S)  = D \hspace{0.5mm}\vecoper F(S).
$$
This is the $mp \times nq$ matrix whose $(i_1,i_2)$-th element is the partial derivative of the $i_1$-th component of $\vecoper F(X)$ with respect to the $i_2$-th element of $\vecoper X$, evaluated at the point $X = S$.

Let ${\mathcal T} \subseteq \bbR^{m \times p}$ be a set such that $F({\mathbf S}) \subseteq {\mathcal T}$, and let $G: {\mathcal T}\rightarrow \bbR^{r \times s}$ be a differentiable function at a point $T \in F(S) \in \textnormal{int}\hspace{0.5mm}{\mathcal T}$. Further define the composite function
$$
H : {\mathbf S} \rightarrow \bbR^{r \times s}, \quad {\mathbf S} \ni X \mapsto H(X) = G [F(X)].
$$
Then, by the chain rule (Magnus and Neudecker \cite{magnus:neudecker:2007}, p.\ 108, Theorem 12), $H$ is differentiable at $S$ and its Jacobian at the point $S$ is given by
\begin{equation}\label{e:chain_rule}
DH(S) = DG[F(S)]\hspace{0.5mm} D F(S).
\end{equation}
In particular, when $p = m$ and $n = q = 1 = r = s$, $DF(S)$ and $DG[F(S)]$ are, respectively, $m^2 \times 1$ and $1 \times m^2$ matrices and we can rewrite \eqref{e:chain_rule} as
\begin{equation}\label{e:chain_rule_p=m,n=q=1=r=s}
DH(S) = D \hspace{0.5mm} \vecoper G[F(x)] \hspace{1mm} D \hspace{0.5mm} \vecoper F(x)\Big|_{x = s},
\end{equation}
where $S =: s \in \bbR$. For the sake of illustration, in the case where $m = 2$, we can write $\Big( F(s)_{ii'}\Big)_{i,i'=1,2}$ and
$$
D \hspace{0.5mm} \vecoper F(s) = \Big( F'(s)_{11}, F'(s)_{21}, F'(s)_{12}, F'(s)_{22} \Big)^{*},
$$
$$
D \hspace{0.5mm} \vecoper G(T) = \Big( \frac{\partial}{\partial t_{11}}G(T), \frac{\partial}{\partial t_{21}}G(T), \frac{\partial}{\partial t_{12}}G(T), \frac{\partial}{\partial t_{22}}G(T)\Big).
$$
Hence,
$$
DH(S) = \sum^{2}_{i_1=1}\sum^{2}_{i_2=1}\frac{\partial}{\partial t_{i_1,i_2}}G[F(s)]F'(s)_{i_1,i_2}.
$$
The following mean value relation is a straightforward consequence of the chain rule \eqref{e:chain_rule_p=m,n=q=1=r=s}.
\begin{proposition}\label{p:mean_value_theorem}
Let $G: {\mathcal T} \rightarrow \bbR $ be a differentiable function, where ${\mathcal T} \subseteq \bbR^{m \times m}$ is a connected, open set in the matrix norm topology. Let $T_0, T_1 \in {\mathcal T}$. Then, there is a matrix $\Theta = \{\theta_{i_1,i_2}\}_{i_1,i_2 = 1,\hdots,m}$ in the segment $\{T \in {\mathcal T}: s T_0 + (1-s) T_1, s \in [0,1]\} \subseteq {\mathcal T}$ such that
\begin{equation}\label{e:mean_value_theorem}
G(T_1) - G(T_0) = \sum^{m}_{i_1=1}\sum^{m}_{i_2=1}\frac{\partial}{\partial t_{i_1,i_2}}G[\Theta]\hspace{1mm}d_{i_1,i_2},
\end{equation}
where $\Delta := T_1 - T_0 = \{\Delta_{i_1,i_2}\}_{i_1,i_2=1,\hdots,m}$.
\end{proposition}
\begin{proof}
Define the path
\begin{equation}\label{e:path_in_matrix_space}
M(m,\bbR) \ni F(s) = T_0 + s \Delta, \quad s \in [0,1].
\end{equation}
Also define the real-valued, composite function $H(s) = G[F(s)]$. Then, by the mean value theorem and \eqref{e:chain_rule_p=m,n=q=1=r=s}, there is $\varsigma \in [0,1]$ such that
\begin{equation}\label{e:H(1)-H(0)=H'(varsigma)}
H(1) - H(0) = H'(\varsigma) = \sum^{m}_{i_1=1}\sum^{m}_{i_2=1}\frac{\partial}{\partial t_{i_1,i_2}}G[F(\varsigma)]\hspace{1mm}\Delta_{i_1,i_2}.
\end{equation}
This shows \eqref{e:mean_value_theorem}. $\Box$\\
\end{proof}

\subsection{On Jordan canonical forms}

For $h \in \bbC$, a Jordan block of size $n_{h}$ is given by
\begin{equation} \label{e:Jordan_block}
J_{h} = \left( \begin{array}{ccccc}
h & 0 & 0 & \ldots & 0 \\
1   & h & 0 & \ldots & 0 \\
0   &    1   & h & \ldots & 0 \\
\vdots & \vdots & \vdots & \ddots & \vdots\\
0   &    0   & \ldots & 1 & h \\
\end{array} \right).
\end{equation}
Then, for $z > 0$,
\begin{equation} \label{e:z^Jlambda}
z^{J_{h}} = \left( \begin{array}{ccccc}
z^{h}      &    0    &    0    &  \ldots & 0\\
(\log z)z^{h} & z^{h} &    0     & \ldots & 0\\
\frac{(\log z)^{2}}{2!} z^{h} & (\log z)z^{h}   &    z^{h} & \ddots & 0 \\
\vdots   &  \vdots  &  \ddots   & \ddots  & 0 \\
\frac{(\log z)^{n_{h}-1}}{(n_{h}-1)!} z^{h} &
\frac{(\log z)^{n_{h}-2}}{(n_{h}-2)!} z^{h} &
\ldots &  (\log z)z^{h} & z^{h}\\
\end{array} \right).
\end{equation}

\bibliography{mlrd}

\small

\bigskip

\noindent \begin{tabular}{lr}
Patrice Abry & \hspace{6cm} Gustavo Didier\\
Physics Lab & Mathematics Department\\
CNRS and \'{E}cole Normale Sup\'{e}rieure de Lyon & Tulane University  \\
46 all\'{e}e d'Italie & 6823 St.\ Charles Avenue\\
F-69364, Lyon cedex 7, France & New Orleans, LA 70118, USA\\
{\it patrice.abry@ens-lyon.fr} & {\it gdidier@tulane.edu}\\
\end{tabular}\\

\smallskip

\end{document}